\documentclass[12pt,reqno]{amsart}

\usepackage{amsmath,amsfonts,amsthm,amssymb,amsxtra}
\usepackage{xcolor}
\usepackage{bbm} % to get \1
\usepackage{hyperref} % creates hyperlinks to the references 
\usepackage{setspace}
\newtheorem{theorem}{Theorem}[section]
\newtheorem{proposition}[theorem]{Proposition}
\newtheorem{lemma}[theorem]{Lemma}
\newtheorem{corollary}[theorem]{Corollary}

\theoremstyle{definition}

\newtheorem{assumption}[theorem]{Assumption}

\theoremstyle{remark}

\newtheorem{remark}[theorem]{Remark}

\numberwithin{equation}{section}

\newcommand{\1}{\mathbbm{1}}

\renewcommand{\epsilon}{\varepsilon}

\newcommand{\N}{\mathbb{N}}

\renewcommand{\phi}{\varphi}
\newcommand{\R}{\mathbb{R}}

\newcommand{\Sph}{\mathbb{S}}

\newcommand{\Xb}{{\bf X}}

\DeclareMathOperator{\per}{Per}

\DeclareMathOperator{\supp}{supp}

\DeclareMathOperator{\Tr}{Tr}

\begin{document}

\title[Pleijel Theorem for  sub-Laplacians]{On Courant and Pleijel theorems\\ for sub-Riemannian Laplacians}

\author{Rupert L. Frank}
\address[Rupert L. Frank]{Mathe\-matisches Institut, Ludwig-Maximilians Universit\"at M\"unchen, The\-resienstr.~39, 80333 M\"unchen, Germany, and Munich Center for Quantum Science and Technology, Schel\-ling\-str.~4, 80799 M\"unchen, Germany, and Mathematics 253-37, Caltech, Pasa\-de\-na, CA 91125, USA}
\email{r.frank@lmu.de}

\author{Bernard Helffer}
\address[Bernard Helffer]{Laboratoire de Math\'ematiques Jean Leray, CNRS, Nantes Universit\'e, F44000. Nantes, France}
\email{Bernard.Helffer@univ-nantes.fr}

\thanks{\copyright\, 2024 by the authors. This paper may be reproduced, in its entirety, for non-commercial purposes.\\
Partial support through US National Science Foundation grant DMS-1954995 (R.L.F.), as well as through the German Research Foundation through EXC-2111-390814868 and TRR 352-Project-ID 470903074 (R.L.F.) is acknowledged.}

\begin{abstract}
	We are interested in the number of nodal domains of eigenfunctions of sub-Laplacians on sub-Riemannian manifolds. Specifically, we investigate the validity of Pleijel's theorem, which states that, as soon as the dimension is strictly larger  than $1$, the number of nodal domains of an eigenfunction corresponding to the $k$-th eigenvalue is strictly (and uniformly, in a certain sense) smaller than $k$ for large~$k$.
	
	In the first part of this paper we reduce this question from the case of general sub-Riemannian manifolds to that of nilpotent groups.
	
	  In the second part, we analyze in detail the case where the nilpotent group is a Heisenberg group times a Euclidean space. Along the way we improve known bounds on the optimal constants in the Faber--Krahn and isoperimetric inequalities on these groups.
\end{abstract}

\maketitle

\tableofcontents

\section{Introduction}\label{s0}

\subsection{Short summary}~\\
The study of nodal domains is a classical topic in spectral geometry. The founding work of R.~Courant \cite{CH} from 1923 showed that an eigenfunction of the Laplacian corresponding to the $k$-eigenvalue has at most $k$ nodal domains. In the past century many contributions have led to a better understanding of nodal domains. An important step for this was a theorem of A.~Pleijel \cite{Pl} from 1956, which shows that asymptotically  the number of nodal sets of an eigenfunction becomes significatively smaller than the bound given by Courant when the dimension is strictly larger than~$1$. This initially involved the Dirichlet condition but was more recently extended to other boundary conditions by I.~Polterovich \cite{Po} and  C.~L\'ena \cite{Le}; see also \cite{DFV}.

On the other hand, it is natural to consider the same question  for other operators and this leads naturally to the consideration of Dirichlet sub-Laplacians, initially called H\"ormander operators \cite{Ho}, which share with the Laplacian the property of hypoellipticity.  An important step in 
their analysis was done by L.~P.~Rothschild and E.~M.~Stein \cite{RS}, who proved their maximal hypoellipticity. They proceeded by comparison with operators on nilpotent Lie groups, a technique that is also known as the nilpotentization procedure. Using this approximation, G.~M\'etivier \cite{Me} proved a beautiful Weyl formula for the asymptotic behavior of the counting function of eigenvalues, provided this approximation can be done ``uniformly".  Motivated by a recent paper \cite{EL} by S.~Eswarathasan and C.~Letrouit on Courant's nodal theorem and many discussions on related problems with C. Letrouit, our aim in this paper is to try to extend Pleijel's theorem to the sub-Riemannian context. 

Our paper is divided into two, rather independent, parts. In the first part, we show how the validity of Pleijel's theorem in the sub-Riemannian case can be reduced to the specific analysis of sub-Laplacian on nilpotent groups. In the second part, we analyze the validity of Pleijel's theorem in open sets of specific groups related to the Heisenberg group. This leads us to the question of finding new explicit and close to optimal bounds on the constants for Sobolev inequalities, Faber--Krahn inequalities and isoperimetric inequalities.

%%%%%%%%%%%%%%%%%%%%

\medskip

\subsection{The H\"ormander operator}~\\
We consider in an open set $\Omega \subset \mathbb R^n$ the Dirichlet realization of a sub-Laplacian (also called H\"ormander's operator)
$$
-\Delta_{\bf  X}^\Omega :=\sum_{j=1}^p X_j^\star X_j \,,
$$
where $X_1,\ldots,X_p$ are real $C^\infty$ vector fields on $\overline\Omega$ satisfying the so-called H\"ormander condition \cite{Ho}, which reads:

\begin{assumption}\label{CHOr} 
For some $ r\geq 1$ the $ X_j$, $j=1,\ldots,p$, and their brackets up to order $ r$ generate at each point $ x\in \overline{\Omega}$ the tangent space $T_x \Omega$. 
\end{assumption}

Note that the terminology ``sub-Riemannian Laplacian" or in short ``sub-Laplacian" is posterior to the work of L. H\"ormander and corresponds to the development of sub-Riemannian geometry at the beginning of the nineties; see for example \cite{BeRi,Je}.

More generally, we are given a connected $C^\infty$ Riemannian manifold  $M$ of dimension $n$  with a given measure $\mu$ (with a $C^\infty$-density with respect to the Lebesgue measure in a local system of coordinates) and a system ${\bf X}= (X_1,\cdots,X_p)$ of real $C^\infty$ vector fields on $M$ satisfying Assumption~\ref{CHOr}, where $p\leq n$. We consider the operator 
$$
-\Delta_{\bf X}^{M,\mu} :=\sum_{j=1}^p X_j^\star X_j \,,
$$
where $X_j^\star$ is the formal adjoint obtained by using the $L^2$ scalar product with respect to the given measure $\mu$. In local coordinates $X_j^\star =- X_j +c_j$ for a function $c_j$. In the case when $M$ has a boundary, we always impose a Dirichlet condition.

The sub-Laplacian $-\Delta_{\bf X}^{M,\mu}$ is known to be hypoelliptic \cite{Ho} (meaning that if $\omega\subset M$ is open and $u$ is a distribution in $M$ with $-\Delta_{\bf X}^{M,\mu} u$ smooth in $\omega$, then $u$ is smooth in $\omega$). By Rothschild--Stein \cite{RS} it is maximally hypoelliptic, that is, it satisfies
$$
\| X_k X_\ell \,  u \|_{L^2(M,\mu)} \leq C \left( \|\Delta_{\bf X}^{M,\mu}  u \|_{L^2(M,\mu)} + \|u\|_{L^2(M,\mu)} \right), \qquad \forall k,\ell\,, \forall u \in C_c^\infty(M)\,.
$$
The latter result is proved through a technique of nilpotentization, which will also be important for us. Moreover, if the boundary is $C^\infty$ and noncharacteristic for ${\bf X}$ (i.e.~at each point of the boundary there exists a vector field $ X_j$ that is transverse to the boundary at the given point), then we have $C^\infty$-regularity up to the boundary. We emphasize that we will not need this latter condition for our results; see Section~\ref{srest}.
  
The  operator $-\Delta_{\bf X}^{M,\mu}$ has compact resolvent for instance when $M$ is closed, and we can ask all the questions about its discrete spectrum that have been solved along the years for the Dirichlet realization of the Euclidean Laplacian on a bounded open set. These include: 
 \begin{itemize}
 \item Simplicity of the ground state or, more generally, its multiplicity.
 \item Local structure of the nodal sets, density of the nodal sets, ...
 \item Courant's theorem: comparison between the minimal labelling $ k$ of an eigenvalue $\lambda_k$  and the number $\nu_k$ of the nodal domains of the eigenfunction in the eigenspace corresponding to $\lambda_k$.
 \item Pleijel's theorem.
 \end{itemize}
We focus in this paper on the two last items. They will be described in more detail in the next two subsections.

To clarify our terminology: nodal domains of a real (eigen)function $u$ are the connected components of $\{ x\in M:\ u(x) \neq 0\}$; nodal sets are their boundaries.

\medskip

\subsection{Courant's Theorem}~\\
As is well known, Courant's theorem in the case of the Dirichlet Laplacian on a bounded open subset of $\mathbb R^n$ states that an eigenfunction associated with the $k$-th eigenvalue has at most $k$ nodal domains:
$$
\nu_k \leq k\,.
$$
If one looks at the standard proof of Courant's theorem, this inequality mainly appears as a consequence of a restriction statement (the restriction of an eigenfunction to its nodal domain is the ground state of the Dirichlet realization of the Laplacian in this domain), the minimax characterization of the eigenvalues,  and the Unique Continuation theorem (UCT). Hence the question is to determine under which conditions these three results extend to sub-Riemannian Laplacians. 

Concerning the restriction statement, having rather limited information about the  nodal sets, we successfully adapt to the sub-Riemannian case a proof proposed in \cite{MP}, which permits to avoid regularity assumptions on $\partial \Omega$. The variational characterization then holds. The UCT was proved by K.\ Watanabe \cite{Wa} in the $C^\infty$ category in dimension 2, but H.\ Bahouri \cite{Ba} gave a discouraging  counterexample to UCT with two vector fields in $ \mathbb R^3$. 
 Here
 we are fortunate to know that J.\ M.\ Bony  proved at the end of the sixties \cite{Bo} that UCT holds when the vector fields are analytic\footnote{Note nevertheless that the sub-Laplacians with analytic vector fields are not in general hypoelliptic analytic. While the sub-Laplacian on the Heisenberg group $\mathbb H_n$ is known to be hypoelliptic analytic (its fundamental solution is explicitly known and analytic outside the origin
  in the exponential coordinates; see \cite{Fo0}), the sub-Laplacian on $\mathbb H_n\times \mathbb R^k$ with $k\geq 1$ is known to be non-hypoelliptic analytic as a direct consequence of a result by  Baouendi--Goulaouic \cite{BaGo}. We refer to \cite{Me2} for a characterization of the nilpotent groups of rank $2$ whose associated sub-Laplacian is hypoelliptic analytic and to \cite{Hel} for other counterexamples  when the rank of nilpotency is strictly larger than $2$. Here we have limited our references to the case of sub-Laplacians on nilpotent groups. Starting in the seventies there have been a lot of contributions on the subject, in particular around an ``evoluting" Tr\`eves conjecture.}.
  Hence Courant's theorem holds in the analytic category \cite{EL}. At the end of the next section, we will extend statements given in \cite{EL} to the case when the boundary is not necessarily non-characteristic.

\medskip

\subsection{Pleijel's Theorem}~\\
In the same spirit, one can hope for an asymptotic control of $ \nu_k/k$ for large $k$ that improves over Courant's bound when $n\geq 2$. In the case of the Dirichlet Laplacian on a bounded open subset of $\mathbb R^n$ with $n\geq 2$, Pleijel's theorem \cite{Pl} says that there is an constant, independent of the open set and denoted by $\gamma(\mathbb R^n)$, such that
\begin{equation}
	\label{eq:pleijelintro}
	\limsup_{k\rightarrow +\infty} \frac{\nu_k}{k} \leq \gamma(\R^n)
\end{equation}
and, importantly,
\begin{equation}
	\label{eq:pleijelintro2}
	\gamma(\R^n)<1 
	\qquad\text{for}\ n\geq 2 \,.
\end{equation}
Because of this latter inequality, Pleijel's bound \eqref{eq:pleijelintro} provides an asymptotic improvement over Courant's theorem.

Later in this paper we  will use the expression ``Pleijel's theorem holds'' (for a given operator) to  mean the assertion that $\limsup_{k\to\infty} \nu_k/k<1$.

The proof of Pleijel's theorem is a nice combination of two ingredients. The first one is Weyl's formula, which describes the asymptotic behavior of the eigenvalue counting function $N(\lambda,-\Delta^\Omega)$ of the Dirichlet Laplacian $-\Delta^\Omega$:
$$
N(\lambda, -\Delta^\Omega) \sim \mathcal W(\mathbb R^n) \,  |\Omega|\,   \lambda^{\frac n2} 
\qquad\text{as}\ \lambda\to\infty \,.
$$
Here $\mathcal W(\R^n)$ is a certain explicit constant, depending only on $n$. The second ingredient is the Faber--Krahn inequality, which gives the following lower bound on the lowest eigenvalue $\lambda_1(-\Delta^\Omega)$ of the operator $-\Delta^\Omega$:
$$
\lambda_1(-\Delta^{\Omega}) \geq C^{\rm FK}(\R^n) \, |\Omega|^{-\frac 2n} \,.
$$
The constant $C^{\rm FK}(\R^n)$ is equal to the first Dirichlet eigenvalue on a ball of unit volume;  for an expression of this constant in terms of Bessel functions see \eqref{eq:isofkexplicit} below. Pleijel's proof combines these two ingredients and leads to the inequality \eqref{eq:pleijelintro} with
\begin{equation}\label{PlConst}
\gamma (\R^n)= (C^{\rm FK}(\R^n))^{-\frac n2} \, \mathcal W(\mathbb R^n)^{-1} \,.
\end{equation}
Using the explicit expressions for the Weyl constant $\mathcal W(\mathbb R^n)$ and the Faber--Krahn constant $C^{\rm FK}(\R^n)$ one can establish \eqref{eq:pleijelintro2}; in this regard we refer to \cite[Part II, Lemma~9]{BerMe}.

\medskip

\subsection{A Pleijel bound in the sub-Riemannian case}~\\
Our goal in the present paper is to generalize Pleijel's theorem to the case of sub-Laplacians, and so we are naturally led to the sub-Riemannian analogues of the two ingredients of its proof, namely Weyl's formula and the Faber--Krahn inequality.

\medskip

We can be optimistic on the side of Weyl's formula. Since the pioneering work of G. M\'etivier \cite{Me} we are rich in results on the asymptotic distribution of eigenvalues, at least if we add to Assumption \ref{CHOr} a certain equiregularity condition (or M\'etivier's condition),  which permits to approximate the vector fields $X_j$ at each point $x$  by the generators of a nilpotent Lie algebra $\mathcal G_x$. As we already mentioned, we work in the setting of a connected $C^\infty$ manifold  $M$ (with or without boundary) of dimension $n$  with a given measure $\mu$ (with a $C^\infty$-density with respect to the Lebesgue measure in a local system of coordinates) and a system of $p$ $C^\infty$ vector fields $X_1,\ldots,X_p$, where $p\leq n$. We assume the vector fields satisfy the following:
 
\begin{assumption}\label{equireg}
 For each $j\leq r$, the dimension of the space spanned by the commutators of length $\leq j$ at each point is constant.
\end{assumption}

In the language of modern sub-Riemannian  geometry,  this is called an \emph{equiregular distribution}.  The simplest example occurs with $r=2$, $p=2$ and $n=3$ with the Heisenberg group $\mathbb H$ and more generally  in contact geometry (see for example \cite{ArRi}).

We denote by $\mathcal{D}_j(x)$ the span at $x\in M$ of all vector fields obtained as brackets of length $\leq j$ of the $X_k$'s. We set $n_j :=\text{dim}(\mathcal{D}_j(x))$, which by Assumption \ref{equireg} above does not depend on the point $x\in M$. We can then introduce, setting $n_0:=0$,
\begin{equation}\label{eq:defQ}
	Q:= \sum_{j=1}^r j \, (n_j-n_{j-1})\,.
\end{equation}
This plays the role of an ``effective dimension''.

Under Assumptions  \ref{CHOr} and \ref{equireg}, G.~M\'etivier shows (using, in particular, the techniques of \cite{RS,Rot}) that there is a constant $c(M,{\bf X})$ such that
\begin{equation}\label{MetForm1}
N(\lambda,-\Delta_{\bf X}^{M,\mu}):= \#  \{j:\ \lambda_j (-\Delta_{\bf X}^{M,\mu}) \leq \lambda\} \sim c (M,{\bf X}) \, \lambda^{\frac Q2}
\qquad\text{as}\ \lambda\to\infty \,.
\end{equation}
We will come back later to the structure of $c (M,{\bf X})$ and to its computation in particular cases. Note that  in the case $r=2$ related results are obtained in \cite{MS1,MS2,Me,Mo}, and M\'etivier's theorem (together with many other results) has been recently revisited in the light of sub-Riemannian geometry  in \cite{CHT0,CHT,CHT1}.

This concludes our discussion of the first ingredient in Pleijel's proof, namely an analogue of Weyl's formula.

\medskip

Concerning the second ingredient, namely an analogue of the Faber--Krahn inequality, our knowledge is rather poor. This is the question about minimizing the first Dirichlet eigenvalue among open sets of given measure. In the case of the  Heisenberg group, one can think of a result by P.\ Pansu \cite{Pa} concerning the isoperimetric inequality. C.\ L\'ena's approach \cite{Le} for treating the Neumann problem for the Laplacian could be helpful (see \cite{HS}) if the set in $M$ where the system $\mathbf{X}$ is not elliptic is ``small" in some sense, but our equiregularity assumption excludes this case.

In the first part of this paper we will follow another way by revisiting the nilpotentization procedure. This permits us to deduce Faber--Krahn inequalities  for sub-Laplacians from Faber--Krahn  inequalities for sub-Laplacians on nilpotent groups.

More precisely, under the above two assumptions we will prove that the Faber--Krahn inequality holds on subsets of small measure with a constant that is arbitrarily close to an integral over the constants of the Faber--Krahn inequalities on the nilpotent approximations $\mathcal G_x$; see Proposition \ref{faberkrahnbm}. This result is in the spirit of a result of B\'erard--Meyer \cite[Part II, Lemma 16]{BerMe}, who have shown that on a Riemannian manifold the Faber--Krahn inequality holds on subsets of small measure with a constant that is arbitrarily close to the constant in the Faber--Krahn inequality on $\R^n$. We emphasize, however, that in the case of a Riemannian manifold the model space is the same at each point, namely $\R^n$, while in the sub-Riemannian setting the approximating model spaces $\mathcal G_x$ may vary with $x$. Our techniques are quite different from those employed in~\cite{BerMe} and its generalizations, e.g., in \cite{DFV}.\\

Combining our result about Faber--Krahn inequalities with M\'etivier's Weyl-type formula, we obtain a sufficient condition for the validity of a Pleijel-type bound; see Theorem \ref{t:manifold}, which is the main result of the first part of this paper. The upper bound on $\limsup_{k\to\infty} \nu_k/k$ is of the form
\begin{equation}
	\label{eq:mainintro}
	\left( \int_M (c_x^{\rm FK})^{-\frac Q2} \,d\mu(x) \right) \left( \int_M c_x^{\rm Weyl}\,d\mu(x) \right)^{-1},
\end{equation}
where $c_x^{\rm FK}$ is a certain local Faber--Krahn constant, defined in terms of the nilpotentization of $-\Delta_{\bf X}^{M,\mu}$ at $x\in M$, and $c_x^{\rm Weyl}$ is a certain local Weyl constant, defined in terms of the same nilpotentization\footnote{This strengthening of our original result \cite{FrHe1,FrHe2} is due to Y.~Colin de Verdi\`ere, who kindly allowed us to include his argument.}. The precise definitions will be given below. 

The role of the Borel measure $D\mapsto \int_D c_x^{\rm Weyl}\,d\mu(x)$ on $M$ is emphasized in \cite{CHT1}, where it is called the \emph{Weyl measure}. Similarly, here we introduce what may be called the \emph{Faber--Krahn measure} $D\mapsto \int_D (c_x^{\rm FK})^{-\frac Q2} \,d\mu(x)$.

It is interesting to compare \eqref{eq:mainintro} with the Pleijel formula \eqref{PlConst}, to which it reduces in the case of open subsets of $\R^n$. More generally, in the Riemannian case (where $p=n$ and where $\mu$ is the Riemannian volume measure) the expression \eqref{eq:mainintro} reduces to \eqref{PlConst} and we recover the result of B\'erard and Meyer \cite{BerMe}. However, our result is already new in this case when $\mu$ is different from the Riemannian volume measure. In the general sub-Riemannian case, the integration with respect to the measure $\mu$ takes into account that the model spaces $\mathcal G_x$ may vary with the point $x\in M$. In this respect it is also interesting to note that \eqref{eq:mainintro} depends on $M$ and the vector fields $X_1,\ldots,X_p$, but does \emph{not} depend on the measure $\mu$. Indeed, both integrals in \eqref{eq:mainintro} do not depend on $\mu$; see Remark \ref{rem:indepmeas}.

According to \eqref{eq:mainintro}, a sufficient condition for the validity of Pleijel's theorem is the following bound on the ``local Pleijel constants'':
$$
\big(c_x^{\rm FK} \big)^{-\frac Q2} \, \big( c_x^{\rm Weyl} \big)^{-1} < 1
\qquad\text{for all}\ x\in M \,;
$$
see Corollary \ref{maincor}. We emphasize that the latter condition involves the corresponding Faber--Krahn constants for Dirichlet realizations of sub-Laplacians in open set of nilpotent groups.

\medskip

This provides a motivation for the second part of this paper, which is devoted to the validity of a Pleijel-type bound for the nilpotent groups $\mathbb H_n \times \mathbb R^k$, where $\mathbb H_n $ is the Heisenberg group of homogeneous dimension $2n+2$ and where $k\in\N_0$. While we have not been able to establish a Pleijel-type theorem in the most important case $(n,k)=(1,0)$, we have succeeded in proving it if one admits the celebrated conjecture of Pansu concerning the isoperimetric constant on $\mathbb H_1$; see Proposition \ref{pansupleijel}. We also have positive results under the assumption that the homogeneous dimension $Q=2n+2+k$ of $\mathbb H_n\times\R^k$ is sufficiently large. Indeed, the validity of Pleijel's theorem remains open for only four pairs $(n,k)$; see Theorem \ref{maincomp}, which is the main result of the second part of this paper.
 
\medskip
 
\subsection{Organization}~\\
 This paper is divided into two main parts.
 
 The first part begins in Section \ref{srest} with a somewhat technical result concerning the restriction of a Sobolev function to a nodal domain, which plays an important role in the arguments of Courant and Pleijel and their generalizations. The main result of the first part is Theorem \ref{t:manifold} in Section \ref{sec:mainsubriem}, which gives a sufficient condition for the validity of Pleijel's theorem via nilpotentization. In the preceding Section \ref{s1} we discuss the setting of this theorem and give the proof in the following Section \ref{sec:mainsubriemproof}. The theorem is illustrated by an example in Section \ref{AppA}.
 
The second part of this paper deals with the case of the Heisenberg group $\mathbb H_n$ and, more generally, with $\mathbb H_n\times\R^k$ with $k\in\N_0$. The main results of that part are summarized in Section \ref{sec:mainhnrk}; see, in particular, Theorem \ref{maincomp} and Proposition \ref{pansupleijel}. The proofs of these results rely on an explicit form of the Weyl asymptotics, treated in Section \ref{sec:weylhnrk}, and bounds on the Faber--Krahn constants. For the latter, we proceed via two different techniques that are spread out over Sections \ref{sec:fksob} (continued in Section \ref{s10}) and \ref{sec:isoperimetric}.

In two appendices we discuss an assumption appearing in the main result of the first part (Appendix \ref{AppA2}) and review an approach to Weyl asymptotics (Appendix \ref{AppC}).

\medskip

\subsection{Acknowledgements}~\\
Thanks to D.\ Mangoubi for inviting  the second author to propose open questions in an  Oberwolfach workshop in August 2023 and for previous discussions. Great thanks also to C.~Letrouit  for stimulating discussions about his joint paper with S. Eswarathasan and sub-Riemannian geometry when we start to attack these questions. The comments of Y.~Colin de Verdi\`ere, to whom we are most grateful, on a preprint version \cite{FrHe1} of this paper helped to significantly strenghten our results. Thanks to J.~Viola and F.~Nicoleau for helping us with the use of Mathematica or Wolframalpha. Thanks to V.~Colin for his help in topology. We are also grateful to L.~Hillairet and P.~Pansu for their help in the understanding of sub-Riemannian results and to L.~Capogna and N.~Garofalo for exchanges about the isoperimetric problem on the Heisenberg group. Thanks to E.~Milman for discussions on Faber--Krahn constants.

\newpage
%%%%%%%%%%%%%%%%%%%%%%%%
%%%%%%%%%%%%%%%%%%%%%%%%

\part{Courant's and Pleijel's bound in the sub-Riemannian case}

\section{The restriction of a Sobolev function to a nodal domain}\label{srest}

\subsection{Presentation}~\\
In this section we will show that the restriction of a continuous Sobolev function to a nodal domain satisfies the boundary values zero in the sense of Sobolev functions. This holds under remarkably weak assumptions on the vector fields used to define the Sobolev spaces. The precise formulation can be found in Theorem~\ref{restriction} below.

That this restriction property requires a proof seems to have been overlooked by R.~Courant in his original proof of the Courant nodal theorem, but  this was proved in the case when $M$ is a compact manifold in \cite{BerMe}. In the case with boundary, the proof goes in the same way when the Dirichlet boundary problem is regular (see \cite{An} or more recently  \cite{LMP}). Without this assumption, it was proved in a paper by M\"uller-Pfeiffer \cite{MP} that seems to be little known in the spectral theory community. An alternative proof, due to D.\ Bucur, is presented in the book \cite{LMP}.  Here we show that the proof in \cite{MP} can be generalized to the sub-Riemannian setting. Notice that Bucur's proof could also have been adapted.

\subsection{Setting}\label{ss2.2}~\\
Let $M$ be a $C^1$-manifold (without boundary) of dimension $n$ and let $\mu$ be a $C^1$ nonnegative Borel measure on $M$. (By $C^1$-measure we mean that in each chart the measure is absolutely continuous with a positive $C^1$-density.) Let $X_1,\ldots,X_p$ be $C^1$-vector fields on $M$.

Notice that in this section our results are established under weaker assumptions than in the rest of this paper. In particular, we emphasize that we do not make an assumption on the span of the vector fields. In fact, even the trivial case $p=0$, where there are no vector fields at all, is formally included in our analysis.

We define first-order differential operators $X_1^\star,\ldots,X_p^\star$ by
$$
\int_M \psi X_j \phi\,d\mu = \int_M \phi X_j^\star \psi\,d\mu 
\qquad\text{for all}\ \phi,\psi\in C^1_c(M) \,.
$$
We define $S^1(M)$ to be the set of all functions $u\in L^2(M)$ for which there are $f_1,\ldots f_p\in L^2(M)$ such that for $j=1,\ldots,p$ one has
$$
\int_M u X_j^\star\phi\,d\mu = \int_M \phi f_j\,d\mu
\qquad\text{for all}\ \phi\in C^1_c(M) \,.
$$
The $f_j$'s are necessarily unique and we denote them by $f_j=:X_j u$, thus extending the usual notation in case $u\in C^1_c(M)$. We set
$$
q[u] := \sum_{j=1}^p \|X_j u\|_{L^2}^2 
\qquad\text{for all}\ u\in S^1(M)\,.
$$
One easily sees that $S^1(M)$ is a vector space that is complete with respect to the norm $\sqrt{q[u] + \|u\|_{L^2}^2}$. We denote by $S^1_0(M)$ the closure of $C^1_c(M)$ in $S^1(M)$. When $\Omega$ is an open subset of $M$, then $\Omega$ itself is a manifold and therefore the spaces $S^1(\Omega)$ and $S^1_0(\Omega)$ are defined. We also introduce the space $S^1_{\rm loc}(M)$ of functions such that $u|_\Omega \in S^1(\Omega)$ for any open $\Omega\subset M$ with $\overline\Omega$ compact.

We emphasize that the spaces $S^1(M)$ and $S^1_0(M)$ depend on $X_1,\ldots,X_p$ and, if $M$ is not compact, on $\mu$, even if this is not reflected in the notation. Note that if $p=0$, then $S^1(M)=S^1_0(M)=L^2(M)$.
 If $M=\mathbb R^n$ (or an open set in $\mathbb R^n$), $p=n$, and $X_i=\partial_{x_i}$, then we recover the classical Sobolev spaces.\\
 We begin by recording truncation properties, which will play an important role in our arguments.

\begin{lemma}\label{truncation}
	If $u,v\in S^1_{\rm loc}(M)$, then $\max\{u,v\}, \min\{u,v\}\in S^1_{\rm loc}(M)$ and
	$$
	X_j \max\{u,v\}(x) = \begin{cases}
		X_j u(x) & \text{a.e.\ in}\ \{ x\in M:\ u(x)\geq v(x)\} \,,\\ 
		X_jv(x) & \text{a.e.\ in}\ \{ x\in M:\ u(x)\leq v(x)\} \,,
	\end{cases}
	$$
	and 
	$$
	X_j \min\{u,v\}(x) = \begin{cases}
		X_j v(x) & \text{a.e.\ in}\ \{ x \in M:\ u(x)\geq v(x) \} \,,\\ 
		X_ju(x) & \text{a.e.\ in}\ \{ x \in M:\ u(x)\leq v(x) \} \,.
	\end{cases}
	$$
	In particular, if $u,v\in S^1(M)$, then $\max\{u,v\}, \min\{u,v\}\in S^1(M)$.
\end{lemma}

\begin{proof}
	This lemma is well known (see, e.g., \cite[Lemma 3.5]{GaNh} when the underlying manifold is $\R^k$). For the sake of completeness we give an outline of the main steps of the proof.
	
 First, one shows that if $\eta\in C^1(\R)$ with $\eta'$ bounded, then $u\in S^1_{\rm loc}(M)$ implies $\eta(u)\in S^1_{\rm loc}(M)$ with $X_j\eta(u) = \eta'(u)X_j u$. (Here one can argue as in \cite[Theorem 1.18]{HeKiMa}.) Next, one applies this result to $\eta(t)=\sqrt{t^2+\epsilon^2}$ and deduces, after passing to the limit $\epsilon\to 0$, that $u\in S^1_{\rm loc}(M)$ implies $|u|\in S^1_{\rm loc}(M)$ with 
	$$
	X_j|u|(x) = 
	\begin{cases}
		X_j u(x) & \text{a.e.\ in}\ \{ x\in M:\ u(x)\geq 0\} \,,\\
		- X_ju(x) & \text{a.e.\ in}\ \{ x\in M:\ u(x)\leq 0 \} \,.
	\end{cases}
	$$
	Note that this shows, in particular, that $X_j u(x)=0$ a.e. in $\{x\in M:\ u(x)=0\}$. (Here one can argue similarly as in \cite[Lemma 1.19]{HeKiMa}.) Since
	\begin{align*}
		\max\{u,v\} & = \tfrac12(u(x)+v(x) + |u(x)-v(x)|) \,, \\
		\min\{u,v\} & = \tfrac12(u(x)+v(x) - |u(x)-v(x)|) \,,
	\end{align*}
	this implies the assertion of the lemma.
\end{proof}

%%%%%%%%%%%%%%%%%%%%%%%%

\medskip

\subsection{The restriction theorem}~\\
The following theorem is the main result of this section.

\begin{theorem}\label{restriction}
	Let $\Omega\subset M$ be open and let $u\in S^1_0(\Omega)\cap C(\Omega)$. If $\omega$ is a connected component of $\{ x\in\Omega:\ u(x)\neq 0\}$, then $u|_\omega \in S^1_0(\omega)$.
\end{theorem}

The proof below of this theorem is essentially taken from \cite{MP}. It is simpler than Bucur's proof presented in \cite{LMP} and, in particular, avoids the notion of capacity. Similar to that proof, it relies on truncation properties of Sobolev functions and on the following simple lemma.

\begin{lemma}\label{s10easy}
	Let $\Omega\subset M$ be open and let $u\in S^1(\Omega)$. Assume that $u$ vanishes outside of a compact set of $M$ and vanishes on $\partial\Omega$ in the sense that for any $y\in\partial\Omega$ and any $\epsilon>0$ there is a neighborhood $U$ of $y$ in $M$ such that $|u|<\epsilon$ a.e.\ in $U\cap\Omega$. Then $u\in S^1_0(\Omega)$.
\end{lemma}

\begin{proof}[Proof of Lemma \ref{s10easy}]
	This is modeled after \cite[Lemma 1.26]{HeKiMa}; see also \cite[Lemma 1]{HeNoSw}. We first note that we may assume that $\overline\Omega$ is compact. Otherwise, we consider an open set $\widetilde \Omega$ with compact closure such that $u$ vanishes almost everywhere in $\Omega\setminus\widetilde \Omega$ and apply the result on $\Omega\cap\widetilde \Omega$. Further, by considering $u_+$ and $u_-$ separately (using Lemma \ref{truncation}), we may assume that $u\geq 0$. Consider $u_\epsilon:=(u-\epsilon)_+$. We shall show that $u_\epsilon\in S^1_0(\Omega)$. Since $u_\epsilon\to u$ in $S^1(\Omega)$, this will imply the assertion.

	We claim that the support of $u_\epsilon$ is a compact subset of $\Omega$. Indeed, any $y\in\partial\Omega$ has an open neighborhood $U_y$ in $M$ with $u<\epsilon$ in $U_y\cap\Omega$. By compactness of $\partial\Omega$, a finite union of these $U_y$ cover $\partial\Omega$ and $u_\epsilon$ vanishes almost everywhere in the intersection of this finite union with $\Omega$.
	
	By the Meyers--Serrin type result in \cite[Theorem 1.13]{GaNh} we see that $u_\epsilon$ can be approximated in $S^1(\Omega)$ by functions from $C^1(\Omega)\cap S^1(\Omega)$. The result in that reference is stated for open subsets of $\R^n$, but since it is a local result, the result remains valid in our situation by localizing via a partition of unity to coordinate neighborhoods. The proof in \cite{GaNh} proceeds by convolution with a compactly supported function. The fact that the support of $u_\epsilon$ is a compact subset of $\Omega$ implies that the approximating functions belong to $C^1_c(\Omega)$. This proves that, indeed, $u_\epsilon\in S^1_0(\Omega)$.
\end{proof}

\begin{proof}[Proof of Theorem \ref{restriction}]
	For the sake of concreteness let us assume that $u> 0$ in  $\omega$, the argument in the opposite case being similar.
	
	\medskip
	
	\emph{Step 1.} Let
	$$
	z:= \begin{cases} u & \text{in}\ \omega \,,\\ 0 & \text{in}\ \Omega\setminus\omega \,. \end{cases}
	$$
	We claim that $z\in S^1_{\rm loc}(\Omega)$.
	
	Let $\psi\in C^1_c(\Omega)$. It is easy to see that $z\in C(\Omega)$ and consequently also $\psi z\in C(\Omega)$. Moreover, $\psi z=0$ on $(\Omega\setminus\omega)\cup(\Omega\setminus\supp\psi)$. Thus $\psi z|_\omega = \psi u|_\omega$ is in $S^1(\omega)\cap C(\omega)$ and extends continuously to $\partial\omega$, where it vanishes. By Lemma \ref{s10easy} this implies that $\psi z|_\omega \in S^1_0(\omega)$. This implies that $\psi z\in S^1_0(\Omega)$. (Indeed, since $\psi z|_\omega\in S^1_0(\omega)$, it can be approximated in $S^1(\omega)$ by functions in $C^1_c(\omega)$. Extending these functions by zero to $\Omega$ gives functions in $C^1_c(\Omega)$ and, since $\psi z$ vanishes in $\Omega\setminus\omega$, these functions approximate $\psi z$ in $S^1(\Omega)$, so $\psi z\in S^1_0(\Omega)$.) Since $\psi\in C^1_c(\Omega)$ is arbitrary, we deduce $z\in S^1_{\rm loc}(\Omega)$, as claimed.
	
	\medskip
	
	\emph{Step 2.} Let $(\phi_j)\subset C^1_c(\Omega)$ such that $\phi_j\to u$ in $S^1(\Omega)$ (such functions exist since $u\in S^1_0(\Omega)$) and set
	$$
	v_j := \min\{ z, (\phi_j)_+ \}|_\omega \,.
	$$
	We claim that $(v_j)$ is a bounded sequence in $S^1_0(\omega)$ and that it converges to $u|_\omega$ in $L^2(\omega)$. 
	
	To prove this, we consider $w_j:=\min\{ z, (\phi_j)_+ \}$. By Step 1 and truncation properties of Sobolev spaces (Lemma \ref{truncation}) we deduce that $w_j\in S^1_{\rm loc}(\Omega)$. Moreover, these truncation properties also imply that
	\begin{align*}
		\sum_{i=1}^p \int_\omega |X_i v_j|^2\,d\mu & = \sum_{i=1}^p \int_\omega |X_i w_j|^2\,d\mu \\
		& = \sum_{i=1}^p \int_{\{u \leq \phi_j \}} |X_i u|^2\,d\mu + \sum_{i=1}^p \int_{\{ u> \phi_j >0 \}} |X_i \phi_j|^2\,d\mu \\
		& \leq \sum_{i=1}^p \int_\Omega (|X_i u|^2 + |X_i \phi_j|^2)\,d\mu \,.
	\end{align*}
	Since $u\in S^1(\Omega)$ and since $(\phi_j)$ converges in $S^1(\Omega)$, the right side is uniformly bounded.
	
	Concerning the $L^2$-norm, we find
	$$
	\int_\omega (v_j - u)^2\,d\mu = \int_{\omega\cap\{ (\phi_j)_+< u\}} (u - (\phi_j)_+ )^2\,d\mu
	\leq \int_\Omega (u-\phi_j)^2\,d\mu \to 0 \,. 
	$$
	
	We have shown, in particular, that $v_j\in S^1(\omega)$. It remains to prove that $v_j\in S^1_0(\omega)$. We note that $w_j$ is continuous in $\Omega$. Its restriction $v_j$ to $\omega$ extends continuously to $\partial\omega$ and vanishes there. (The vanishing on $\partial\omega\cap\Omega$ comes from the continuity of $u$. The vanishing on $\partial\omega\cap\partial\Omega$ comes from the compact support property of $\phi_j$.) Thus, again by Lemma \ref{s10easy}, we have $w_j|_\omega\in S^1_0(\omega)$.
	
	\medskip
	
	\emph{Step 3.} We can now finish the proof. Since $(v_j)$ is a bounded sequence in $S^1_0(\omega)$, after passing to a subsequence we may assume that it converges weakly in $S^1_0(\omega)$ to some $v\in S^1_0(\omega)$. In particular, it converges weakly in $L^2(\omega)$ to $v$. Meanwhile, by Step~ 2, it converges strongly in $L^2(\omega)$ to $u|_\omega$, so the uniqueness of the weak limit in $L^2(\omega)$ implies that $v=u|_\omega$ a.e. In particular, $u|_\omega\in S^1_0(\omega)$, as claimed.	
\end{proof}

%%%%%%%%%%%%%%%%%%%%%%%%%%%%%
%%%%%%%%%%%%%%%%%%%%%%%%%%%%%

\subsection{Application to Courant's theorem}\label{ss2.7}~\\
Using Theorem \ref{restriction} we can recover the results in \cite{EL} concerning the Dirichlet realization $-\Delta^{\Omega,\mu}_\Xb$ in an open set $\Omega$ of $M$, but without having to assume that $\partial \Omega$ is noncharacteristic for the system of vector fields $X_j$ (see Assumption 1.5 in \cite{EL} or Assumption ($H\Gamma$) in \cite{De}). We record this as follows.

\begin{theorem}\label{courant}
	For any $k\in \mathbb N$, any eigenfunction of $-\Delta_\Xb^{\Omega,\mu}$ with eigenvalue $\lambda_k$ has at most 
	$k+ {\rm mult}(\lambda_k)-1$ nodal domains, where ${\rm mult}(\lambda_k)$ denotes the multiplicity of $\lambda_k$.\\
	If, moreover, one of the two following assumptions holds
	\begin{itemize} 
		\item $n=2$\,,
		\item $M$, $\mu$ and $(X_1,\dots, X_p)$ are real-analytic\,,
	\end{itemize}
	then we get an upper bound by $k$.
\end{theorem}

The first part of the statement relies on a remark of D.\ Mangoubi that permits to avoid the use of the Unique Continuation Theorem (see \cite[Exercise 4.1.15]{LMP}). For the second part of the theorem, \cite{EL} refers in the first case to \cite{Wa}. For the second case corresponding to the standard statement of  Courant's nodal Theorem, the proof is based on a result of J.\ M.~Bony \cite{Bo}. For the proof of both parts of the theorem we apply our Theorem \ref{restriction}. The required continuity of eigenfunctions in $\Omega$ follows from \cite{Ho}.

\bigskip

Let us further discuss the assumption of being noncharacteristic in \cite{EL}. A characteristic point of $\partial \Omega$ relative to $\mathcal D_1$ is a point $x$ for which all the elements of  $\mathcal D_1(x)$ belong to the tangent space $T_x\partial \Omega$. The assumption that $\partial\Omega$ is noncharacteristic for $\mathcal D_1$ guarantees that eigenfunctions belong to $C^\infty(\overline{\Omega})$; see \cite{De}, \cite{EL} and the book \cite{DGN} (particularly its Chapter 3). This regularity allowed \cite{EL} to prove Theorem \ref{courant} (under their noncharacteristic assumptions). This extra assumption seems to be crucial for the regularity of eigenfunctions at the boundary\footnote{Notice that the question of the analyticity at the boundary (in case the sub-Laplacians are hypoelliptic analytic in $\Omega$) seems open.}. Remarkably, however, it is not necessary for the validity of Theorem \ref{restriction}.

The condition that $\partial\Omega$ is noncharacteristic for $\mathcal D_1$ may appear rather strict. Examples where this condition holds in the case of the Heisenberg group are given in \cite{DGN,MM}. An example in $\mathbb H$  is given (see \cite[Example 3.4]{MM} with $k=1$) by the domain
  $$
  (\sqrt{x^2+y^2} -2)^2 + 16 \, t^2  < 1\,.
  $$
  Note that for this example $\partial \Omega$ is homeomorphic to  $\mathbb T^2$. 
 
 Meanwhile, by topological considerations the condition that $\partial\Omega$ is noncharacteristic for $\mathcal D_1$ is never satisfied if $\partial \Omega$ is homeomorphic to $\mathbb S^2$.\footnote{Many thanks to V.~Colin for this remark. 
 	If $\Omega \subset \mathbb R^3$ and if we have an equiregular distribution ($n_1=p=2$) and $r=2$, then
 	the transversality condition is never satisfied if $\partial \Omega$ is homeomorphic to $S^2$.
 	The reason is that the non-characteristic condition implies the existence of a continuous (with respect to $x$)  unique straight line in $T_x \partial \Omega$ (this is the intersection of $\mathcal D_1(x)$ with $T_x \partial \Omega$) and this is impossible when $\partial \Omega$ is homeomorphic to the sphere by the Poincar\'e--Hopf theorem \cite{Miln} (the Euler characteristic is not $0$).}
 	 
   Finally it was shown in \cite{De} that the measure in the boundary of the characteristic points is zero. Other connected results are obtained in the case of the Heisenberg group by D.~Jerison \cite{Jer}.

%%%%%%%%%%%%%%%%%%%%%%%%%%%%%%%%%%%%
%%%%%%%%%%%%%%%%%%%%%%%%%%%%%%%%%%%%

\section{The nilpotent approximation}\label{s1}

\subsection{Nilpotentization of vector fields and measures}~\\
Throughout this section we follow the presentation of Rothschild \cite{Rot}, which is based on assumptions and definitions given earlier by Goodman \cite{Go}, Folland--Stein \cite{FS}, Folland \cite{Fo}, M\'etivier \cite{Me}, Rothschild--Stein \cite{RS}. Since this period in the seventies, a huge literature has been devoted to the topic of sub-Riemannian geometry, for which we refer to the appendix in \cite{CHT} and references therein. We attempt to combine the two formalisms in this section as well as in Appendices \ref{AppA2} and \ref{AppC}.

\medskip

We consider the situation presented in the introduction and suppose that Assumptions \ref{CHOr} and \ref{equireg} are satisfied. We recall that $\mathcal D_1(x)$ was defined after Assumption \ref{equireg}, and also that $n$ denotes the dimension of $M$ and $n_1$ the (constant) dimension of $\mathcal D_1(x)$. 

Clearly, we have $p\geq n_1$, but this inequality may be strict (meaning that the vectors $X_1(x),\ldots,X_p(x)$ are not linearly independent at some and then, by equiregularity, any $x\in M$). In order to deal with this situation, we will apply Lemma \ref{locindep}, which says that for every point in $M$ there is a open neighborhood $W\subset M$ and vector fields $\tilde X_1,\ldots \tilde X_{n_1}$ defined in $W$ such that
\begin{equation}
	\label{eq:loclinindep}
	{\rm span}\{ \tilde X_1(x),\ldots \tilde X_{n_1}(x) \} = \mathcal D_1(x)
	\qquad\text{for all}\ x\in W
\end{equation}
and
\begin{equation}
	\label{eq:loclinindeplapl}
	-\Delta_{\bf X}^{M,\mu} f = \sum_{j=1}^{n_1} \tilde X_j^\star \tilde X_j f
	\qquad\text{for all} \ f\in C^2(W) \,.
\end{equation}
The vector fields $\tilde X_1,\ldots,\tilde X_{n_1}$ again satisfy Assumptions \ref{CHOr} and \ref{equireg} (with the same $r$ and $n_j$).

We will first discuss the nilpotent approximation on an open set $W\subset M$ where \eqref{eq:loclinindep} and \eqref{eq:loclinindeplapl} are satisfied. Later we will argue that this gives a nilpotent approximation on all of $M$. Of course, in the special case $p=n_1$ we can immediately take $W=M$, which simplifies the argument.

It is known that under our equiregularity assumption, for any $x\in W$ there is an open neighborhood $U\subset W$ of $x$ and vector fields $Y_1,\ldots,Y_n$ defined in $U$ such that for any $x'\in U$ we have
$$
\text{Span}(Y_1(x'),\ldots,Y_{n_j}(x'))=\mathcal{D}_j(x')
\qquad\text{for all}\ j=1,\ldots, r
$$
and
\begin{equation}\label{eq:assadd}
Y_i(x') =\tilde X_i(x') \qquad\text{for all}\ i=1,\ldots,n_1 \,.
\end{equation}
A family of vector fields satisfying the first assumption is called said to be adapted to the flag at $x'$.

Given an adapted flag $(Y_1,\ldots,Y_n)$ at $x\in W$ satisfying \eqref{eq:assadd} with $x'=x$, we can define canonical privileged coordinates of the first kind\footnote{Other choices are possible but we only need that some privileged coordinates exist and will only consider this one, which is actually the one introduced by G. M\'etivier \cite{Me} in the proof of his Theorem~3.1.} at $x$ by the mapping $\theta_x$ given by
 \begin{equation} \label{eq:pc}
 \theta_x (y) := u = (u_{i}) \qquad\mbox{ if } y = \exp (\sum_{i=1}^n u_{i}Y_i) \cdot x\,,
 \end{equation} where $\exp$ denotes the exponential map defined in some small neighborhood of $x$.
 Thus we identify a neighborhood of $x\in M$ via $\theta_x$ with a neighborhood of $0$ in $\mathbb R^n$. It has been shown by G.\ M\'etivier (see below)  that everything depends smoothly on $x$. In particular, $\theta_x$ is also $C^\infty$ with respect to $x$.

We denote by $Y_{i,x}$ the image of $Y_i$ by $\theta_x$, 
which is simply $Y_i$ written in the local canonical coordinates around $x$. Thus $Y_{i,x}$ is a vector field defined in an open neighborhood of $0$ in $\R^n$.

On $\mathbb R^n$, with coordinates $u=(u_{i})$, we introduce the family of dilations given by
\begin{equation}
	\label{eq:dilation}
	\delta_t (u_{i}) = (t^{w_i} u_{i})\,,
\end{equation}
where positive integers $w_1,\ldots,w_n$ are defined as follows: for any $i\in\{1,\ldots,n\}$ there is a unique $j\in\{1,\ldots,n\}$ such that $n_{j-1}+1 \leq i\leq  n_j$, and we set $w_i=j$. We note that the homogeneous dimension $Q$, defined in \eqref{eq:defQ}, satisfies
\begin{equation}\label{eq:defQa}
Q =\sum_{i=1}^n w_i\,.
\end{equation}
Via the family of dilations we have a natural definition of homogeneous functions of degree $s$ on $\mathbb R^n\setminus\{0\}$, a definition of ``homogeneous norm" (of degree one) and corresponding notions of vanishing function to order $p$.
A differential operator of the form
$f(u) \frac{\partial}{\partial u_{i}}$ is of  order $w_i -s$ if $f$ is homogeneous of degree $s$. When $f$ is defined in a pointed neighborhood of $0$ and can be expanded into a sum of homogeneous terms of increasing  order, we will say that $f$ is of order $\leq s$ if the term of lowest order is homogeneous of degree $s$. 

G. M\'etivier \cite[Theorem 3.1]{Me} proves the following theorem (in addition to the regularity of $\theta_x$ already mentioned above).

\begin{theorem}\label{lemmaMet}
	For any $x$,  $\tilde X_{j,x}$ is of order $\leq 1$. Furthermore, 
	\begin{itemize}
		\item We have
		$$
		\tilde X_{j,x} = \widehat X_{j,x} + R_{j,x}\,,
		$$
		where $\widehat X_{j,x}$ is homogeneous of order $1$ and $R_{j,x}$ is of order $\leq 0$.
		\item 
		The  $\widehat X_{j,x}$, $j=1,\ldots,n_1$, generate a nilpotent Lie algebra $\mathcal G_x$ of dimension $n$ and rank $r$. 
		\item 
		The mapping $x \mapsto \widehat X_{j,x}$ is smooth.
	\end{itemize}
\end{theorem}

By  the nilpotent approximation, we can associate with  each point $x\in W$ a nilpotent group $G_x$ (identified with the algebra $\mathcal G_x$ in the $u$-coordinates) and a corresponding sub-Laplacian 
$$
\widehat \Delta_x= \sum_{j=1}^{n_1} \widehat X_{j,x}^{\,2}
$$
in $\mathcal U_2(\mathcal G_x)$ (the elements in the enveloping algebra  $\mathcal U(\mathcal G_x)$ that are homogeneous of degree $2$). The H\"ormander condition (Assumption \ref{CHOr}) for the vector fields $\tilde X_1,\ldots,\tilde X_{n_1}$ implies that for every $x\in W$ the vector fields $\widehat X_{1,x},\ldots,\widehat X_{n_1,x}$ satisfy the corresponding H\"ormander condition on $\R^n$.

\medskip

At this point it is important to notice that for any $x\in W$ the operator $-\widehat\Delta_x$ depends only on the vector fields $X_1,\ldots,X_p$ (and on the measure $\mu$) and \emph{not} on the choice of the auxiliary vector fields $\tilde X_1,\ldots,\tilde X_{n_1}$ satisfying \eqref{eq:loclinindep} and \eqref{eq:loclinindeplapl}. This claim is justified in \cite[Subsection A.5.4]{CHT1}. The choice of other auxiliary vector fields corresponds to a different choice of privileged coordinates. In passing we note that this argument also shows that instead of the above canonical privileged coordinates of the first kind one can also choose other so-called privileged coordinates leading to the same operator and measure; see also the example in Section \ref{AppA}.

\medskip

Next, we introduce the notion of nilpotentized measure. That is, given the measure $\mu$ on $M$, for any point $x\in W$ we define a measure $\widehat \mu_x$ on $\R^n$. We refer to  \cite[Appendix A.5.6]{CHT1} or to our Appendix \ref{AppC} for a definition in the formalism of sub-Riemannian geometry  and explain here ``by hand" how it can be constructed for our specific choice of privileged coordinates.
On $\mathbb R^n$ we have the Lebesgue measure 
$$
du= \prod_{i=1}^n du_{i}\,,
$$
and in these local coordinates  the measure $\mu$ is of the form
$$
d\mu = a(x,u) \,du\,,
$$
where $(x,u)\mapsto a(x,u)$ is $C^\infty$ in both variables $x$ and $u$.
In a small neighborhood of $0$,  the nilpotentized measure at $x$  can be defined by
\begin{equation}\label{eq:nilpmeas}
	d\widehat \mu_x:= a(x,0) \,du\,.
\end{equation}
Note that for $u$ small and locally in $x$ we have a good control on $a(x,u)/a(x,0)$ and its inverse.

It is important to note that for every $x\in M$, the nilpotentized measure $\widehat\mu_x$ is invariant with respect to the group operation on $G_x$. As a consequence, the formal adjoint $(\widehat X_{j,x})^\star$ of $\widehat X_{j,x}$ with respect to the scalar product in $L^2(G_x,\widehat\mu_x)$ is equal to $-\widehat X_{j,x}$; see, e.g., \cite[Remark A.5]{CHT1}. Therefore the sub-Laplacian
$$
-\widehat \Delta_x = \sum_{j=1}^{n_1} (\widehat X_{j,x})^\star \widehat X_{j,x}
$$ 
arises from the construction in the introduction when we replace $M$ by $G_x=\R^n$, $\mu$ by $\widehat\mu_x$ and $X_1,\ldots,X_p$ by $\widehat X_{1,x},\ldots,\widehat X_{n_1,x}$. In particular, $-\widehat \Delta_x$ is selfadjoint in $L^2(G_x,\widehat\mu_x)$.

\medskip

Again, the nilpotentized measure depends only on $\mu$ and the vector fields $X_1,\ldots, X_p$ and \emph{not} on the choice of the auxiliary vector fields $\tilde X_1,\ldots,\tilde X_{n_1}$ satisfying \eqref{eq:loclinindep} and \eqref{eq:loclinindeplapl}. This claim is implicit in \cite[Subsection A.5.6]{CHT1}, where the nilpotentized measure is defined through arbitrary privileged coordinates. The choice of other auxiliary vector fields corresponds to a different choice of privileged coordinates.

\medskip

This concludes our presentation of the nilpotent approximation on $W$. In order to obtain a nilpotent approximation on all of $M$ we apply Lemma \ref{locindep} to cover $M$ by open sets $W$ on which \eqref{eq:loclinindep} and \eqref{eq:loclinindeplapl} are satisfied. The fact that $-\widehat\Delta_x$ and $\widehat\mu_x$ depend only on $\mu$ and $X_1,\ldots,X_p$ implies that when $x$ belongs to two different sets $W$, then the corresponding nilpotentized sub-Laplacians and nilpotentized measures coincide. Therefore the nilpotent approximation is well defined on $M$.

\medskip

\subsection{The local Weyl constant}~\\
We now turn our attention to the asymptotic distribution of the eigenvalues of the operator $-\Delta_{\Xb}^{M,\mu}$, which is given by M\'etivier's Weyl formula \cite{Me} that we have already mentioned in \eqref{MetForm1}. We will describe the constant that appears in this asymptotic formula.

For each fixed $x\in M$, we consider the selfadjoint operator $-\widehat \Delta_x$ in $L^2(G_x,\widehat\mu_x)$. Its spectral projections $\1(-\widehat\Delta_x<\lambda)$, $\lambda>0$, are integral operators in $G_x$, that is,
$$
(\1(-\widehat\Delta_x<\lambda) f)(u) = \int_{G_x} \1(-\widehat\Delta_x<\lambda)(u,v) f(v)\,d\widehat\mu_x(v)
$$
for all $u\in G_x$ and $f\in L^2(G_x,\widehat\mu_x)$ with a certain integral kernel $\1(-\widehat\Delta_x<\lambda)(u,v)$, the \emph{spectral function}. Since the operator $-\widehat\Delta_x$ is invariant under the group operation in $G_x$, the integral kernel satisfies
$$
\1(-\widehat\Delta_x<\lambda)(u,u) =  \1(-\widehat\Delta_x<\lambda)(0,0)
\qquad\text{for all}\ u\in G_x \,.
$$
Moreover, since $-\widehat\Delta_x$ is homogeneous of degree $-2$ under dilations in $G_x$, we deduce that
\begin{equation}
	\label{eq:specfcndiag1}
	\1(-\widehat\Delta_x<\lambda)(0,0) = c_x^{\rm Weyl} \ \lambda^\frac Q2
	\qquad\text{for all}\ \lambda>0
\end{equation}
with
\begin{equation}
	\label{eq:specfcndiag2}
	c_x^{\rm Weyl} := \1(-\widehat\Delta_x<1)(0,0) \,.
\end{equation}
Explicit formulas for $c_x^{\rm Weyl}$ can be obtained in certain special cases, for instance, in the case where $G_x$ is a Heisenberg group; see Section \ref{sec:weylhnrk}. In general it is know that $c^{\rm Weyl}_x$ is positive for every $x\in M$ and that $x\mapsto c^{\rm Weyl}_x$ is continuous

Using this definition we can state a more precise version of \eqref{MetForm1}.

\begin{theorem}\label{weylmet}
	The spectral counting function of the selfadjoint realization of $-\Delta^{M,\mu}_{\bf X}$ in $L^2(M,\mu)$ satisfies, as $\lambda \rightarrow +\infty$,
\begin{equation}\label{eq:met1}
N(\lambda,-\Delta_{\Xb}^{M,\mu}):= \#  \{j:\ \lambda_j (-\Delta_{\Xb}^{M,\mu}) \leq \lambda\} \sim \Big(\int_M c_x^{\rm Weyl} \, d\mu(x) \Big)\, \lambda^{\frac Q2}\,.
\end{equation}
\end{theorem}

As we already mentioned, this result is due to M\'etivier \cite{Me}. In the special case where $r=2$, there were important contributions on the subject starting from the end of the seventies  \cite{MS1,MS2,Mo}. For recent developments related to Theorem \ref{weylmet} we refer to \cite{CHT0,CHT,CHT1}; see also our Appendix \ref{AppC} for its relation to \cite{CHT1}.

\begin{remark}\label{rem:weylsubsets}
	Theorem \ref{weylmet} remains valid when we consider the Dirichlet realization of the operator $-\Delta_{\Xb}^{M,\mu}$ in an open set $\Omega\subset M$. In this case the integral on the right side of \eqref{eq:met1} is restricted to $\Omega$.
\end{remark}

\begin{remark}
	The constant $c_x^{\rm Weyl}$ itself appears through a Weyl-type formula. Indeed, let $x\in M$ and let $\Omega\subset G_x$ be open with $\widehat\mu_x(\Omega)<\infty$. Then, if $-\widehat\Delta_x|_\Omega$ is the Dirichlet realization of $-\widehat \Delta_x$ on $\Omega$, one has
	$$
	N(\lambda,-\widehat\Delta_x|_\Omega) \sim c_x^{\rm Weyl} \ \widehat\mu_x(\Omega) \ \lambda^{\frac Q2}
	\qquad\text{as}\ \lambda\to\infty \,.
	$$
	This follows from the previous remark since the nilpotentization of the sub-Laplacian on a nilpotent group is the sub-Laplacian itself.
\end{remark}

\begin{remark}\label{rem:weylindepmeas}
	The integral $\int_M c_x^{\rm Weyl} \, d\mu(x)$ depends on $M$ and on the vector fields $X_1,\ldots, X_p$, but is independent of the measure $\mu$. This can probably be extracted from \cite{Me} and is made explicit in \cite[Subsection 2.1]{CHT1}.
\end{remark}

\medskip

\subsection{The local Faber--Krahn constant}~\\
For any $x\in M$ we have a Faber--Krahn inequality, that is, for every $x\in M$ there is a constant $c>0$ such that
\begin{equation} \label{eq:FKx1}
	\langle - \widehat \Delta_x v,v\rangle_{L^2(G_x,\widehat\mu_x)} \geq c \ \widehat\mu_x(\Omega)^{-\frac 2Q} \| v\|_{L^2(G_x,\widehat\mu_x)}^2\,,
	\qquad \forall \Omega\subset G_x \ \text{open}\,,\,\forall v\in C_c^\infty (\Omega)\,.
\end{equation}
We recall that $Q$ denotes the homogeneous dimension of $G_x$; see \eqref{eq:defQ}. By our assumption of equiregularity, $Q$ is independent of $x\in M$.

By definition, $c_x^{\rm FK}$ is the largest constant such that \eqref{eq:FKx1} holds. We will prove momentarily the positivity of this constant, even uniformly in $x$.

\begin{remark}\label{rem:fkindepmeas}
	Having in mind the proof of Pleijel's theorem, it is important to write the above estimates using the appropriately normalized Lebesgue measure $\widehat \mu_x$ on $G_x$. Note that the Faber--Krahn constant $c_x^{\rm FK}$ depends both on $G_x$ and on a specific normalization constant determined by the measure $\mu$; see \cite[Appendix A.5.6]{CHT1} and our Appendix \ref{AppC}. When $\tilde\mu$ is a second measure on $M$ satisfying the same properties as $\mu$, then $d\tilde\mu = h\, d\mu$ for a smooth, positive function $h$ on $M$. Then, for any $x\in M$,
	$$
	d\widehat{\tilde\mu}_x = h(x) \, d\widehat\mu_x \,,
	$$
	which shows that
	$$
	\tilde c^{\rm FK}_x = h(x)^\frac 2Q \, c^{\rm FK}_x \,.
	$$
	As an important consequence, we see that the integral
	$$
	\int_M \left( c^{\rm FK}_x \right)^{-\frac Q2}d\mu(x) 
	$$
	depends on $M$ and $X_1,\ldots, X_p$, but is independent of the measure $\mu$.\footnote{We are grateful to Y.~Colin de Verdi\`ere for pointing this out to us.}
\end{remark}

Further information about the Faber--Krahn constant is contained in the following lemma, whose proof we defer to the Subsection \ref{sec:fk}. 

\begin{lemma}\label{faberkrahn}
	We have
	\begin{equation}\label{eq:fk1}
		\inf_{x\in M}  c_x^{\rm FK} >0\,.
	\end{equation}
	Moreover, the function $M\ni x\mapsto c_x^{\rm FK}$ is uniformly H\"older continuous.
\end{lemma}

%%%%%%%%%%%%%%%%%%%%%%%%%%%%%%%%%%%%%%%%%%
%%%%%%%%%%%%%%%%%%%%%%%%%%%%%%%%%%%%%%%%%%

\section{Main result for sub-Laplacians in the equiregular case}\label{sec:mainsubriem}

We continue to work in the setting of the previous section. In particular, we suppose that Assumptions \ref{CHOr} and \ref{equireg} are satisfied. Our main statement concerning nilpotent approximation is  the following theorem:

\begin{theorem}\label{t:manifold}
Let $-\Delta_\Xb^{M,\mu} =\sum_{j=1}^p X_j^\star X_j$ be an equiregular sub-Riemannian  Laplacian on a closed connected manifold $M$ with  given measure $\mu$. Then
\begin{equation}\label{eq:mainmfd}
\limsup_{k\rightarrow+\infty} \frac{\nu_k}{k} \leq \Big( \int_M \big( c^{\rm FK}_x \big)^{-\frac Q2} \,d\mu(x) \Big) \Big( \int_M c_x^{\rm Weyl}  \,d\mu(x) \Big)^{-1},
\end{equation}
where $\nu_k$ denotes the maximal number of nodal domains of an eigenfunction of $-\Delta_\Xb^{M,\mu}$ associated with eigenvalue $\lambda_k$. 
\end{theorem}

\begin{corollary}\label{maincor}
If 
\begin{equation}\label{eq:hyp}
\Big( \int_M \big( c^{\rm FK}_x \big)^{-\frac Q2} \,d\mu(x) \Big) \Big( \int_M c_x^{\rm Weyl}  \,d\mu(x) \Big)^{-1} < 1\,,
\end{equation}
then Pleijel's theorem holds. In particular, if
$$
\big( c^{\rm FK}_x \big)^{-\frac Q2} \, \big( c_x^{\rm Weyl} \big)^{-1} < 1
\qquad\text{for all}\ x\in M \,,
$$
then \eqref{eq:hyp} and therefore Pleijel's theorem holds.
\end{corollary}

The first part of the corollary follows immediately from Theorem \ref{t:manifold}. For the second part, we note that both $x\mapsto c^{\rm FK}_x$ and $x\mapsto c_x^{\rm Weyl}$ are continuous, so under the assumption of the second part of the corollary there is a constant $\gamma<1$ such that $\big( c^{\rm FK}_x \big)^{-\frac Q2} \, \big( c_x^{\rm Weyl} \big)^{-1}\leq\gamma$ for all $x\in M$. This implies that the left side of \eqref{eq:hyp} is at most $\gamma$.

\medskip

Theorem \ref{t:manifold} and Corollary \ref{maincor} are improvements due to Y.~Colin de Verdi\`ere of bounds that appeared in a preprint version of this paper \cite{FrHe1}; see also the announcement \cite{FrHe2}. There we had $(\inf_{x\in M} c_x^{\rm FK})^{-\frac Q2} \mu(M)$ instead of $\int_M \big( c^{\rm FK}_x \big)^{-\frac Q2} \,d\mu(x)$. The usefulness of this improvement can be seen, for instance, in the examples in Section \ref{AppA}. We are very grateful to Y.~Colin de Verdi\`ere for allowing us to incorporate his ideas into our paper.

\begin{remark} \label{remdom}
Theorem \ref{t:manifold} remains valid when we consider the Dirichlet realization $-\Delta^{\Omega,\mu}_\Xb$ of an equiregular sub-Riemannian Laplacian in a relatively compact open set $\Omega$ in a manifold $M$. The proof relies on Remark \ref{rem:weylsubsets}. It is enough to have Assumptions \ref{CHOr} and \ref{equireg} satisfied in a neighborhood of $\overline{\Omega}$. \\ 
Examples are discussed in Section \ref{AppA}. \\
In particular, we can consider an open, relatively compact set $\Omega$ in a fixed stratified group $G$, where $\Xb$ is a basis of $\mathcal G_1$ 
 and $\mu$ is (in the exponential coordinates) the Lebesgue measure. In this case, the function $x\mapsto c^{\rm Weyl}_x$ is constant. Situations of this type are further discussed in Part \ref{part2}.
\end{remark}

\begin{remark}\label{remcourantasymp}
The weak version of Courant's theorem (that is, the first part of Theorem~\ref{courant}) yields
\begin{equation}\label{eq:weakcourant}
\limsup_{k\rightarrow+\infty} \frac{\nu_k}k \leq 1\,.
\end{equation}
Here we use the fact that Weyl's formula implies $\lim_{k\to\infty} \frac{{\rm mult}(\lambda_k)}{k}=0$.
\end{remark}

In view of the previous remark, Theorem~\ref{t:manifold} is only interesting for its application in the corollary, that is, when \eqref{eq:hyp} holds. In Part \ref{part2} we will investigate the validity of this condition in the case of an open subset of $\mathbb H_n\times\R^k$. Other instances where one might be able to prove the validity of \eqref{eq:hyp} are, for example, in the form $M_3\times \mathbb T^k$ with $k$ large enough and with $M_3$ a compact $3$-dimensional contact manifold.

\begin{remark}
	In the Riemannian case (that is, when $p=n$ and when $\mu$ is the Riemannian volume measure) the assertion of Theorem \ref{t:manifold} reduces to the theorem of B\'erard and Meyer \cite{BerMe}. Indeed, in this case we have $Q=n$ and for every $x\in M$ the operator $-\widehat\Delta_x$ is the ordinary Laplacian on $G_x=\R^n$. Consequently, in the notation of the introduction and of Part \ref{part2},
	$$
	c^{\rm FK}_x = C^{\rm FK}(\R^n)
	\qquad\text{and}\qquad
	c_x^{\rm Weyl} = \mathcal W(\R^n) \,.
	$$
	Both quantities are independent of $x$ and we arrived at the same bound with constant
	$$
	\gamma(\R^n) = (C^{\rm FK}(\R^n))^{-\frac Q2} \, (\mathcal W(\R^n))^{-1}
	$$
	as in the case of domains in Euclidean space; see \eqref{eq:pleijelintro}.
\end{remark}

\begin{remark}\label{rem:indepmeas}
	The upper bound in \eqref{eq:mainmfd} depends on the manifold $M$ and on the vector fields $X_1,\ldots, X_p$, but it does \emph{not} depend on the measure $\mu$. This follows from Remarks \ref{rem:weylindepmeas} and \ref{rem:fkindepmeas}. Indeed, according to these remarks both integrals on the right side of \eqref{eq:mainmfd} are independent of $\mu$.
\end{remark}

%%%%%%%%%%%%%%%%%%%%%%%%%%%%%

\section{Proof of the main sub-Riemannian results}
\label{sec:mainsubriemproof}

\subsection{Comparing Laplacians}\label{sec:comparinglaplacians}~\\
Throughout this section we choose a Riemannian structure on $M$ that is compatible with its smooth structure. (This is always possible, first in local coordinates and then globally via a partition of unity.) This Riemannian structure allows us to consider (open) geodesic balls $B(x,r)$ at $x\in M$ of radius $r>0$.

Recall that our nilpotent approximation in Section \ref{s1} was carried out on open subsets $W\subset M$ where \eqref{eq:loclinindep} and \eqref{eq:loclinindeplapl} are satisfied. Also in this subsection we will work locally. More precisely, for an open set $W\subset M$ as before we choose a compact subset $K \subset W$. By compactness, there is an $\epsilon_*>0$ such that for any $x\in K$ the privileged coordinates at $x$ are well defined in $B(x,\epsilon_*)$.

In the formulation of the following lemma we identify functions $v$ on $M$ with support in $B(x,\epsilon_*)$ with functions $v\circ\theta_x^{-1}$ on $G_{x}$ with support in $\theta_x(B(x,\epsilon_*))$.

\begin{lemma}\label{l:remainder2}
	Fix $W$ and $K$ as above. Then there are constants $C$, $\epsilon_0>0$ and $s>0$ such that for any $x\in K$, any $0<\epsilon\leq\epsilon_0$ and any function $v\in C^\infty_0(B(x,\epsilon))$ one has
	\begin{align*}
		\left| \langle \widehat{\Delta}_{x} v, v \rangle_{L^2(G_{x,\widehat\mu_{x}})} - \langle \Delta_\Xb^{M,\mu} v, v \rangle_{L^2(M,\mu)} \right| \leq C \epsilon^s \langle - \widehat{\Delta}_{x} v, v \rangle_{L^2(G_{x},\widehat\mu_{x})} \,.
	\end{align*}
\end{lemma}

\begin{proof}
	Throughout the proof, we will make use of the uniformity with respect to $x$ of several geometric constructions around a point $x\in K$. This is discussed and proved in \cite{Me,HN, HN1,Rot} and will be used freely in what follows.
	
	\medskip
	
	We give the proof in three steps.
	
	\medskip
	
	\paragraph{\bf Step 1: Change of the measure}~\\
	We denote by $\tilde\nabla$ the sub-Riemannian gradient, so
	$$
	-\langle\Delta v,v\rangle_{L^2(M,\mu)} = \sum_{j=1}^p \|X_j v\|_{L^2(M,\mu)}^2 = \sum_{j=1}^{n_1} \|\tilde X_j v\|_{L^2(M,\mu)}^2 = \| \tilde\nabla v \|_{L^2(M,\mu)}^2 \,.
	$$
	Due to the localization of the support of $v$, we have
	\begin{align*}
		(1-C\varepsilon) \| \tilde\nabla v \|_{L^2(G_x,\widehat\mu_x)}^2
		\leq \| \tilde\nabla v \|_{L^2(M,\mu)}^2
		\leq (1+C\varepsilon) \| \tilde\nabla v \|_{L^2(G_x,\widehat\mu_x)}^2 \,.
	\end{align*}
	
	\medskip
	
	\paragraph{\bf Step 2: Comparing $\tilde X_{i,x}$ and  $\widehat X_{i,x}$}~\\
	Due to the compactness of $M$, there exists $\epsilon_0>0$ such that, for each $x\in M$ we have a localization function $\tilde \chi_x$ such that $\tilde \chi_x = 1$ on $B(x, 2\epsilon_0)$, $\supp(\tilde \chi_x) \subset B (x,4\epsilon_0)$ and all the estimates on the derivatives are controlled uniformly. After replacing $\epsilon_0$ by $\min\{\epsilon_0,\frac14\epsilon_*\}$ if necessary, we may assume that for each $x\in K$ the privileged coordinates at $x$ are well defined in $B(x,4\varepsilon_0)$.
	
	For technical reasons we also have to introduce another cut-off function $\widehat \chi_x$ of the same type, viz.\ such that $\widehat \chi_x = 1$ on $B(x, \epsilon_0)$, $\supp(\widehat \chi_x) \subset B (x,2\epsilon_0)$ and with uniform bounds on the derivatives.
	
	We observe that in the privileged coordinates centered at $x$  we have\footnote{See either Helffer--Nourrigat \cite[Proposition 5.1]{HN},  M\'etivier \cite{Me},  \cite[Subsection 4.2.4]{EL}, or \cite[Section~1]{Rot}.}  by Lemma~\ref{lemmaMet} that 
	$
	\tilde \chi_x (\tilde X_{i,x} -\widehat X_{i,x}) \tilde\chi_x
	$
	is of degree $0$.
	
	There is also a notion of type\footnote{The authors introduce first a notion of function of type $\lambda$ corresponding to a function that is homogeneous of degree $-Q +\lambda$ with respect to the dilation (with an addition condition when $\lambda=0$). By integration against test functions, this defines a distribution of type $\lambda$. More generally, assuming that $W\subset M$ is such that $\theta_x(y)$ is defined for all $x,y\in W$, then a function $K$ on $W\times W$ is a kernel of type $\lambda$ if for any $\ell \geq 0$, we have
		$$
		K(x,y) = \sum_{i=1}^s a_i(x) k_x^{(i)} (\theta_x(y)) b_i(y) + E^\ell (x,y)\,,
		$$
		where $a_i, b_i \in C_c^\infty(W)$, where $k_x^{(i)}$ is a kernel of type $\geq \lambda$ with $(x,u)\mapsto k_x^{(i)}(u)$  smooth away from $u=0$, and where $E^\ell \in C^\ell( W \times W)$. An operator of type $\lambda$ is a mapping originally defined on $C_c^\infty(W)$ whose distribution kernel is a kernel of type $\lambda$.}
	in \cite{RS} (see \cite[pp.\ 654--655]{Rot} for the adaptation to the equiregular situation), which roughly speaking corresponds to the operator of ``non positive" degree.
	
	Note that (after localization around $0$ in the privileged coordinates) $(-\widehat \Delta_x)^{-1}$ is an operator of type 2 and that $\widehat X_{m,x}(-\widehat \Delta_x)^{-1} $ is an operator of type $1$ \cite[Theorem 8]{RS}. Then $\tilde \chi_x (\tilde X_{i,x} -\widehat X_{i,x}) \tilde\chi_x$  is an operator of type 0. Consequently, for each $m=1,\ldots,n_1$,
	$$
	K_{im,x} := \widehat \chi_x \tilde \chi_x (\tilde X_{i,x} -\widehat X_{i,x}) \tilde\chi_x (-\widehat \Delta_x)^{-1} \widehat X_{m,x}
	$$
	is an operator of type 1. 
	
	For $v \in C_c^\infty (G_x)$,  we have the identity
	\begin{align*}
		\widehat \chi_x (\tilde X_{i,x} -\widehat X_{i,x}) \tilde\chi_x v & = - \sum_m \widehat \chi_x \tilde \chi_x ( \tilde X_{i,x} -\widehat X_{i,x}) \tilde\chi_x (-\widehat \Delta_x)^{-1} \widehat X_{m,x}  \widehat X_{m,x}  \tilde\chi_x v \\
		& = \sum_m  K_{im,x}  \widehat X_{m,x} \tilde \chi_x v \,.
	\end{align*}
	
	To estimate the $L^2$-norm of $ K_{im,x}  \widehat X_{m,x} \tilde \chi_x u$, we first compute the $L^q$-norm with $\frac 1q = \frac 12 - \frac 1Q$. Here we use \cite[Theorem 7]{RS} and obtain for $w$ with support in a fixed compact subset in Euclidean space
	$$
	\|   K_{im,x}   w  \|_q \leq C \|  w  \|_2 \,.
	$$
	
	Assume now that $v\in C^\infty_0(B(x,\epsilon_0))$ and note that, identifying $v$ with a function on $G_x$,
	$$
	\widehat \chi_x ( \tilde X_{i,x} -\widehat X_{i,x}) \tilde\chi_x v
	= ( \tilde X_{i,x} -\widehat X_{i,x}) v 
	\qquad\text{and}\qquad
	\widehat X_{m,x} \tilde \chi_x v = \widehat X_{m,x} v \,.
	$$
	Thus, applying the above inequality to $w= \widehat X_{m,x} \tilde \chi_x v$, which has support in the fixed compact set $B(x,\epsilon_0)$, we obtain
	\begin{align*}
		\| ( \tilde X_{i,x} -\widehat X_{i,x})v \|_q & \leq \sum_m \| K_{im,x}  \widehat X_{m,x} \tilde \chi_x v \|_q \leq C \sum_m \|  \widehat X_{m,x} \tilde \chi_x v \|_2 \\
		& \leq C \sqrt p \left( \sum_m \| \widehat X_{m,x} v \|_2^2 \right)^\frac12.
	\end{align*}
	If, moreover, we have $v\in C^\infty_0(B(x,\epsilon))$ with $\epsilon\leq\epsilon_0$, then, by H\"older's inequality, since $(\tilde X_{i,x} -\widehat X_{i,x})v$ is supported in $B(x,\epsilon)$,
	$$
	\| ( \tilde X_{i,x} -\widehat X_{i,x})v \|_2 \leq \tilde C \epsilon^{\frac n Q} \|  v \|_2
	$$
	for some $\tilde C$. Here we use the fact that $\mu(B(x,\epsilon))$ is bounded by a constant times $\epsilon^n$, uniformly in $x$ and $\epsilon\leq\epsilon_0$.
	
	\medskip
	
	\paragraph{\bf Step 3: End of the proof of the lemma}~\\
	We can now achieve the proof in the following way. Using Step 1, it is enough to use the measure $\widehat\mu_{x}$. We omit this below. Writing
	\begin{equation*}\begin{array}{l}
			\| \tilde \nabla v \|^2  +
			\langle  \widehat \Delta_{x} v, v \rangle \\ 
			=  \sum_i \left( 2 \langle (\tilde X_{i,x}-\widehat X_{i,x}) v, \widehat X_{i,x} v \rangle  + \langle ( \tilde X_{i,x}-\widehat X_{i,x}) v, (\tilde X_{i,x}-\widehat X_{i,x}) v\rangle \right)
		\end{array}
	\end{equation*}
	and using the inequality from Step 2, we arrive at the inequality in the lemma.
\end{proof}

%%%%%%%%%%%%%%%%%%%%%%%%%%%%%
\medskip

\subsection{On the Faber--Krahn constant}\label{sec:fk}~\\
Our goal in this subsection is to prove Lemma \ref{faberkrahn}, which contains the basic properties of the local Faber--Krahn constant $c^{\rm FK}_x$.

\begin{proof}[Proof of Lemma \ref{faberkrahn}]
	We divide the proof into two steps.
	
	\medskip
		
	\paragraph{\bf Step 1: Positivity of the Faber--Krahn constant}~\\
	We shall prove that $c_x^{\rm FK}>0$ for every $x\in M$. The uniform positivity asserted in the lemma then follows from our continuity arguments in Step 2.
	
	Thus, let $x\in M$ and, using Lemma \ref{locindep}, choose an open neighborhood $W\subset M$ of $x$ on which the nilpotentization procedure in Section \ref{s1} can be carried out. In \cite[Theorem 4.24]{Rot} it is shown that for any $1<p<\frac Q2$ there is a $C_{p} >0$ such that for any $f \in C_c^\infty (\mathbb R^n)$, we have
	\begin{equation}
		\label{eq:rothschild}
		\| (-\widehat \Delta_x)^{-1} f \|_q \leq C_{p} \| f \|_p
	\end{equation}
	with $\frac 1q=\frac 1p -\frac 2 Q$. 
	(This inequality is actually used as an intermediate result toward the proof that $-\widehat \Delta_x$ has a $(-Q+2)$-homogeneous fundamental solution $k_x(u)$ that depends smoothly on $x$; see \cite[Theorem 3.6]{Rot}.)\footnote{In \cite{Rot} it is also shown that $C_p$ can be chosen locally bounded with respect to the point $x\in W$, which can be used to give a direct proof of the claimed uniform positivity of $c^{\rm FK}_x$, independent of Step 2.}
	
	Assuming first that $Q>2$, we can apply \eqref{eq:rothschild} with $p=\frac{2Q}{Q+2}$ and $q=\frac{2Q}{Q-2}$ and obtain
	$$
	\| (-\hat\Delta_x)^{-\frac12} f\|_2^2 = 
	\int_{\R^d} f (-\widehat \Delta_x)^{-1} f \,d\widehat\mu_x \leq \| f\|_{\frac{2Q}{Q+2}} \|(-\hat\Delta_x)^{-1} f\|_{\frac{2Q}{Q-2}} \leq C_{\frac{2Q}{Q+2}} \| f \|_{\frac{2Q}{Q+2}}^2
	$$
	Since $(-\hat\Delta_x)^{-\frac12}$, considered as an operator $L^\frac{2Q}{Q+2}\to L^2$, and its adjoint, considered as an operator $L^2\to L^\frac{2Q}{Q-2}$, coincide on $L^\frac{2Q}{Q+2}\cap L^2$ (as a consequence of the selfadjointness of $(-\hat\Delta)^{-\frac12}$ on $L^2$), we obtain
	$$
	\| (-\hat\Delta_x)^{-\frac12} g\|_{\frac{2Q}{Q-2}}^2 \leq C_{\frac{2Q}{Q+2}} \| g \|_{2}^2 \,.
	$$
	Substituting $u=(-\hat\Delta_x)^{-\frac12} g$ we obtain the Sobolev inequality
	$$
	C_{\frac{2Q}{Q+2}}^{-1} \| u \|_{\frac{2Q}{Q-2}}^2 \leq \langle (-\hat\Delta_x) u,u \rangle \,.
	$$
	By H\"older's inequality, as in the proof of Proposition \ref{fksob} below, this implies that $c_x^{\rm FK}\geq C_{\frac{2Q}{Q+2}}^{-1}$. This proves the claimed positivity when $Q>2$.
	
	To deal with the cases $Q=1,2$ we follow an idea of Helffer and Nourrigat already used in \cite{Rot}. We study the operator $-\hat\Delta_x - \sum_{i=1}^s \partial_i^2$ on $G_x\times\R^s$ with $s\in\N$ chosen such that $Q+s>2$. Proceeding exactly as before, we obtain the Sobolev inequality
	$$
	C_{\frac{2Q}{Q+2}}^{-1} \| U \|_{\frac{2(Q+s)}{Q+s-2}}^2 \leq \langle (-\hat\Delta_x - \sum_{i=1}^s \partial_i^2) U,U \rangle \,.
	$$
	for functions $U$ on $G_x\times\R^s$. Applying this to a product function $U = u\otimes \phi$ with a fixed function $\phi\in C^1_c(\R^s)$, we obtain
	$$
	c\, C_{\frac{2Q}{Q+2}}^{-1} \| u \|_{\frac{2(Q+s)}{Q+s-2}}^2 \leq \langle (-\hat\Delta_x ) u, u \rangle + \|u\|_2^2 \,.
	$$
	with $c>0$ depending only on $\phi$. By a simple scaling argument in $G_x$ (see, e.g., \cite[Remark 2.47]{FLW}), this inequality can be brought in the form of a Sobolev interpolation inequality where on the right side a geometric mean of $\langle (-\hat\Delta_x ) u, u \rangle$ and $\|u\|_2^2$ appears. In this form one can again use H\"older's inequality, similarly as in the proof of Proposition \ref{fksob} below, to obtain the desired positive lower bound on the Faber--Krahn constant.
	
	\medskip
	
	An alternative way of proving this step could rely on Varopoulos's proof \cite{Va} of the Sobolev inequality.

	\medskip
	
	\paragraph{\bf Step 2: Continuity of the Faber--Krahn constant}\footnote{The statement and the following proof was suggested to us by Y.~Colin de Verdi\`ere.}~\\
	Let $W$ and $K$ be as in Subsection \ref{sec:comparinglaplacians}. Our aim is to prove that $x\mapsto c_x^{\rm FK}$ is H\"older continuous on $K$. More precisely, we will show that there are positive constants $C,\epsilon',\alpha$ (depending on $K$) such that for all $0<\epsilon\leq\epsilon'$ we have
	\begin{equation}
		\label{eq:fkholder}
		\left| \frac{c^{\rm FK}_x}{c^{\rm FK}_{x_0}} - 1 \right| \leq C \epsilon^\alpha
		\qquad\text{for all}\ x,x_0\in K \ \text{with}\ x\in B(x_0,\epsilon) \,.
	\end{equation}
	Once we have shown this, we can deduce the asserted H\"older continuity on $M$. Indeed, according to Lemma \ref{locindep} we can choose an open neighborhood $W$ around each point and then we can choose a slightly smaller neighborhood $W'$ with $K:=\overline{W'}\subset W$. By compactness we can cover $M$ by finitely many sets $W'$ and then the H\"older continuity on each $K$ implies the H\"older continuity on $M$.
	
	We also note that the pointwise positivity of $c^{\rm FK}_x$, proved in Step 1, together with \eqref{eq:fkholder} implies its uniform positivity on $K$. This implies the uniform positivity on $M$ asserted in the lemma by compactness of $M$.
	
	We turn now to the proof of \eqref{eq:fkholder}. Let $K\subset M$ and recall that there is an $\epsilon_*>0$ such that, for all $x\in K$, the map $\theta_x$ is defined in $B(x,\epsilon_*)$. It maps a neighborhood of $x$ in $M$ to a neighborhood of $0$ in $\R^n$, which we will identify with $G_x$. We fix $x_0\in K$ and restrict our attention to $x\in B(x_0,\frac12\epsilon_*)$. Note that for such $x$ the map $\theta_x$ is defined in $B(x_0,\frac12\epsilon_*)$.
	
	Let $U_x :=\theta_{x}(B(x_0,\frac12\epsilon_*))\subset G_x$. For a set $\Omega\subset U_{x_0}\subset G_{x_0}$ and a function $u\in C^\infty_c(\Omega)$, let 
	$$
	\Omega_x:= \theta_x\circ \theta_{x_0}^{-1}(\Omega)
	\qquad\text{and}\qquad
	u_x := u \circ \theta_{x_0} \circ \theta_x^{-1} \,.
	$$
	Then $u_x\in C^\infty_c(\Omega_x)$. Let $\epsilon_0$ be as in Lemma \ref{l:remainder2}. In the following we consider $x\in B(x_0,\epsilon)$ with $\epsilon\leq\min\{\frac12\epsilon_*,\epsilon_0\}$. It follows from Lemma \ref{l:remainder2} that
	$$
	(1-C \epsilon^s) \langle -\widehat\Delta_{x_0} u, u \rangle_{L^2(G_{x_0},\widehat\mu_{x_0})} \leq 
	\langle -\widehat\Delta_{x} u_x, u_x \rangle_{L^2(G_{x},\widehat\mu_{x})}
	\leq (1+C \epsilon^s) \langle -\widehat\Delta_{x_0} u, u \rangle_{L^2(G_{x_0},\widehat\mu_{x_0})}
	$$
	with constants $C$ and $s$ that are independent of $x$, $x_0$, $\epsilon$, $u$ and $\Omega$. Moreover,
	$$
	(1-C \epsilon) \| u \|_{L^2(G_{x_0},\widehat\mu_{x_0})}^2 \leq 
	\| u_x \|_{L^2(G_{x},\widehat\mu_{x})}^2
	\leq (1+C \epsilon^s) \| u \|_{L^2(G_{x_0},\widehat\mu_{x_0})}^2
	$$	
	with a (possibly different) constant $C$, but again independent of $x$, $x_0$, $\epsilon$, $u$ and $\Omega$. Note that the mapping $u\mapsto u_x$ is a bijection from $C^\infty_c(\Omega)$ to $C^\infty_c(\Omega_x)$. Therefore, combining the above bounds with the variational characterization of the first eigenvalue, we infer that
	\begin{equation}
		\label{eq:fkcontproof1}
		\left| \frac{\lambda_1(\Omega_x)}{\lambda_1(\Omega)} - 1 \right| \leq C \epsilon^{\min\{s,1\}} \,.\end{equation}
	Here $\lambda(\Omega_x)$ denotes the first eigenvalue of the Dirichlet realization of $-\widehat\Delta_x$ in $L^2(\Omega_x,\widehat\mu_x)$ and $\lambda(\Omega)$ denotes the first eigenvalue of the Dirichlet realization of $-\widehat\Delta_{x_0}$ in $L^2(\Omega,\widehat\mu_{x_0})$. Clearly, we also have
	\begin{equation}
		\label{eq:fkcontproof2}
		\left| \frac{\widehat\mu_x(\Omega_x)}{\widehat\mu_{x_0}(\Omega)} - 1 \right| \leq C \epsilon \,.
	\end{equation}
	We deduce that
	$$
	\left( \widehat\mu_x(\Omega_x)\right)^{\frac2Q} \lambda_1(\Omega_x)
	\geq \left(1- C \epsilon^{\min\{s,1\}}\right) \left( \widehat\mu_{x_0}(\Omega) \right)^{\frac2Q} \lambda_1(\Omega) \geq \left(1- C \epsilon^{\min\{s,1\}} \right) c^{\rm FK}_{x_0} \,.
	$$
	Since the map $\Omega\mapsto\Omega_x$ is a bijection from open subsets of $U_{x_0}$ to open subsets of $U_x$, we obtain the inequality
	\begin{equation}
		\label{eq:fkcontproof}
		\left( \widehat\mu_x(\omega)\right)^{\frac2Q} \lambda_1(\omega)
		\geq \left(1- C \epsilon^{\min\{s,1\}} \right) c^{\rm FK}_{x_0}
	\end{equation}
	for any open $\omega\subset U_x$. 
	
	Recall that we have dilations on $G_x$. Under a dilation of $\omega$, the eigenvalue $\lambda_1(\omega)$ of $-\widehat\Delta_x$ is homogeneous of degree $-2$, while the measure $\widehat\mu_x(\omega)$ is homogeneous of degree $Q$. Therefore the product $\left( \widehat\mu_x(\omega)\right)^{\frac2Q} \lambda_1(\omega)$ is homogeneous of degree zero. Since $0\in U_x$ (as $x\in B(x_0,\frac12\epsilon_*)$), we can dilate any bounded set $\omega\subset G_x$ (where boundedness is understood, for instance, with respect to the Euclidean metric on $G_x$ identified with $\R^n$) so that it becomes a subset of $U_x$ and we deduce that \eqref{eq:fkcontproof} holds for any bounded open set $\omega\subset G_x$. Finally, the boundedness assumption on $\omega$ can be relaxed to a finite-measure assumption if we recall that the variational quotient defining $\lambda_1(\omega)$ only needs to be considered for functions in $C^\infty_c(\omega)$. To summarize, we have shown that \eqref{eq:fkcontproof} holds for any open set $\omega\subset G_x$ of finite measure. This proves that
	$$
	c_x^{\rm FK} \geq \left(1- C \epsilon^{\min\{s,1\}} \right) c^{\rm FK}_{x_0} \,.
	$$
	The analogous inequality where the roles of $x$ and $x_0$ are interchanged is deduced from \eqref{eq:fkcontproof1} and \eqref{eq:fkcontproof2} in essentially the same way. This concludes the proof of (H\"older continuity) of $x\mapsto c^{\rm FK}_x$ at $x_0$. Moreover, the constants in this H\"older continuity bound only depend on $K$, as claimed.
\end{proof}

%%%%%%%%%%%%%%%%%%%%%%%%%%%%%

\medskip

\subsection{The Faber--Krahn inequality on small sets}~\\
We now come to the main step in the proof of Theorem \ref{t:manifold}. As in the work of B\'erard and Meyer~\cite{BerMe}, the idea is to prove a Faber--Krahn inequality where the constant is ``almost'' the ``good'' constant, provided the sets on which the inequality is applied are ``small''. For us, the ``good'' constant is in fact a function on $M$, namely $x\mapsto c^{\rm FK}_x$, and we capture the variation of this constant in terms of the measure $(c_x^{\rm FK})^{-\frac Q2}\mu$. The ``smallness'' of sets is understood with respect to their $\mu$-measure. The precise statement is the following:

\begin{proposition}\label{faberkrahnbm}
	For any $\theta>0$ there is an $\eta>0$ such that for any open set $\Omega\subset M$ with $\mu(\Omega)\leq\eta$ and any $v\in C^\infty_0(\Omega)$,
	\begin{equation}
		\langle -\Delta^{M,\mu}_\Xb v,v \rangle \geq  (1-\theta) \left( \int_\Omega \left( c_x^{\rm FK}\right)^{-\frac Q2} d\mu(x) \right)^{-\frac2Q} \,  \| v\|^2\,.
	\end{equation}
\end{proposition}

\begin{proof}
	Recall that by Lemma \ref{locindep} any point $a\in M$ has a neighborhood $W_a$ where the nilpotentization procedure in Section \ref{s1} can be carried out. We fix an open neighborhood $W_a'$ of $a$ with $\overline{W_a'}\subset W_a$. By compactness there are finitely many points $a_1,\ldots,a_L\in M$ such that $\bigcup_{\ell=1}^L W_{a_\ell}' = M$. We apply Lemma \ref{l:remainder2} with $W_{a_\ell}$ and $\overline{W_{a_\ell} '}$ in place of $W$ and $K$ and obtain constants $C$, $\epsilon_0$ and $s$ such that the conclusion of that lemma holds. We may and will assume that these constants are independent of $\ell$.	
	
	For each $\varepsilon\in(0,\epsilon_0]$, we introduce a family of smooth cut-off functions  $\chi_j:M\rightarrow \R$ such that
	\begin{itemize}
		\item  $\sum_j \chi_j^2= 1$ everywhere, 
		\item  for each $j$ there exists $x_j = x_j(\epsilon)\in M$ with $\supp (\chi_j) \subset B(x_j(\epsilon),\epsilon)$,
		\item there exists $C>0$ (independent of $\varepsilon>0$) such that everywhere in $M$,
		$$
		\sum_j |\nabla \chi_j|^2 \leq C \varepsilon^{-2} \,.
		$$
	\end{itemize}
	For any $v\in C^\infty(M)$ we have the identity
	\begin{align}\label{e:identity}
		\langle -\Delta^{M,\mu}_\Xb v,v\rangle_{L^2(M)} & = - \sum_j \langle \chi_j \Delta^{M,\mu}_\Xb v,\chi_j v \rangle_{L^2(M)} \nonumber \\
		&= \sum_j \left( - \langle [\chi_j,\Delta^{M,\mu}_\Xb]v,\chi_jv \rangle_{L^2(M)} - \langle \Delta^{M,\mu}_\Xb (\chi_jv ), \chi_j v\rangle_{L^2(M)} \right) \nonumber \\
		&=\sum_j \left( - \| v \tilde \nabla\chi_j \|_{L^2(M)}^2 - \langle \Delta^{M,\mu}_\Xb (\chi_j v),\chi_j v \rangle_{L^2(M)} \right). 
	\end{align}
	Here, as in the previous proof, $\tilde \nabla$  is not the Euclidean gradient, but the sub-Riemannian gradient, and 
	$$
	\| v\tilde \nabla\chi_j \|_{L^2(M)}^2 = \sum_{\ell} \| v X_\ell \,\chi_j \|_{L^2(M)}^2\,.
	$$
	
	Note first that, by our construction of the $\chi_j$ and the compactness of $M$, there exists a constant $C_1$ such that
	$$
	\sum_{\ell} \| v X_\ell \,\chi_j \|_{L^2(M)}^2 \leq C_1 \epsilon^{-2} \|v\|_{L^2(M)}^2\,.
	$$
	Let $\eta>0$ be a parameter that will be chosen later depending on $\epsilon$. Assuming that $v\in C^\infty_0(\Omega)$ with $\mu(\Omega)\leq\eta$, we use the previous bound to get
	\begin{align}
		\label{eq:remainder1}
		\sum_{\ell} \| v X_\ell \,\chi_j \|_{L^2(M)}^2 \leq C_1 \epsilon^{-2} 
		\eta^\frac2Q \left( \inf_{x\in M} c_x^{\rm FK} \right)^{-1}
		\left( \int_\Omega \left( c^{\rm FK}_x \right)^{-\frac Q2} d\mu(x) \right)^{-\frac 2Q}
		\|v\|_{L^2(M)}^2 \,.
	\end{align}
	Note that the infimum on the right side is finite by Lemma \ref{faberkrahn}.
	
	We turn our attention to the last term in \eqref{e:identity}. Note that Lemma \ref{l:remainder2} is applicable since $\epsilon\leq\epsilon_0$ and since for any $j$ there is an $\ell$ with $x_j\in W_{a_\ell}'$. We infer that
	\begin{align*}
		& \left| \langle \Delta^{M,\mu}_\Xb \chi_j v, \chi_j v \rangle_{L^2(M,\mu)} - \langle \widehat{\Delta}_{x_j} \chi_j v, \chi_j v \rangle_{L^2(G_{x_j,\widehat\mu_{x_j}})} \right| \\
		& \leq C_2 \varepsilon^{s} \langle - \widehat{\Delta}_{x_j} (\chi_j v), \chi_j v \rangle_{L^2(G_{x_j},\widehat\mu_{x_j})} \,.
	\end{align*}
	We combine this bound with the local Faber--Krahn inequality \eqref{eq:FKx1},
	\begin{equation}
		- \langle \widehat \Delta_{x_j}\chi_jv, \chi_jv \rangle_{L^2(G_{x_j},\widehat\mu_{x_j})} \geq c^{\rm FK}_{x_j} \ \widehat\mu_{x_j}(\supp (\chi_j v))^{-\frac2Q} \, \|\chi_j v \|_{L^2(G_{x_j},\widehat\mu_{x_j})}^2 \,.
	\end{equation}
	By the H\"older continuity of $x\mapsto c^{\rm FK}_x$ (Lemma \ref{faberkrahn}) we have
	$$
	c_{x_j}^{\rm FK} \geq (1-C_3 \epsilon^t) \, c_x^{\rm FK}
	\qquad\text{for all} \ x\in B(x_j,\epsilon) \,.
	$$
	This, together with the smoothness of the measure, implies that
	\begin{align*}
		c^{\rm FK}_{x_j} \ \widehat\mu_{x_j}(\supp (\chi_j v))^{-\frac2Q}
		& \geq (1-C_4 \epsilon^t) \left( \int_{\supp(\chi_j v)} \left( c^{\rm FK}_x \right)^{-\frac Q2} d\mu(x) \right)^{-\frac 2Q} \\
		& \geq (1-C_4 \epsilon^t) \left( \int_\Omega \left( c^{\rm FK}_x \right)^{-\frac Q2} d\mu(x) \right)^{-\frac 2Q}.
	\end{align*}
	Moreover, again by the smoothness of the measure,
	$$
	\|\chi_j v \|_{L^2(G_{x_j},\widehat\mu_{x_j})}^2 \geq (1-C_5 \varepsilon) \|\chi_j v \|_{L^2(M)}^2 \,.
	$$
	Thus, we have proved that
	\begin{align*}
		\langle -\Delta^{M,\mu}_\Xb (\chi_j v),\chi_j v \rangle_{L^2(M)} & \geq (1-C_2 \epsilon^s)(1-C_4 \epsilon^t) (1+C_5\epsilon)^{-1} \\
		& \quad \times
		\left( \int_\Omega \left( c^{\rm FK}_x \right)^{-\frac Q2} d\mu(x) \right)^{-\frac 2Q}
		\|\chi_j v \|_{L^2(M)}^2 \,.
	\end{align*}
	
	Summing over $j$, and inserting the resulting bound together with \eqref{eq:remainder1} into \eqref{e:identity}, we obtain
	\begin{align*}
		- \langle \Delta v, v \rangle_{L^2(M)} & \geq \left( (1-C_2 \epsilon^s)(1-C_4 \epsilon^t) (1+C_5\epsilon)^{-1} - C_1 \epsilon^{-2} \left( \inf_{x\in M} c_x^{\rm FK} \right)^{-1} \eta^\frac2Q \right) \\
		& \quad \times 	\left( \int_\Omega \left( c^{\rm FK}_x \right)^{-\frac Q2} d\mu(x) \right)^{-\frac 2Q} \| v \|_{L^2(M)}^2 \,.
	\end{align*}
	Now given $\theta>0$ we first choose $\epsilon\in(0,\epsilon_0]$ such that
	$$
	(1-C_2 \epsilon^s)(1-C_4 \epsilon^t) (1+C_5\epsilon)^{-1} \leq 1- \frac{\theta}{2}
	$$
	and then $\eta>0$ such that
	$$
	C_1 \epsilon^{-2} \left( \inf_{x\in M} c_x^{\rm FK} \right)^{-1} \eta^\frac2Q \leq \frac{\theta}{2} \,.
	$$
	In this way we obtain the claimed inequality.
\end{proof}

%%%%%%%%%%%%%%%%%%%%%%%%%%%%%

\medskip

\subsection{Proof of Theorem \ref{t:manifold}}~\\
We are finally in position to give the proof of our main result for sub-Laplacians in the equiregular case.

\begin{proof}[Proof of Theorem \ref{t:manifold}]
	Let $u_k$ be an eigenfunction associated with $\lambda_k$. Given $\theta>0$, let $\eta$ be as in Proposition \ref{faberkrahnbm} and fix a nodal domain $D_{k\ell}$ of $u_k$ with $\mu(D_{k\ell})\leq\eta$. We denote by $\underline{u}_{k\ell}$ the restriction of $u_k$ to $D_{k\ell}$, extended by $0$ outside $D_{k\ell}$. Then, $\underline{u}_{k\ell}$ is not necessarily in the operator domain of $-\Delta$, but it is in the form domain, as shown in Theorem \ref{restriction}. Therefore it can be approximated with respect to the form-norm by $C^\infty_0(D_{k\ell})$ functions. Therefore we deduce from Proposition \ref{faberkrahnbm} that
	$$
	\langle -\Delta \underline{u}_{k\ell},\underline{u}_{k\ell}\rangle_{L^2(M)}
	\geq (1-\theta) \left( \int_{D_{k\ell}} \left( c_x^{\rm FK} \right)^{-\frac Q2} d\mu(x) \right)^{-\frac 2Q} \| \underline{u}_{k\ell} \|_{L^2(M)}^2 \,. 
	$$
	Here we are slightly abusing notation by writing $\langle -\Delta \underline{u}_{k\ell},\underline{u}_{k\ell}\rangle_{L^2(M)}$ instead of the more precise $\| \tilde \nabla \underline{u}_{k\ell}\|^2_{L^2(M)}$. Similarly, using the weak formulation of the eigenvalue equation and the fact that $\underline{u}_{k\ell}=u_k$ on $D_{k\ell}$ we find
	$$
	\langle -\Delta \underline{u}_{k\ell},\underline{u}_{k\ell}\rangle_{L^2(M)} = \lambda_k \|\underline{u}_{k\ell}\|_{L^2(M)}^{2} \,.
	$$
	Combining the two previous equations and noting that $\underline{u_{k\ell}}$ does not vanish identically, we obtain the inequality
	\begin{equation}\label{e:almostthere} 
		\lambda_k \geq (1-\theta) \left( \int_{D_{k\ell}} \left( c_x^{\rm FK} \right)^{-\frac Q2} d\mu(x) \right)^{-\frac 2Q} \,.
	\end{equation}
	We denote by $\mathcal A_\eta$ the family of nodal sets of $u_k$ satisfying $\mu(D_{k\ell})\leq \eta$. Raising \eqref{e:almostthere} to the power $Q/2$ and summing over $\ell$ we obtain
	$$
	\lambda_k^{\frac Q2} \int_M \left( c_x^{\rm FK} \right)^{-\frac Q2} d\mu(x) 
	\geq \lambda_k^{\frac Q2} \sum_{D_{k\ell}\in \mathcal{A}_\eta} \int_{D_{k\ell}} \left( c_x^{\rm FK} \right)^{-\frac Q2} d\mu(x)  \geq (1-\theta)^\frac Q2 \, ( \#\mathcal{A}_\eta) \,.
	$$
	Clearly, for the number of nodal sets of $u_k$ with $\mu$-measure exceeding $\eta$ we have
	$$
	\eta \#\{ \ell:\ D_{k\ell}\not\in \mathcal A_\eta \} \leq \sum_{D_{k\ell}\not\in \mathcal{A}_\eta} \mu(D_{k\ell}) \leq \mu(M) \,.
	$$
	The two previous relations imply that
	$$
	\frac{\nu_k}{k} \leq (1-\theta)^{-\frac Q2} \left( \int_M \left( c_x^{\rm FK} \right)^{-\frac Q2} d\mu(x) \right) \frac{\lambda_k^\frac{Q}{2}}{k} + \eta^{-1} \mu(M) \frac{1}{k} \,.
	$$
	Combining this bound with the Weyl law from Theorem \ref{weylmet}, we obtain
	$$
	\limsup_{k\to\infty} \frac{\nu_k}{k}\leq (1-\theta)^{-\frac Q2} \left( \int_M \left( c_x^{\rm FK} \right)^{-\frac Q2} d\mu(x) \right) \Big( \int_M c_x^{\rm Weyl} \, d\mu(x) \Big)^{-1}.
	$$
	Since $\theta>0$ is arbitrary, we obtain the bound claimed in Theorem \ref{t:manifold}.
\end{proof}

%We fix local coordinates $x_1,\ldots,x_n$ around $q$, with $q$ the origin. We write $f(x_1,\ldots,x_n)=c+\sum_{j=1}^n a_jx_j$ and identify $f$ with $(a_1,\ldots,a_n)$. Then  $$Q_q(a_1,\ldots,a_n)=\sum_{j,k=1}^n  b_{jk}a_ja_k, \qquad b_{jk}=\sum_{i=1}^m \langle X_i x_j,X_i x_k\rangle_{L^2(M,\mu)} $$ We write $B=(b_{jk})_{\substack{1\leq j\leq n\\1\leq k\leq n}}$. We diagonalize the non-negative quadratic form $Q_q$: there exists an orthogonal matrix $P$ such that $PBP^\top=D$ where $D=(d_i)_{1\leq i\leq n}$ is diagonal, with non-negative elements. In particular, setting $y_i=\sum_{j=1}^n p_{ij}x_j$ we obtain $$ Q_q(a_1,\ldots,a_n)=\sum_{i=1}^n d_ia_i^2, \qquad d_i=\sum_{j,k=1}^n p_{ij}p_{ik}b_{jk}=\sum_{\alpha=1}^m \|X_\alpha y_i\|_{L^2(M,\mu)}^2. $$ When the point $x$ is regular, this may be done in a smooth way in a neighborhood of $x$.

%%%%%%%%%%%%%%%%%%%%%%%
%%%%%%%%%%%%%%%%%%%%%%%

\section{Basic examples}\label{AppA}

In this section we give examples of the applicability of Theorem \ref{t:manifold}. More precisely, we will use the version from Remark \ref{remdom}, which concerns the result on an open subset with Dirichlet boundary conditions. We will present the same computation in various forms which, we hope, illustrates the different techniques that can be applied for a concrete operator.

%%%%%%%%%%%%%%%%%%%%%%%%%%%%%

\medskip

\subsection{Presentation}~\\
We denote coordinates on $\R^3$ by $(x,y,z)$. In an open subset $\Omega\subset\R^3$, we consider the vector fields 
$$ 
X_1= \frac{\partial}{\partial x} + K_1(x,y) \frac{\partial}{\partial z}\,,\,
\qquad
 X_2= \frac{\partial}{\partial y} + K_2 (x,y) \frac{\partial}{\partial z} \,,
$$
und the assumption that
$$
{\rm curl\,} \, \vec K = \frac{\partial}{\partial x} K_2 - \frac{\partial}{\partial y} K_1 >0
$$
As measure $\mu$ we take simply the Lebesgue measure $dx\,dy\,dz$.

Our aim is to give an explicit criterion for getting Pleijel's theorem for 
$$
\Delta=X_1^2+X_2^2 \,.
$$

Let $(x_0,y_0,z_0)\in\Omega$. Then in the construction $G_{x_0,y_0,z_0}$ is the Heisenberg group $\mathbb H$ and the privileged coordinates at $(x_0,y_0,z_0)$ are given (modulo higher order term if we use the canonical privileged coordinates) in the form
\begin{equation} \label{eq:cdv}
	u_1= x-x_0  \,,\qquad  u_2 = y-y_0 \,,\qquad u_3 = \widehat \delta  (z-z_0) + P (x-x_0,y-y_0)\,,
\end{equation}
where $P$ is a polynomial of order $2$ and
$$
\widehat \delta =\frac{1}{{\rm curl\,}\, \vec{K} (x_0,y_0)} \,.
$$
Hence the candidate to be the nilpotentized measure at $(x_0,y_0,z_0)$ is
\begin{equation}\label{eq:exmeas}
	d\widehat\mu_{(x_0,y_0,z_0)} = {\rm curl\,}\, \vec{K} (x_0,y_0)\, du_1 \, du_2 \,du_3\,.
\end{equation}
In these coordinates (which are not a priori privileged\footnote{We refer to \cite{ChoPo1,ChoPo2} for a complete description of possible privileged coordinates} but see the detailed discussion below), we have
\begin{equation}\label{eq:chapeau}
	\widehat X_1= \frac{\partial}{\partial u_1}  - \frac 12 u_2 \frac{\partial}{\partial u_3} \,,\qquad \widehat X_2= \frac{\partial}{\partial u_2}  + \frac 12 u_1  \frac{\partial}{\partial u_3} \,,\qquad	 \widehat X_3 = \frac{\partial}{\partial u_3} \,.
\end{equation}

\medskip

\subsection{Direct change of variables}~\\
Given a point $(x_0,y_0,z_0)$, it is enough to replace $K_1$ and $K_2$ by their linear approximation and it remains to show that 
this linearization will give us $\widehat X_1$ and $\widehat X_2$ in suitable coordinates.\\ Modulo a translation, we can assume that
$(x_0,y_0,z_0)= (0,0,0) \in \Omega$.\\ It remains to find the changes of variables transforming
$$ X_1^{lin}= \frac{\partial}{\partial x} + ( K_1(0,0)) + \alpha x + \beta y) \frac{\partial}{\partial z}\,,\qquad  X_2^{lin}= \frac{\partial}{\partial y} +( K_2 (0,0) + \gamma x + \delta y) ) \frac{\partial}{\partial z} \,,
$$
Note that
$$
[X_1^{lin},X_1^{lin}] = (\gamma -\beta)  \frac{\partial}{\partial z}
$$
with
$$
\gamma -\beta = {\rm curl } K (x_0,y_0)\,.
$$

We proceed step by step for pedagogical reasons, each change of variable having the form described above.
\begin{itemize}
	\item First change to replace $K_1(0,0)$ and $K_2(0,0)$ by $0$.\\
	For this, we take a change of variable in the form 
	$$
	\tilde x = x\,,\qquad
	\tilde y =y\,,\qquad \tilde z = z - K_1(0,0) x - K_2(0,0) y\,,
	$$
	and get in the new coordinates
	$$ 
	X_1^{lin}= \frac{\partial}{\partial \tilde x} + ( \alpha \tilde x + \beta \tilde y) \frac{\partial}{\partial \tilde z} \,,\qquad
  	X_2^{lin}= \frac{\partial}{\partial \tilde y} +( \gamma \tilde x + \delta \tilde y) \frac{\partial}{\partial \tilde z} \,,
	$$
	From now on, we forget the tilde.
	\item Second change of variables in order to have $\alpha=0$ and $\delta =0$. For this,  we take a change of variable in the form
	$$
	\tilde x = x\,,\qquad \tilde y =y\,,\qquad \tilde z = z - \frac 12 \alpha x^2 - \frac 12 \delta y^2\,,
	$$
	and get in the new coordinates
	$$ 
	X_1^{lin}= \frac{\partial}{\partial \tilde x} +  \beta \tilde y \frac{\partial}{\partial \tilde z} \,,\qquad  X_2^{lin}= \frac{\partial}{\partial \tilde y} +  \gamma \tilde x \frac{\partial}{\partial \tilde z} \,.
	$$
	Again  we forget the tilde.
	\item Third change of variables in order to have $\beta^{new} = - \gamma^{new}$. This time, we  we take a change of variable in the form, for some $\rho$ 
	$$
	\tilde x = x\,,\qquad \tilde y =y\,,\qquad \tilde z = z - \rho xy \,,
	$$
	and we impose $\beta -\rho = - (\gamma -\rho)$, which leads to $\rho = ( \beta + \gamma)	/2$. We  get in the new coordinates
	$$ 
	X_1^{lin}= \frac{\partial}{\partial \tilde x} + \frac{ \beta-\gamma}{2}  \tilde y \frac{\partial}{\partial \tilde z} \,,\qquad  X_2^{lin}= \frac{\partial}{\partial \tilde y} - \frac{( \beta -\gamma)}{2}  \tilde x \frac{\partial}{\partial \tilde z} \,.$$
	Note that till now we have respected the Lebesgue measure!
	\item The last step consists in the change of variable
	$$\tilde x = x\,,\qquad \tilde y =y\,,\qquad \tilde z = \widehat \delta z \,.$$
	This time, this changes of variable does not respect the Lebesgue measure and involves the curl of $\vec K$ at the chosen point $(x_0,y_0)$.  Note that with this change of variables (although it is not the canonical privileged coordinate, but see below the second approach) we have obtained the conclusion of Lemma \ref{lemmaMet} which is what is used in the proof referring to \cite{Rot}.\\
\end{itemize}
More precisely, if we write $(X_1,X_2)$ in the new system of coordinates we get
$$
\begin{array}{ll}
	X_{1,(x_0,y_0,z_0)} & = \widehat X_1 + \big(R_{1,x_0,y_0}/ {\rm curl } K (x_0,y_0) \big)\frac{\partial}{\partial u_3},\\  X_{2,(x_0,y_0,z_0)} & = \widehat X_2 + \big(R_{2,x_0,y_0}/ {\rm curl } K (x_0,y_0)\big) \frac{\partial}{\partial u_3},
\end{array}
$$
Here, if we write $K_j(x,y)= K_j(x_0,y_0) + \nabla K_j(x_0,y_0) \cdot (x-x_0, y-y_0) + r_{j,x_0,y_0} (x,y)$, we have
$$
R_{j,x_0,y_0}(u_1,u_2) := r_{j,x_0,y_0}(x_0+u_1,y_0+u_2)= \mathcal O (u_1^2 + u_2^2)
$$

\medskip

\subsection{Computation of the canonical privileged coordinates and application}~\\
We start from \eqref{eq:pc} with
$$
Y_1:=X_1 \,,\qquad
Y_2:=X_2 \,,\qquad
Y_3:=[X_1,X_2] = {\rm curl\,} \vec K\, \frac{\partial}{\partial z}\,.
$$
This equation reads with the notation of above and assuming $(x_0,y_0,z_0)= (0,0,0)$
\begin{equation}\label{eq:pc3}
	(x,y,z)=  \exp (\sum u_{i}Y_i) \cdot (0,0,0)\,.
\end{equation}
We observe that 
$$
\sum u_{i}Y_i = u_1 \frac{\partial}{\partial x} + u_2 \frac{\partial}{\partial y} +(  {\rm curl\,} \vec K\, u_3  + K_1 u_1 + K_2 u_2  ) \frac{\partial}{\partial z} \,.
$$
To compute the right hand side of \eqref{eq:pc3} we have to compute the solution at $t=1$ of the differential system
$$
\begin{array}{ll}
	\frac{dx}{dt} &=u_1\\
	\frac{dy}{dt} &=u_2 \\
	\frac{dz}{dt} &= {\rm curl\,} \vec K (x,y)\, u_3  + K_1(x,y) u_1 + K_2 (x,y)u_2\,.
\end{array}
$$
with $(x(0),y(0),z(0))=(0,0,0)$.\\
This is easy to solve. We have $x(t)= tu_1$, $y(t) =t u_2$ and
$$
z(t) = \int_0^t ({\rm curl\,} \vec K (su_1,su_2)\, u_3  + K_1(su_1,su_2) u_1 + K_2 ( su_1,su_2)u_2 )\, ds \,.
$$

Hence we get 
\begin{align*}
	(x,y,z) & :=\exp (\sum u_{i}Y_i) \cdot (0,0,0)\\
	&= \left( u_1,u_2,  \int_0^1({\rm curl\,} \vec K (su_1,su_2)\, u_3  + K_1(su_1,su_2) u_1 + K_2  (su_1,su_2)u_2) \, ds\right) .
\end{align*}
This is a local diffeomorphism since the differential of the map $(u_1,u_2,u_3) \mapsto (x,y,z)$ at $(u_1,u_2,u_3)=(0,0,0)$ is
$$
\left(\begin{array}{ccc}
	1&0&0\\
	0&1&0\\
	K_1(0,0)&K_2(0,0)&{\rm curl\,} \vec K (0,0)
\end{array}\right).
$$
We note (for comparison with the first approach) that
\begin{align*}
	& \int_0^1({\rm curl\,} \vec K (su_1,su_2)\, u_3  + K_1(su_1,su_2) u_1 + K_2 ( (su_1,su_2)u_2)\, ds \\
	& \quad = {\rm curl\,} \vec K (0,0) u_3 + K_1(0,0) u_1+ K_2(0,0) u_2\\
	& \qquad
	+ \frac12 \Big[ \frac{\partial K_1}{\partial x}(0,0) u_1^2+ \frac{\partial K_1}{\partial y}(0,0) u_1u_2 + \frac{\partial K_2}{\partial x}(0,0) u_1u_2 + \frac{\partial K_2}{\partial y}(0,0) u_2^2\Big]\\ 
	& \qquad + r(u_1,u_2,u_3)
\end{align*}
with $r$ of degree  $\leq -3$ for the dilation. Note that if we neglect the remainder $r$, this change of variable is exactly \eqref{eq:cdv}.

For the inverse we get indeed $u_1=x$, $u_2=y$ and
\begin{align*}
	u_3 & = \frac{1}{{\rm curl\,} \vec K (0,0)} \Big( z - \frac12 \big[ \frac{\partial K_1}{\partial x}(0,0) x^2+ ( \frac{\partial K_1}{\partial y}(0,0)  + \frac{\partial K_2}{\partial x}(0,0) )xy 
	+ \frac{\partial K_2}{\partial y}(0,0) y^2\big]\Big) \\ & \quad  + r(x,y,z)\,.
\end{align*}
The computation then go in the same way as in the previous approach with this choice of privileged coordinates. Hence we get the same $\widehat X_i$ at $(x_0,y_0,z_0)$, which are independent of $(x_0,y_0,z_0)$ and given by \eqref{eq:chapeau}. It is clear that  \eqref{eq:exmeas} also holds.

We can then apply Lemma \ref{lemmaMet}.

\medskip

\subsection{Conclusion in the setting of the example}~\\
In order to apply Theorem \ref{t:manifold} we need to compute the local Weyl constant $c^{\rm Weyl}_{(x_0,y_0,z_0)}$, as well as the local Faber--Krahn constant $c_{(x_0,y_0,z_0)}^{\rm FK}$. We will express them in terms of the corresponding constants on the Heisenberg group $\mathbb H_1$, which will be studied in more detail in Part \ref{part2} of this paper.

There is, however, a slight notational inconsistency coming from different normalizations. In the present section we arrived at the vector fields \eqref{eq:chapeau}, whereas in Part~\ref{part2} we will find it more convenient to work in the formulation \eqref{Otherformulation}. These two formulations are equivalent via scaling. More precisely, if $\mathcal W(\mathbb H_1)$ and $C^{\rm FK}(\mathbb H_1)$ denote the Weyl and Faber--Krahn constants in the normalization of \eqref{Otherformulation} and if
$\widetilde{\mathcal W}(\mathbb H_1)$ and $\widetilde C^{\rm FK}(\mathbb H_1)$ denote the corresponding constants in the normalization of \eqref{eq:chapeau}, then
$$
\widetilde{\mathcal W}(\mathbb H_1) = 4\, \mathcal W(\mathbb H_1)
\qquad\text{and}\qquad
\widetilde C^{\rm FK}(\mathbb H_1) = 2^{-1}\, C^{\rm FK}(\mathbb H_1) \,.
$$
As a consequence, the combination
\begin{equation}
	\label{eq:pleijelheisenexample}
	\left( \widetilde C^{\rm FK}(\mathbb H_1) \right)^{-2} \left( \widetilde{\mathcal W}(\mathbb H_1) \right)^{-1} = \left( C^{\rm FK}(\mathbb H_1) \right)^{-2} \left( \mathcal W(\mathbb H_1) \right)^{-1} = \gamma(\mathbb H_1)
\end{equation}
is independent of the normalization. The right equality in \eqref{eq:pleijelheisenexample} is a definition; see \eqref{eq:defgamma}. The number $\gamma(\mathbb H_1)$ plays the role of the Pleijel constant on $\mathbb H_1$.

\begin{lemma}\label{exampleconst}
	In our example we have for every $(x_0,y_0,z_0)\in\Omega$,
	$$
	c^{\rm Weyl}_{(x_0,y_0,z_0)} = \frac{\widetilde{\mathcal W}(\mathbb H_1)}{\mathrm{curl}\, \vec{K}(x_0,y_0)}
	\qquad\text{and}\qquad
	c_{(x_0,y_0,z_0)}^{\rm FK} = \sqrt{\mathrm{curl}\, \vec{K}(x_0,y_0)} \ \widetilde C^{\rm FK}(\mathbb H_1) \,.
	$$
\end{lemma}

\begin{proof}
	The assertions follow by a simple scaling argument. According to \eqref{eq:chapeau} the differential expression for the nilpotent approximation $\widehat\Delta_{(x_0,y_0,z_0)}$ coincides with the Laplacian on the Heisenberg group, the only difference is that the measure $\widehat\mu_{(x_0,y_0,z_0)}$ is Lebesgue measure multiplied by the positive constant $\mathrm{curl}\, \vec{K}(x_0,y_0)$; see \eqref{eq:exmeas}. This factor leads to the corresponding expressions for the local Weyl and Faber--Krahn constants.
\end{proof}

It follows from Lemma \ref{exampleconst} that
$$
\int_\Omega c^{\rm Weyl}_{(x,y,z)} \,dx\,dy\,dz = \widetilde{\mathcal W}(\mathbb H_1) \ \int_\Omega \frac{dx\,dy\,dz}{\text{curl}\, \vec{K}(x,y)}
$$
and
$$
\int_\Omega (c_{(x,y,z)}^{\rm FK})^{-2} \,dx\,dy\,dz = \widetilde C^{\rm FK}(\mathbb H_1) \, \int_\Omega \frac{dx\,dy\,dz}{\text{curl}\, \vec{K}(x,y)} \,.
$$
Therefore, from Theorem \ref{t:manifold} (or, more precisely, its version for the sub-Laplacian on an open set with Dirichlet boundary conditions in Remark \ref{remdom}) we deduce the bound
$$
\limsup_{k\rightarrow+\infty} \frac{\nu_k}{k} \leq \left( \widetilde C^{\rm FK}(\mathbb H_1) \right)^{-2} \left( \widetilde{\mathcal W}(\mathbb H_1) \right)^{-1} = \gamma(\mathbb H_1)
$$
Here we used \eqref{eq:pleijelheisenexample}. Our conclusion is that \emph{if the Pleijel constant $\gamma(\mathbb H_1)$ on the Heisenberg group is $<1$, then Pleijel's theorem holds in the present example.}

We will investigate the validity of the inequality $\gamma(\mathbb H_1)<1$ in Part \ref{part2}. Currently we have no proof of this bound, but we will prove that it holds provided a well-known conjecture by Pansu concerning the isoperimetric inequality on the Heisenberg group is true.

%%%%%%%%%%%%%%%%%%%%%%%%%%%%%%%%%%

\medskip

\subsection{Variations of the example}~\\
To have a example where we can prove that Pleijel's theorem holds, we could repeat the above analysis with
$$
X_1^2 + X_2^2 +\Delta_w
$$
on $\Omega \subset \R^3 \times \mathbb R^k$ with $k\geq 3$. As will be discussed in Part \ref{part2}, in this case the analogue $\frac{1}{\gamma (\mathbb H \times \mathbb R^k)}$  of $\left( \widetilde C^{\rm FK}(\mathbb H_1) \right)^2 \widetilde W (\mathbb  H_1) $ is indeed $>1$ (see Theorem \ref{maincomp} together with \eqref{eq:defgamma}). Therefore Pleijel's theorem holds in this example. We do not carry out the details.

\medskip

More generally, in an open subset of $\R^n\times\R^n\times\R$, $n\in\mathbb N$, with coordinates $(\vec x,\vec y,z)$, $\vec x =(x_1,\dots,x_n)$ and $\vec y=(y_1,\cdots, y_n)$, we consider the following vector fields for $j=1,\ldots, n$:
$$
\begin{array}{ll}
	X'_j&= \partial_{x_j} - K_1^j (x_j,y_j) \partial_z\,,\\
	X''_j& = \partial_{y_j} - K_2^j (x_j,y_j) \partial_z\,.
\end{array}
$$
We assume that ${\rm curl \,} K^j >0$ for $j=1,\ldots, n$.\\
For the nilpotentization at $(\vec x,\vec y,z)$ this leads to the following vector fields on $\mathbb H_n$,
$$
\begin{array}{ll}
	\hat X'_j & = {\rm curl} \,K^j(x_j,y_j)^{1/2} (\partial _{u_j} -\frac 12 v_j \partial_z )\,,\\
	\hat X''_j& = {\rm curl} \,K^j(x_j,y_j)^{1/2} (\partial _{v_j} +\frac 12 u_j \partial_z)\,.
\end{array}
$$
together with the Haar measure $(\prod_{j=1}^n {\rm curl} \, K^j(x_j,y_j))^{-1} \,du dv dz\,.$\\
For the Faber--Krahn part, one has to consider on $\mathbb H_n$
$$
\widehat \Delta_{\vec x,\vec  y}= \sum_j  {\rm curl} \, K^j(x_j,y_j) \Big((\partial _{u_j} -\frac{ v_j}{2}  \partial_z )^2+(\partial _{v_j} +\frac {u_j}{2}\partial_z )^2 \Big)\,.
$$

When ${\rm curl\,} K^j \equiv 1$, it will be clear from our analysis in Part \ref{part2} (Theorem \ref{maincomp}) that Pleijel's theorem holds for $n\geq 4$ in any open set of $\Omega$. This property remains true for general $K^j$ by a somewhat similar analysis as in the case $n=1$.

%%%%%%%%%%%%%%%%%%%%%%%%%%%%
%%%%%%%%%%%%%%%%%%%%%%%%%%%%
\newpage 

\part{Pleijel's bound for $\mathbb H_n\times\R^k$}\label{part2}

\section{The Pleijel argument for $\mathbb H_n\times\R^k$ }\label{sec:mainhnrk}

We work on $\mathbb H_n\times\R^k$, where $n\in\N$, $k\in\N_0$. The case $k=0$ corresponds to the Heisenberg group $\mathbb H_n$. (Everything remains valid for $n=0$ as well, that is, for $\R^k$, but in this case the results below are well known.) Typically, we will denote coordinates in $\mathbb H_n$ by $(x,y,z)$ with $x,y\in\R^n$ and $z\in\R$, and we will denote coordinates in $\R^k$ by $w$. The measure $dx\,dy\,dz\,dw$ is the Lebesgue measure on $\R^{2n+1+k}$. For the vector fields we will use the following normalization\footnote{In this and the remaining sections it is more convenient to use another normalization than in the first sections. This simply corresponds to a scaling of the $z$-variable.},
\begin{subequations}\label{Otherformulation}
\begin{equation}
X_j = \partial_{x_j} + 2y_j \partial_z \,,\qquad
Y_j = \partial_{y_j} - 2x_j \partial_z \,,\qquad
W_j = \partial_{w_j} \,.
\end{equation}
The sub-Laplacian is
\begin{equation}
\Delta^{\mathbb H_n\times\R^k} = \sum_{j=1}^n (X_j^2 + Y_j^2) + \sum_{i=1}^k W_i^2 \,.
\end{equation}
\end{subequations}

If $\Omega\subset\mathbb H_n\times\R^k$ is an open set of finite measure, then the spectrum of the Dirichlet realization of $ -\Delta^{\mathbb H_n\times\R^k}_\Omega $ is discrete and we can denote its eigenvalues, in nondecreasing order and repeated according to multiplicities, by $\lambda_\ell(\Omega)$, $\ell\in\N$. We know that eigenfunctions are  $C^\infty$ in $\Omega$ (\cite{Ho}) and therefore the nodal domains are well defined as the connected components of the complement of their zero set in $\Omega$. We denote by $\nu_\ell(\Omega)$ the maximal number of nodal domains of eigenfunctions corresponding to eigenvalue $\lambda_\ell(\Omega)$. We are interested in an upper bound on
$$
\limsup_{\ell\to\infty} \frac{\nu_\ell(\Omega)}{\ell}
$$
that depends only on $n$ and $k$. The simplest such upper bound is given by one, as mentioned in Remark \ref{remcourantasymp}.

In the spirit of Pleijel's theorem, here we try to improve upon the upper bound by one. Just as Pleijel's bound, our bound depends on two constants, namely the constant in the Weyl asymptotics and the Faber--Krahn constant. Let us introduce these constants. The Weyl asymptotics in the case with boundary (which was also established by G.~M\'etivier in \cite[Theorem 1.3]{Me}) states
that, for any open set $\Omega\subset\mathbb H_n\times\R^k$ of finite measure,
\begin{equation}\label{Metwithbdry}
\lambda^{-\frac{2n+2+k}{2}} \#\{ \ell:\ \lambda_\ell(\Omega)<\lambda \} \to \mathcal W(\mathbb H_n\times\R^k) \, |\Omega| 
\qquad\text{as}\ \lambda\to\infty \,.
\end{equation}
 We will give a (relatively) explicit expression for the constant $\mathcal W(\mathbb H_n\times\R^k)$ in the next section.

The Faber--Krahn constant $C^{\rm FK}(\mathbb H_n\times\R^k)$ is defined to be the largest constant such that for any open $\Omega\subset\mathbb H_n\times\R^k$ of finite measure and for any $u\in S^1_0(\Omega)$ one has
\begin{align}
	\label{eq:fk}
	& \int_\Omega (\sum_{j=1}^n ((X_ju)^2 + (Y_ju)^2) + \sum_{i=1}^k (W_iu)^2)\,dx\,dy\,dz\,dw \notag \\
	& \quad \geq C^{\rm FK}(\mathbb H_n\times\R^k) |\Omega|^{-\frac{2}{2n+2+k}} \int_\Omega u^2 \,dx\,dy\,dz\,dw \,.
\end{align}
Here $S^1_0(\Omega)$ (see Subsection \ref{ss2.2})  denotes the form domain of the Dirichlet realization of $-\Delta$ on $\Omega$ or, equivalently, the completion of $C^1_c(\Omega)$ with respect to the quadratic form $$u \mapsto \int_\Omega (\sum_{j=1}^n ((X_ju)^2 + (Y_ju)^2) + \sum_{i=1}^k (W_iu)^2 + u^2)\,dx\,dy\,dz\,dw\,.$$The defining inequality for the Faber--Krahn constant can also be stated as
$$
\lambda_1(\Omega) \geq C^{\rm FK}(\mathbb H_n\times\R^k) \, |\Omega|^{-\frac{2}{2n+2+k}}
$$
for all open $\Omega\subset\mathbb H_n\times\R^k$ of finite measure.

Let us set
\begin{equation}\label{eq:defgamma}
\gamma(\mathbb H_n\times\R^k) := \left( C^{\rm FK}(\mathbb H_n\times\R^k) \right)^{-\frac{2n+2+k}{2}} \left( \mathcal W(\mathbb H_n\times\R^k) \right)^{-1} \,.
\end{equation}
Here is our Pleijel-type bound.

\begin{theorem}\label{pleijelnk}
	For any open $\Omega\subset\mathbb H_n\times\R^k$ of finite measure,
	$$
	\limsup_{\ell\to\infty} \frac{\nu_\ell(\Omega)}{\ell} \leq \gamma(\mathbb H_n\times\R^k) \,.
	$$
\end{theorem}

\begin{proof}[Proof of Theorem \ref{pleijelnk}]
	We consider an eigenfunction $u$ corresponding to the eigen\-value $\lambda_\ell(\Omega)$. Let $(\omega_\alpha)_\alpha$ be its nodal domains and let $\nu_\ell(u)$ be their number. (We will see shortly that this number is finite.) By Theorem \ref{restriction} with $M=\mathbb H_n\times\R^k$, 
 we know that $\lambda_\ell(\Omega) = \lambda_1(\omega_\alpha)$ and that $u|_{\omega_\alpha}$ is the ground state of the Dirichlet realization on $\omega_\alpha$. Thus,
	\begin{align*}
		\frac{\nu_\ell(u)}{\ell} & \!=\! \frac{\lambda_\ell(\Omega)^\frac{2n+2+k}{2}}{\ell} \sum_\alpha \lambda_1(\omega_\alpha)^{-\frac{2n+2+k}{2}} \leq \frac{\lambda_\ell(\Omega)^\frac{2n+2+k}{2}}{\ell} \left( C^{\rm FK}(\mathbb H_n\times\R^k) \right)^{-\frac{2}{2n+2+k}} \sum_\alpha |\omega_\alpha| \\
		& \!\leq\! \frac{\lambda_\ell(\Omega)^\frac{2n+2+k}{2}}{\ell} \left( C^{\rm FK}(\mathbb H_n\times\R^k) \right)^{-\frac{2}{2n+2+k}} |\Omega| \,.
	\end{align*}
	Since this is true for any eigenfunction corresponding to $\lambda_\ell(\Omega)$, we deduce that
	$$
	\frac{\nu_\ell(\Omega)}{\ell} \leq \frac{\lambda_\ell(\Omega)^\frac{2n+2+k}{2}}{\ell} \left( C^{\rm FK}(\mathbb H_n\times\R^k) \right)^{-\frac{2}{2n+2+k}} |\Omega| \,.
	$$
	Taking the limsup as $\ell\to\infty$ and with in mind  the Weyl asymptotics, we arrive at the claimed bound.
\end{proof}

We recall that by Theorem \ref{courant} we have $\limsup_{\ell\to\infty} \frac{\nu_\ell(\Omega)}{\ell} \leq 1$; see also Remark \ref{remcourantasymp}. In the remaining sections of this paper we will give sufficient conditions on $n$ and $k$ to have $\gamma(\mathbb H_n\times\R^k)<1$. We recall that for $n=0$ this was shown to be the case for $k=2$ by Pleijel \cite{Pl} and for general $k$ by B\'erard and Meyer \cite{BerMe}. Moreover, Helffer and Persson Sundqvist \cite{HPS} showed that, for $n=0$, the sequence $k\mapsto\gamma(\R^k)$ is decreasing. Here we shall prove, among other things, the following.

\begin{theorem}\label{maincomp}
	Let $n\geq 1$ and $k\geq 0$ with $(n,k)\not\in\{(1,0),(2,0),(3,0),(1,1)\}$. Then $\gamma(\mathbb H_n\times\R^k)<1$.
\end{theorem}

The proof of Theorem \ref{maincomp} is somewhat long and spread out over several sections. Here is a guide. The part concerning $k=0$ is proved in Subsection \ref{sec:maincompk0} and that concerning $n\geq 3$ in Subsection \ref{sec:maincompk}. The part concerning $n=1$ and $n=2$ is proved in Subsection \ref{sec:maincompn12}.

There is a well known conjecture, due to Pansu \cite{Pa2}, about the sharp isoperimetric constant on the Heisenberg group. It is generally believed to be true and supported by several partial results. We will discuss this in some detail in Section \ref{sec:isoperimetric}. We shall show that the validity of this conjecture implies Pleijel's bound.

\begin{proposition}\label{pansupleijel}
	Let $n\in\N$ and assume that Pansu's conjecture \eqref{eq:pansusconj} holds. Then $\gamma(\mathbb H_n\times\R^k)<1$ for all $k\in\N_0$.
\end{proposition}

We will prove this proposition for $k=0$ and $n=1,2,3$ in Corollary \ref{pansupleijelk0} and for $k=1$ and $n=1$ in Subsection \ref{sec:maincompn12}. In the remaining cases Theorem \ref{maincomp} applies.

%%%%%%%%%%%%%%%%%%%%%%%%%%%%%%%%%%%%
%%%%%%%%%%%%%%%%%%%%%%%%%%%%%%%%%%%%
\medskip

\section{Computing the constant in the Weyl asymptotics}\label{sec:weylhnrk}

\medskip

\subsection{The case of $\mathbb H_n$}\label{sec:weylheis1}~\\
As discussed in Theorem \ref{weylmet} and \eqref{Metwithbdry} for the case with boundary, the Weyl asymptotics for the Dirichlet realization of the sub-Laplacian in open subsets $\Omega$ of $\mathbb H_n$ state that
$$
\lambda^{-\frac{2n+2}{2}} \#\{ \ell:\ \lambda_\ell(\Omega)<\lambda \} \to  \mathcal W(\mathbb H_n) \ |\Omega| 
\qquad\text{as}\ \lambda\to\infty \,.
$$
Hansson--Laptev \cite{HL} have shown that these asymptotics hold under the sole assumption that $\Omega$ is an open set of finite measure.

Since we are interested in a relatively explicit expression of the constant $\mathcal W(\mathbb H_n)$ and since Hansson and Laptev use a different normalization from ours, we repeat part of their argument. They show that
\begin{equation}\label{w1}
	N(\lambda,-\Delta_\Omega^{\mathbb H_n}) \sim \int_\Omega \1(-\Delta^{\mathbb H_n} <\lambda)((x,y,z),(x,y,z)) \,dx\,dy\,dz 
	\qquad\text{as}\ \lambda\to\infty \,,
\end{equation}
where $\1(-\Delta^{\mathbb H_n} <\lambda)((x,y,z),(x,y,z))$ is the on-diagonal spectral density of the sub-Laplacian on all of $\mathbb H_n$. (To be more precise, using coherent states Hansson and Laptev show \eqref{w1} when integrated over $\lambda$. Then a  Tauberian theorem yields \eqref{w1} as stated.)

By translation invariance of the sub-Laplacian on $\mathbb H^n$, we know that the diagonal of the spectral density $\1(-\Delta^{\mathbb H_n} <\lambda)((x,y,z),(x,y,z))$ is independent of the point $(x,y,z)$. Moreover, by dilation covariance, we know that it is proportional to $\lambda^\frac{Q}{2}$. Thus, there is a constant $\mathcal W(\mathbb H_n)>0$ such that 
\begin{equation}\label{w2}
	\1(-\Delta^{\mathbb H_n} <\lambda)((x,y,z),(x,y,z)) = \mathcal W(\mathbb H_n)\ \lambda^\frac Q2 \,,
\end{equation}
and we obtain the above form of the spectral asymptotics.

In the following we are interested in finding an explicit expression for $\mathcal W(\mathbb H_n)$. We proceed by an explicit diagonalization of the operator $-\Delta^{\mathbb H_n}$. By a Fourier transform with respect to $z$ one arrives at the family of operators
$$
- \sum_{j=1}^n \left( (\partial_{x_j} + 2iy_j  \zeta)^2 + (\partial_{y_j} - 2i x_j  \zeta)^2 \right),
$$
where $\zeta\in\R$ is the Fourier variable dual to the variable $z$. 

Now for each fixed $j$,
$$
- \left( (\partial_{x_j} + 2i y_j  \zeta)^2 + (\partial_{y_j} - 2ix_j \zeta)^2 \right),
$$
is a Landau Hamiltonian corresponding to constant magnetic field with intensity $4|\zeta|$. These $n$ Hamiltonians are independent of each other. The spectrum of each one is given by $4|\zeta|(2k_j +1)$, $k_j\in\N_0 =\{0,1,2,\ldots\}$. 

The spectral function of the operator $-\Delta$ on $\mathbb H_n$ is then given by
\begin{align*}
	\1(-\Delta<\lambda)((x,y,z),(x,y,z)) & = \int_\R \frac{d\zeta}{2\pi} \left( \frac{4|\zeta|}{2\pi} \right)^n \sum_{k \in\N_0^n} \1(4|\zeta|(2(k_1+\ldots+k_n)+n)<\lambda) \\
	& = \frac{4^n}{(2\pi)^{n+1}} \sum_{k\in\N_0^n} \frac{2}{n+1} \left( \frac{\lambda}{4(2(k_1+\ldots+k_n)+n)} \right)^{n+1} \\
	& = \frac{1}{2(n+1)} \frac{\lambda^{n+1}}{(2\pi)^{n+1}} \sum_{k\in\N_0^n} \left( \frac{1}{2(k_1+\ldots+k_n)+n} \right)^{n+1} \\
	& = \frac{1}{2(n+1)} \frac{\lambda^{n+1}}{(2\pi)^{n+1}} \sum_{m\in\N_0} \binom{m+n-1}{m} \left( \frac{1}{2m+n} \right)^{n+1}
	\,.
\end{align*}
Thus, we have shown that
\begin{equation}\label{w4}
	\mathcal W(\mathbb H_n) = \frac{1}{2(n+1)} \frac{1}{(2\pi)^{n+1}} \sum_{m\in\N_0} \binom{m+n-1}{m} \frac{1}{(2m+n)^{n+1}} \,.
\end{equation}
Note that 
\begin{equation}\label{eq:W1}
	\mathcal W(\mathbb H) = \frac{1}{4} \frac{1}{(2\pi)^2} \sum_{m\in\N_0} \frac{1}{(2m+1)^2} = \frac14 \frac1{(2\pi)^2} \frac{\pi^2}8 = \frac1{128}
\end{equation}
and
\begin{equation}\label{eq:W2}
	\mathcal W(\mathbb H_2) = \frac{1}{6} \frac{1}{(2\pi)^3} \sum_{m\in\N_0} \frac{m+1}{(2m+2)^3} = \frac{1}{6} \frac{1}{(2\pi)^3} \frac{\pi^2}{48} = \frac{1}{48^2 \pi} \,.
\end{equation}
It seems like Hansson--Laptev \cite{HL} were not aware of a more explicit form of $\mathcal W(\mathbb H_n)$ for $n\geq 3$. We will give explicit formulas below for $3\leq n \leq 13$ and propose a general conjecture.

Let us set
\begin{equation}\label{eq:defcn}
	c_n := \sum_{m\in\N_0} \binom{m+n-1}{m} \frac{1}{(2m+n)^{n+1}} \,.
\end{equation}
We have seen above that 
$$
c_1=\frac{\pi^2}8 \qquad \text{and}\qquad
c_2 = \frac{\pi^2}{48}\,.
$$
Numerical values for the $c_n$, $n=3,4,5,6$, are given in the Hansson--Laptev paper \cite{HL} and can be completed by using Wolfram Alpha or  Mathematica\footnote{Thanks to J. Viola and F. Nicoleau for their help.}.
This program shows first  that explicit formulas can be found for $c_n$ and then gives numerical values, which are consequently quite accurate. We get
$$
\begin{tabular}{lllll}
	$c_3 $& $= $& $ \pi^2(12-\pi^2)/768$ &$\approx $&$ 2.7378\;\cdot \; 10^{-2}$\\
	$c_4$ & $=$ &$ \pi^2 (15-\pi^2) / 17280$& $\approx$&$ 2.9303 \cdot \; 10^{-3}$\\
	$c_5$&$ = $&$ \pi^2 (120-100\pi^2 + 9 \pi^4) / 368640$&$ \approx $&$ 2.6027\cdot \; 10^{-4}$\\
	$c_6$&$=$&$ \pi^2(315-105 \pi^2+ 8 \pi^4) / 29030400$&$ \approx $&$ 1.9706 \cdot \; 10^{-5}$\\
	$c_7$&$=$&$\pi^2 (6720 - 19600 \pi^2 + 14504 \pi^4 - 1275 \pi^6) /2477260800$&$\approx$&$ 1.2988\cdot \;10^{-6}$\\
	$c_8 $&$= $&$\pi^2 (1575 - 1470 \pi^2 + 490 \pi ^4 - 36 \pi ^6) /24385536000$&$\approx$&$  7.5736\cdot\;10^{-8}$\\
	$c_9$&$=$&$\pi^2(40 320 - 282 240 \pi^2 + 663 264 \pi^4 - 439 144 \pi^6 + 37975 \pi^8)/ 3329438515200$&$ \approx$&$ 3.9589\cdot \; 10^{-9}$\\
	$c_{10}$&$=$&$ \pi^2 (3465 - 6930 \pi^2 + 6006 \pi ^4 - 1804 \pi^6 + 128 \pi ^8)/15450675609600$&$ \approx $&$1.8749\;\cdot\;10^{-10}$
\end{tabular}
$$
Continuing with Mathematica, we get for  the quotients:
\smaller[4]
\begin{align*}
	\frac{c_{11}}{c_{10}} & = \frac{3 \left(1774080 - 24393600 \pi ^2 + 129773952 \pi ^4 - 258523760 \pi ^6 + 160227716 \pi ^8 - 13712895 \pi ^{10}\right)}{10240
		\left(3465-6930 \pi ^2+6006 \pi ^4-1804 \pi ^6+128 \pi ^8\right)} \,, \\
	\frac{c_{12}}{c_{11}} & = \frac{256 \left(2837835-10405395 \pi ^2 + 18432414 \pi ^4 - 13774761 \pi ^6 + 3835832 \pi ^8 - 265344 \pi ^{10}\right)}{27027
		\left(1774080 - 24393600 \pi ^2 + 129773952 \pi ^4 - 258523760 \pi ^6 + 160227716 \pi ^8 - 13712895 \pi ^{10}\right)} \,,\\
	\frac{c_{13}}{c_{12}} & = \frac{7 \left(2075673600-49470220800 \pi ^2+497175719040 \pi ^4-2161554183360 \pi ^6+3895229400920 \pi ^8-2314322017956 \pi ^{10}+196697984175
		\pi ^{12}\right)}{40960 \left(2837835 - 10405395 \pi ^2 + 18432414 \pi ^4 - 13774761 \pi ^6 + 3835832 \pi ^8 - 265344 \pi ^{10}\right)} \,.
\end{align*}
\normalsize
This leads to
\begin{equation*}
	c_{11}  \approx 8.1149\,\cdot\, 10^{-12} \,, \quad
	c_{12}  \approx 3.23369\,\cdot\, 10^{-13} \,,\quad
	c_{13}  \approx 1.1938 \,\cdot\, 10^{-14} \,.
\end{equation*}

Although not important for our applications, it is nice to see that this leads to the following guess:
\begin{quotation}
	For any $n$, there is a polynomial  $P_n$ of degree $[\frac {n+1} 2]$  with rational coefficients  such that
	$$
	c_n= P_n(\pi^2)\,,
	$$
	where $ [x]$ denotes the largest integer satisfying $[x]\leq x$.
\end{quotation}
This is related to formulas concerning multi-zeta functions, which are recognized by Mathematica and, as communicated to us by F. Nicoleau, to the Lerch function \cite[p.\ 32]{MOS} and its derivatives.

%%%%%%%%%%%%%%%%%%%%%%%%%%%%

\medskip

\subsection{The Weyl constant on $\mathbb H_n\times\R^k$ with $k\geq 1$}~\\
We show that the constant in the Weyl formula on $\mathbb H_n\times\R^k$ can be expressed in terms of that on $\mathbb H_n$.

\begin{lemma}\label{wnk}
	For any $n,k\in\N$,
	\begin{equation}
		\label{eq:wnk}
		\mathcal W(\mathbb H_n\times\R^k) = \mathcal W(\mathbb H_n)\  (4\pi)^{-\frac{k}{2}} \frac{\Gamma(n+2)}{\Gamma(\frac{2n+k+4}{2})} \,.
	\end{equation}
\end{lemma}

\begin{proof}
	By the same argument as in the previous subsection we know that $\mathcal W(\mathbb H_n\times\R^k)$ is given by
	\begin{equation*}%\label{sp1}
		\1\big(-\Delta_{\mathbb H_n\times \mathbb R^k}<\lambda\big)(x,y,t,w,x,y,t,w)=\mathcal W(\mathbb H_n\times\R^k)\ \lambda^{\frac{2n+2+k}{2}}\,.
	\end{equation*}
	Since
	\begin{align*}
		& \1\big(-\Delta_{\mathbb H_n\times \mathbb R^k}<\lambda\big)(x,y,t,w,x,y,t,w) \\
		& = \int_\R \frac{d\zeta}{2\pi} \int_{\R^k} \frac{d\tau}{(2\pi)^k} \left( \frac{4|\zeta|}{2\pi} \right)^n \sum_{k \in\N_0^n} \1(4|\zeta|(2(k_1+\ldots+k_n)+n)+ \tau^2<\lambda) \,,
	\end{align*}
	we find
	\begin{align*}\label{sp2}
		\mathcal W(\mathbb H_n\times\R^k) & = (2\pi)^{ - k} |\Sph^{k-1}| \Big(\int_0^1(1-\rho^2)^{n+1} \rho^{k-1} d\rho\Big) \mathcal W(\mathbb H_n) \\
		& = (2\pi)^{ - k} |\Sph^{k-1}| \Big( 2^{-1} \int_0^1(1-\sigma)^{n+1} \sigma^{\frac{k-2}2} d\sigma \Big) \mathcal W(\mathbb H_n) \\
		& = (2\pi)^{ - k} |\Sph^{k-1}| 2^{-1} \frac{\Gamma(n+2)\,\Gamma(\frac k2)}{\Gamma(\frac{2n+4+k}{2})} \ \mathcal W(\mathbb H_n)
		\,.
	\end{align*}
	Here we expressed the beta function integral appearing  on the second line just above in terms of gamma functions. Inserting $|\Sph^{k-1}| = \Gamma(\frac k2)^{-1} 2 \pi^\frac k2$, we arrive at the claimed formula.
\end{proof}

%%%%%%%%%%%%%%%%%%%%%%%%%%%%%%%%%%%%

\section{A first bound on the Faber--Krahn constant}\label{sec:fksob}

\smallskip

\subsection{A bound via the Sobolev constant}~\\
We obtain a bound on the Faber--Krahn constant in terms of the (critical) Sobolev inequality on $\mathbb H_n\times\R^k$. By definition, $C^{\rm Sob}(\mathbb H_n\times\R^k)$ is the largest constant such that for all $u\in S^1_0(\mathbb H_n\times\R^k)$
\begin{align*}
	& \int_{\mathbb H_n\times\R^k} \left( \sum_{j=1}^n ((X_ju)^2 + (Y_ju)^2) + \sum_{i=1}^k (W_iu)^2\right) dx\,dy\,dz\,dw \\
	& \quad \geq C^{\rm Sob}(\mathbb H_n\times\R^k) \left( \int_{\mathbb H_n\times\R^k} |u|^\frac{2(2n+2+k)}{2n+k} \,dx\,dy\,dz\,dw \right)^{\frac{2n+k}{2n+2+k}}.
\end{align*}

\begin{lemma}\label{fksob}
	$
	C^{\rm FK}(\mathbb H_n\times\R^k) \geq C^{\rm Sob}(\mathbb H_n\times\R^k).
	$
\end{lemma}

\begin{proof}
	If $\Omega\subset\mathbb H_n\times\R^k$ is open with finite measure and if $u\in S^1_0(\Omega)$, then, by H\"older,
	$$
	\int_\Omega u^2 \,dx\,dy\,dz\,dw \leq |\Omega|^{\frac{2}{2n+2+k}} \left( \int_{\mathbb H_n\times\R^k} |u|^\frac{2(2n+2+k)}{2n+k} \,dx\,dy\,dz\,dw \right)^{\frac{2n+k}{2n+2+k}}. 
	$$
	Bounding the right side by Sobolev, we obtain a Faber--Krahn-type inequality with constant $C^{\rm Sob}(\mathbb H_n\times\R^k)$. This implies the claimed bound for the optimal constants.
\end{proof}

An explicit expression for $C^{\rm Sob}(\mathbb H_n)$ was found by Jerison and Lee \cite{JL}; for an alternative proof see also \cite{FrLi}. We have
\begin{equation}\label{eq:jl}
	C^{\rm Sob}(\mathbb H_n) = \frac{4\pi n^2}{(2^{2n} n!)^\frac{1}{n+1}} \,.
\end{equation}

%%%%%%%%%%%%%%%%%%%%%%%%%%%%%%%

\medskip

\subsection{The Pleijel constant $\gamma(\mathbb H_n)$}\label{sec:maincompk0}~\\
Our goal in this subsection is to prove the part of Theorem \ref{maincomp} for $k=0$, that is, we are going to prove that $\gamma(\mathbb H_n)<1$ for $n\geq 4$. To bound $\gamma(\mathbb H_n)$, defined in \eqref{eq:defgamma}, we use Lemma \ref{fksob} and the explicit expression for $C^{\rm Sob}(\mathbb H_n)$ from \eqref{eq:jl} to get
\begin{equation}
	\label{eq:gammatilden}
	\gamma(\mathbb H_n) \leq \left( C^{\rm Sob}(\mathbb H_n) \right)^{-n-1} \mathcal W(\mathbb H_n)^{-1} = \frac{2^n(n+1)!}{n^{2(n+1)}} \frac{1}{c_n} =: \tilde\gamma_n \,,
\end{equation}
where $c_n$ is defined in \eqref{eq:defcn}. Inspired by \cite{HPS} we will consider the quotients $\tilde\gamma_n/\tilde\gamma_{n-1}$.

In view of \eqref{eq:gammatilden}, the part of Theorem \ref{maincomp} for $k=0$ is an immediate consequence of the following assertion.

\begin{proposition}\label{gammatildemono}
	The sequence $n\mapsto \tilde\gamma_n$ is decreasing for $n\geq 1$. Moreover, $\tilde\gamma_n<1$ for $n\geq 4$.
\end{proposition}

Our proof of this proposition relies on numerical computations for $n\leq 13$. Using the values of $c_n$ from the previous section, we get
$$
\begin{tabular}{lll}
	$\tilde\gamma_1$ & $\approx$ & $3.242$\\
	$\tilde\gamma_2$ & $\approx$ & $1.824$\\
	$\tilde\gamma_3$ & $\approx$ & $1.069$\\
	$\tilde\gamma_4$ &$\approx$ & $6.251 \cdot 10^{-1}$ \\ 
	$\tilde\gamma_5$ &$\approx$  & $3.628 \cdot 10^{-1}$ \\ 
	$\tilde\gamma_6$ & $\approx$ & $ 2.088 \cdot 10^{-1}$ \\ 
	$\tilde\gamma_7$ & $\approx$ & $1.195\cdot 10^{-1}$  \\ 
	$\tilde\gamma_8$ & $\approx$ & $ 6.808\cdot 10^{-2}$ \\ 
	$\tilde\gamma_9$ & $\approx$ & $  3.860\cdot10^{-2}$ \\ 
	$\tilde\gamma_{10}$ & $\approx$ & $2.180 \cdot 10^{-2}$  \\
	$\tilde\gamma_{11}$ &$\approx$ & $1.227\cdot 10^{-2}$ \\ 
	$\tilde\gamma_{12}$ &$\approx$ &$6.891\cdot 10^{-3} $ \\ 
	$\tilde\gamma_{13}$ & $\approx$ & $3.859\cdot 10^{-3}$  \end{tabular}
$$
It is also instructive to look at the quotients $\tilde\gamma_n/\tilde\gamma_{n-1}$, for which we get
$$
\begin{tabular}{lll}
	$\tilde\gamma_2/\tilde\gamma_1$ & $\approx$ & $1.777$ \\
	$\tilde\gamma_3/\tilde\gamma_2$ & $\approx$ & $0.5861$ \\
	$\tilde\gamma_4/\tilde\gamma_3$ &$\approx$ & $0.5848$ \\ $\tilde\gamma_5/\tilde\gamma_4$ & $\approx$  & $0.5804$ \\
	$\tilde\gamma_6/\tilde\gamma_5$ &$\approx$   & $0.5757$ \\
	$\tilde\gamma_7/\tilde\gamma_6$&$\approx$ & $0.5721 $\\
	$\tilde\gamma_8/\tilde\gamma_7$&$\approx$ &$0.5697$\\
	$\tilde\gamma_9/\tilde\gamma_8$&$\approx$ &$0.5670$\\ 
	$\tilde\gamma_{10}/\tilde\gamma_9$&$\approx$ & $0.5648$\\ 
	$\tilde\gamma_{11}/\tilde\gamma_{10}$&$\approx$ &0.5630\\
	$\tilde\gamma_{12}/\tilde\gamma_{11}$&$\approx$ &0.5614\\
	$\tilde\gamma_{13}/\tilde\gamma_{12}$&$\approx$ &0.5601
\end{tabular}
$$
These computations also suggest that the sequence $\tilde\gamma_n/\tilde\gamma_{n-1}$ is decreasing and convergent, although this remains unproved.

\begin{proof}
	We are going to show that $\frac{\tilde\gamma_n}{\tilde\gamma_{n-1}}<1$ for $n\geq 13$. Since the same holds for $n\leq 12$ by the above numerical computations, we will obtain the claimed monotonicity. This monotonicity, together with the numerical fact that $\tilde\gamma_4<1$, implies the corresponding inequality for all $n\geq 4$.
	
	To show that $\tilde\gamma_n/\tilde\gamma_{n-1} <1$ for $n\geq 13$, we note that
\begin{equation} \label{eq:b}
	\frac{\tilde\gamma_n}{\tilde\gamma_{n-1}} = \frac{2(n+1)}{n^2} \, (1-1/n)^{2n} \, \frac{c_{n-1}}{c_n} \,.
\end{equation}
We write the definition \eqref{eq:defcn} of $c_n$ in the form
$$
c_n=  \sum_{m\in\N_0} \frac{ (m+n-1)! }{(n-1)! m!}  \frac{1}{(2m+n)^{n+1}}\,.
$$
In view of \eqref{eq:b} we are mainly interested in a lower bound on the quotient $c_n/c_{n-1}$. Observing that
\begin{align*}
	\frac{ (m+n-1)! }{(n-1)! m!}  \frac{1}{(2m+n)^{n+1}} & = \frac{m+n-1}{(n-1)(2m+n)} (1- \frac{1}{2m+n})^n \\
	& \quad \times \frac{ (m+n-2)! }{(n-2)! m!}  \frac{1}{(2m+n-1)^{n}}\,,
\end{align*}
we get
\begin{equation}\label{eq:a}
	\frac{c_n}{c_{n-1}} \geq \frac{1}{n-1} \inf_m \frac{m+n-1}{2m+n} (1- \frac{1}{2m+n})^n \,.
\end{equation}
Hence we have to analyze $\inf_m \theta_n(m)$ with
\begin{equation}\label{eq:theta}
	\theta_n(m):= \frac{m+n-1}{2m+n} (1- \frac{1}{2m+n})^n\,.
\end{equation}

We need to carefully analyze the sequence $m \mapsto \theta_n(m)$. Before doing this, let us provide some heuristics coming from limiting regimes. As $m=0$, we have $\theta_n(0)=\frac{n-1}{n}(1- \frac{1}{n})^n$, which tends to $e^{-1}$ as $n\rightarrow +\infty$. As $m$ tends to $+\infty$, we have $\theta_n(m)\to\frac 12$, but we need uniform lower bounds with respect to $m$ and $n$.

Asymptotically as $n\rightarrow +\infty$, a lower bound  for $\theta_n(m)$ is given by considering the infimum of the function (think of the change of variable $y=m/n$)
$$
(0,+\infty) \ni y \mapsto \frac{1+y}{1+2y} e^{-1/(2y+1)}\,.
$$
After the change of variable $u=y+\frac 12$, we have to analyze
$$
(\frac 12,+\infty) \ni u \mapsto \frac{\frac12 +u}{2u} e^{-1/2u}= (\frac 12 + \frac{1}{4u}) e^{-1/2u}\,.
$$
This function is increasing and its minimum is at $u=\frac 12$ and equals $e^{-1}$. Hence we have
\begin{equation}
	\frac{m+n}{2m+n}e^{- \frac{n}{n+2m}} \geq e^{-1}\,.
\end{equation}
This leads to
\begin{equation}\label{eq:expdecaytildegammanupper}
	\limsup_{n\rightarrow +\infty} \frac{\tilde\gamma_n}{\tilde\gamma_{n-1}} \leq 2 e^{-1} \approx 0.735\;.
\end{equation}
This is closer to the guess.
For the lower bound and using an upper bound for $\limsup_{n\rightarrow +\infty} \theta_n(m)$, we get
$$
\liminf_{n\rightarrow +\infty} \frac{\tilde\gamma_n}{\tilde\gamma_{n-1}} \geq 4 e^{-2} \approx 0.541\;.
$$
This is quite close to the numerics.

After having discussed these heuristics, we will turn to the proof of rigorous bounds. Returning to  $\theta_n(m)$ in \eqref{eq:theta}, we write
$$
\theta_n(m)\geq  \frac{m+n}{2m+n} (1- \frac{1}{2m+n})^n - \frac 1n\,.
$$
Using 
$$
- \log (1-x) \leq x + \int_0^x \frac{t}{1-t} \, dt\,,\qquad x\in [0,1)\,,
$$
we get
$$
x\leq - \log (1-x) \leq x + \frac{1}{1-x} \frac{x^2}{2}
\,,\qquad x\in [0,1)\,.
$$
With $x=\frac{1}{2m+n}$, we obtain
$$
(1- \frac{1}{2m+n})^n= e^{n \log (1- \frac{1}{2m+n})}\geq e^{- \frac{n}{n+2m} } e^{ - \frac{1}{2 (n-1)}} \,.
$$
Coming back to  $\theta_n(m)$ and what we have done for the limsup
$$
\theta_n(m)\geq  e^{-1} e^{- \frac {1}{2(n-1)}} -\frac 1n\,.
$$
So we finally get
$$
c_n \geq \frac{1}{n-1} \, \Big( e^{-1} e^{- \frac 1{2(n-1)}} -\frac 1n\Big) c_{n-1}\,.
$$
Coming back to \eqref{eq:b}, 
we get
\begin{align*}
	\frac{\tilde\gamma_n}{\tilde\gamma_{n-1}}  
	& \leq \frac{2(n+1)(n-1)}{n^2}(1-1/n)^{2n} \Big( e^{-1} e^{- \frac {1}{2(n-1)}} -\frac 1n\Big)^{-1} \\
	& \leq 2 e^{-1} \, \frac{(n+1)(n-1)}{n^2}\Big(  e^{- \frac 1{2(n-1)}} -\frac en\Big)^{-1} \leq 2 e^{-1} \Big(  e^{- \frac {1}{2(n-1)}} -\frac en\Big)^{-1} \,.
\end{align*}
So finally, we have shown that
\begin{equation}\label{eq:quotientgammatilde}
	\frac{\tilde\gamma_n}{\tilde\gamma_{n-1}} \leq 2 e^{-1} \Big(  e^{- \frac {1}{2(n-1)}} -\frac en\Big)^{-1} \,.
\end{equation}
For $n=13$, we have
\begin{equation} \label{eq:9.9}
e^{- \frac 1{2(n-1)}} -\frac en \approx 0.75009 \,,
\end{equation}
and, consequently,
$$
\frac{\tilde\gamma_{13}}{\tilde\gamma_{12}} <1\,.
$$
Looking at the bound \eqref{eq:quotientgammatilde} and its monotonicity with respect to $n$, the bound $\frac{\tilde\gamma_n}{\tilde\gamma_{n-1}}<1$ holds for \emph{any} $n\geq 13$. This completes the proof of Proposition \ref{gammatildemono}.	
\end{proof}

%%%%%%%%%%%%%%%%%%%%%%%%%%%%%%%%%%%%

\medskip

\subsection{A lower bound on the Sobolev constant on $\mathbb H_n\times\R^k$ for $k\geq 1$}~\\
In this subsection we prove a lower bound on the Sobolev constant $C^{\rm Sob}(\mathbb H_n\times\R^k)$ in terms of the Sobolev constant $C^{\rm Sob}(\mathbb H_n)$ and the constant appearing  in a certain Sobolev interpolation inequality on $\R^k$. Assume $2\leq q<\infty$ if $k\leq 2$ and $2\leq q\leq\frac{2k}{k-2}$ if $k>2$, and denote by $C^{\rm GN}_q(\R^k)$ the largest possible constant in the inequality, valid for $u\in H^1(\R^k)$,
\begin{subequations}\label{eq:interineq}\begin{equation}
		\left( \int_{\R^k} |\nabla u|^2\,dw \right)^\theta \left( \int_{\R^k} | u|^2\,dw \right)^{1-\theta}
		\geq C^{\rm GN}_q(\R^k) \left( \int_{\R^k} |u|^q \,dw \right)^\frac2q \,,
	\end{equation}
	where
	\begin{equation}\label{eq:interineqb}
		\theta = k \left( \frac12 - \frac1q \right).
	\end{equation}
	(The value of $\theta$ is determined by scaling.)
\end{subequations}
For $k=1$ the explicit value of the constant $C^{\rm GN}_q(\R)$ is known from a work of Nagy \cite{Na}. For $k\geq 2$ its explicit value is not known, but we will still be able to derive some results in Subsection \ref{sec:maincompk}.

\begin{proposition}\label{lifting}
	For all $n,k\in\N$, setting $Q=2n+2$ and $q=\frac{2(Q+k)}{Q+k-2}$,
	$$
	C^{\rm Sob}(\mathbb H_n\times\R^k) \geq C^{\rm GN}_{q}(\R^k) \ ( C^{\rm Sob}(\mathbb H_n))^{\frac{Q}{Q+k}} \ \frac{Q+k}{Q^\frac{Q}{Q+k} \, k^\frac{k}{Q+k}} \,.
	$$
\end{proposition}

The argument that follows is inspired by the Laptev--Weidl method of lifting in dimension \cite{LaWe} and similar to one used in \cite{FGL}.

\begin{proof}
We begin by applying the inequality
$$
ab \leq \frac1p a^p + \frac1{p'} b^{p'}
$$
with $p=\frac{Q+k}{k}$, so $p'=\frac{Q+k}{Q}$, and $a=\mu^{-\frac{k}{Q+k}} g^\frac{Q+k-2}{Q+k}$, $b= \mu^\frac{k}{Q+k} g^\frac{2}{Q+k}$. We get
$$
g \leq \frac{k}{Q+k} \mu^{-1} g^\frac{Q+k-2}{k} + \frac{Q}{Q+k} \mu^\frac kQ g^\frac{2}{Q} \,.
$$
We apply this with
$$
g = \int_{\R^k} |u(\zeta,w)|^q\,dw
\qquad\text{and}\qquad
\mu = c\left( \int_{\R^k} |u(\zeta,w)|^2\,dw \right)^{\frac{Q}{k}} \,,
$$
where $\zeta\in\mathbb H_n$ is fixed and $c>0$ is a parameter. We obtain
\begin{align*}
	\int_{\R^k} |u(\zeta,w)|^q\,dw & \leq c^{-1} \frac{k}{Q+k} \frac{\left( \int_{\R^k} |u(\zeta,w)|^q\,dw \right)^\frac{Q+k-2}{k}}{\left( \int_{\R^k} |u(\zeta,w)|^2\,dw \right)^\frac{Q}{k}} \\
	& \quad + c^\frac{k}{Q} \frac{Q}{Q+k} \left( \int_{\R^k} |u(\zeta,w)|^q\,dw \right)^\frac2Q \int_{\R^k} |u(\zeta,w)|^2\,dw \\
	& \leq c^{-1} \frac{k}{Q+k} (C^{\rm GN}_{q}(\R^k))^{-\frac{Q+k}{k}} \int_{\R^k} |\nabla_w u(\zeta,w)|^2\,dw \\
	& \quad
	+ c^\frac{k}{Q} \frac{Q}{Q+k} \left( \int_{\R^k} |u(\zeta,w)|^q\,dw \right)^\frac2Q \int_{\R^k} |u(\zeta,w)|^2\,dw \,.
\end{align*}
Here we used the Sobolev interpolation inequality on $\R^k$.

We integrate this inequality with respect to $\zeta$ and obtain
\begin{align*}
	\|u\|_q^q & \leq c^{-1} \frac{k}{Q+k} (C^{\rm GN}_{q}(\R^k))^{-\frac{Q+k}{k}} \| \nabla_w u\|_2^2 \\
	& \quad + c^\frac{k}{Q} \frac{Q}{Q+k}  \int_{\mathbb H_n} \left( \int_{\R^k} |u(\zeta,w)|^q\,dw \right)^\frac2Q \int_{\R^k} |u(\zeta,w)|^2\,dw\,d\zeta \,.
\end{align*}
Now by Fubini and H\"older
\begin{align*}
	& \int_{\mathbb H_n} \left( \int_{\R^k} |u(\zeta,w)|^q\,dw \right)^\frac2Q \left(\int_{\R^k} |u(\zeta,w)|^2\,dw\right) \,d\zeta \\
	& \quad \leq \| u \|_q^\frac{2q}{Q} \int_{\R^k} \left( \int_{\mathbb H_n} |u(\zeta,w)|^\frac{2Q}{Q-2} \,d\zeta \right)^\frac{Q-2}{Q}\,dw \,.
\end{align*}
Applying the Sobolev inequality on $\mathbb H_n$ on the right side and inserting the inequality in the above bound, we get
\begin{align*}
	\|u\|_q^q & \leq c^{-1} \frac{k}{Q+k} (C^{\rm GN}_{q}(\R^k))^{-\frac{Q+k}{k}} \| \nabla_w u\|_2^2 \\ 
	& \quad + c^\frac{k}{Q} \frac{Q}{Q+k} (C^{\rm Sob}(\mathbb H_n))^{-1} \| u \|_q^\frac{2q}{Q} \int_{\R^k} \sum_{j=1}^n \int_{\mathbb H_n} (|X_ju|^2 + |Y_ju|^2)\,d\zeta \,dw \,.
\end{align*}

We finally choose $c$ such that
$$
c^{-1} \frac{k}{Q+k} (C^{\rm GN}_{q}(\R^k))^{-\frac{Q+k}{k}} = c^\frac{k}{Q} \frac{Q}{Q+k} (C^{\rm Sob}(\mathbb H_n))^{-1} \| u \|_q^\frac{2q}{Q} \,,
$$
that is, we choose
$$
c = \left( \frac{k}{Q} \right)^\frac{Q}{Q+k} (C^{\rm GN}_{q}(\R^k))^{-\frac{Q}{k}} ( C^{\rm Sob}(\mathbb H_n))^{\frac{Q}{Q+k}} \| u \|_q^{-\frac{2q}{Q+k}} \,.
$$
In this way we get
$$
\begin{array}{l}
\| u \|_q^q \leq (C^{\rm GN}_{q}(\R^k))^{-1} ( C^{\rm Sob}(\mathbb H_n))^{-\frac{Q}{Q+k}} \| u \|_q^{\frac{2q}{Q+k}}
\frac{Q^\frac{Q}{Q+k} k^\frac{k}{Q+k}}{Q+k}\\
\qquad \qquad \left( \| \nabla_w u\|_2^2 
+ \int_{\R^k} \sum_{j=1}^n \int_{\mathbb H_n} (|X_ju|^2 + |Y_ju|^2)\,d\zeta \,dw \right) .
\end{array}
$$
This proves the assertion.
\end{proof}

From Proposition \ref{lifting} we obtain the following bound on the Pleijel constant. We recall that $\tilde\gamma_n$ was defined in \eqref{eq:gammatilden}.

\begin{corollary}\label{pleijellifting}
	For all $n,k\in\N$, setting $Q= 2n+2$ and $q=\frac{2(Q+k)}{Q+k-2}$,
	$$
	\gamma(\mathbb H_n\times\R^k) \leq 
	( C^{\rm GN}_{q}(\R^k) )^{-\frac{Q+k}{2}} \, \frac{Q^\frac{Q}{2} \, k^\frac{k}{2}}{(Q+k)^{\frac{Q+k}2}} \,
	(4\pi)^{\frac{k}{2}} \, \frac{\Gamma(\frac{Q+k+2}{2})}{\Gamma(\frac{Q+2}2)} \, \tilde\gamma_n \,.
	$$
\end{corollary}

\begin{proof}
	According to Lemma \ref{fksob}, Proposition \ref{lifting} and \eqref{eq:wnk} we have
	\begin{align*}
		\gamma(\mathbb H_n\times\R^k) & \leq \left( C^{\rm Sob}(\mathbb H_n\times\R^k) \right)^{-\frac{Q+k}{2}} \left( \mathcal W(\mathbb H_n \times\R^k) \right)^{-1} \\
		& \leq ( C^{\rm GN}_{q}(\R^k) )^{-\frac{Q+k}{2}} \ ( C^{\rm Sob}(\mathbb H_n))^{-\frac{Q}{2}} \, \frac{Q^\frac{Q}{2} \, k^\frac{k}{2}}{(Q+k)^{\frac{Q+k}2}} \,
		\mathcal W(\mathbb H_n)^{-1}  (4\pi)^{\frac{k}{2}} \frac{\Gamma(\frac{Q+k+2}{2})}{\Gamma(\frac{Q+2}2)} \,.
	\end{align*}
	This is the claimed inequality.
\end{proof}

%%%%%%%%%%%%%%%%%%%%%%%%%%%%%%%%%

\medskip

\subsection{The Pleijel constant $\gamma(\mathbb H_n\times\R^k)$ for $k\geq 1$} \label{sec:maincompk}~\\
Our goal in the present subsection is to prove the following result.

\begin{proposition}\label{largedimension}
	If $n\geq 3$ and $k\geq 1$ then $\gamma(\mathbb H_n\times\R^k)<1$.
\end{proposition}

Our proof will rely on the bound on $\gamma(\mathbb H_n\times\R^k)$ from Corollary \ref{pleijellifting}. In order to exploit this corollary we need a good lower bound on the constant $C^{\rm GN}_q(\R^k)$. As we already mentioned, an explicit expression for this constant is known in dimension $k=1$ due to the work of Nagy \cite{Na}. The idea behind the following result is to reduce the case of $\R^k=\R^{k-1}\times\R$ to that of $\R^{k-1}$ and $\R$, and to iterate. This argument is most easily expressed in the parametrization
$$ 
q=\frac{2(\gamma+\frac k2)}{\gamma+\frac k2 -1} \,.
$$
The theorem will require the assumption $\gamma\geq\frac12$, which however will not present a restriction in our application. It appears in the proof since the constant $C_{\tilde q}^{\rm GN}(\R)$ with $\tilde q= \frac{2(\gamma+\frac 12)}{\gamma+\frac 12 -1}$ is only defined for $\gamma\geq\frac12$.

\begin{lemma}
	Let $k\geq 1$ and $\gamma\geq\frac12$. Then
	$$
	C_{\frac{2(\gamma+\frac k2)}{\gamma+\frac k2-1}}^{\rm GN}(\R^k) \geq \left( \left( \frac{k\pi}{2} \right)^\frac k2 \frac{\Gamma(\gamma+\frac12)}{\Gamma(\gamma+\frac{k+1}2)} \, \frac{(\gamma+\frac{k-1}{2})^{\gamma+\frac{k-1}2}}{(\gamma-\frac12)^{\gamma-\frac12}} \right)^\frac{1}{\gamma+\frac k2}
	$$
	with the convention $(\gamma-\frac12)^{\gamma-\frac12}=1$ for $\gamma=\frac12$.
\end{lemma}

As the following proof will show, the inequality in the lemma is an equality if $k=1$.

\begin{proof}
	For $d\in\N$, let us introduce the constant
	$$
	L_{\gamma,d}^{(1)} = \frac{\gamma^\gamma \left(\frac d2\right)^\frac d2}{(\gamma+\frac d2)^{\gamma+\frac d2}} \, \Big(C_\frac{2(\gamma+\frac d2)}{\gamma+\frac d2-1}^{\rm GN}(\R^d)\Big)^{-\gamma-\frac d2} \,.
	$$
	The main step of the proof will be to show that
	\begin{equation}
		\label{eq:ltproduct}
			L_{\gamma,k}^{(1)} \leq \prod_{j=0}^{k-1} L_{\gamma+\frac j2,1}^{(1)} \,.
	\end{equation}
	Accepting this for the moment, let us complete the proof. Inserting Nagy's explicit expression for the constant $C_q^{\rm GN}(\R)$ \cite{Na} (see also \cite[Theorem 2.48]{FLW}), which reads
\begin{equation}\label{eq:Nag}
	C_{q}^{\rm GN}(\R) = \left( \frac{(q+2)^{q+2}}{(q-2)^{q-2}2^{2(q+2)}} \right)^\frac{1}{2q} \left( \sqrt\pi \, \frac{\Gamma(\frac{q}{q-2})}{\Gamma(\frac{q}{q-2}+\frac12)} \right)^\frac{q-2}{q}\,,
\end{equation} 
into the definition of $L_{\gamma,1}^{(1)}$, we obtain
	$$
	L_{\gamma,1}^{(1)} = \frac{1}{\sqrt\pi}\, \frac{\Gamma(\gamma+1)}{\Gamma(\gamma+\frac12)}\, \frac{(\gamma-\frac12)^{\gamma-\frac12}}{(\gamma+\frac12)^{\gamma+\frac12}} \,;
	$$
	see also \cite[Corollary 5.4]{FLW}. Inserting this into \eqref{eq:ltproduct} and simplifying we obtain the inequality in the lemma.
	
	It remains to prove \eqref{eq:ltproduct}. According to \cite[Proposition 5.3]{FLW}, the constant $L_{\gamma,d}^{(1)}$ coincides with the so-called one-particle Lieb--Thirring constant. (This observation is essentially due to Lieb and Thirring \cite{LiTh}.) We use now the inequality
	\begin{equation}
		\label{eq:ltproductstep}
		L_{\gamma,k}^{(1)} \leq L_{\gamma,1}^{(1)} \, L_{\gamma+\frac12,k-1}^{(1)}
	\end{equation}
	from \cite{FGL}; see also \cite[Proposition 5.14]{FLW}. This bound can be translated into the claimed lower bound on $C_{q_k}^{\rm GN}(\R^k)$ in terms of the constants $C_{q_1}^{\rm GN}(\R)$ and $C_{q_{k-1}}^{\rm GN}(\R^{k-1})$. It is more convenient, however, to stay in the notation of the Lieb--Thirring constants. Iterating the above inequality, we arrive at \eqref{eq:ltproduct}.
	
	We emphasize that \eqref{eq:ltproductstep} can be translated into a lower bound on $C_{q_k}^{\rm GN}(\R^k)$ in terms of the constants $C_{q_1}^{\rm GN}(\R)$ and $C_{q_{k-1}}^{\rm GN}(\R^{k-1})$, where $q_d=\frac{2(\gamma+\frac d2)}{\gamma+\frac d2-1}$. This bound could also be proved directly using the method of proof of Proposition \ref{lifting}.
\end{proof}

In our application, we need $C^{\rm GN}_q(\R^k)$ with $q=\frac{2(Q+k)}{Q+k-2}$ where $Q=2n+2$. We see that $\frac{2(\gamma+\frac k2)}{\gamma+\frac k2 -1} = q$ if $\gamma=\frac Q2$. Note also that, since $Q\geq 4$, the assumption $\gamma\geq\frac12$ is satisfied. Therefore in our application we have
\begin{equation}
	\label{eq:cgnlower}
	C_q^{\rm GN}(\R^k) \geq \left( \left( \frac{k\pi}{2} \right)^\frac k2 \frac{\Gamma(\frac{Q+1}2)}{\Gamma(\frac{Q+k+1}2)} \, \frac{(\frac{Q+k-1}{2})^{\frac{Q+k-1}2}}{(\frac{Q-1}2)^{\frac{Q-1}2}} \right)^\frac{2}{Q+k}.
\end{equation}

We will deduce our main result from this bound, Corollary \ref{pleijellifting}, the known facts about $\tilde\gamma_n$ and the monotonicity properties in the following lemma.

\begin{lemma}\label{mono}
	For each $Q\geq 4$,
	$$
	k \mapsto 2^{2k} \, \frac{\Gamma(\frac{Q+k+1}2)\,\Gamma(\frac{Q+k+2}2)}{(Q+k-1)^{\frac{Q+k-1}2} \, (Q+k)^\frac{Q+k}2}
	$$
	is decreasing for $k\geq 0$.
\end{lemma}

\begin{proof}
	Denoting
	$$
	\alpha_k := 2^{2k} \, \frac{\Gamma(\frac{Q+k+1}2)\,\Gamma(\frac{Q+k+2}2)}{(Q+k-1)^{\frac{Q+k-1}2} \, (Q+k)^\frac{Q+k}2} \,,
	$$
	we have
	$$
	\frac{\alpha_k}{\alpha_{k-1}} = 2 \, \frac{(Q+k-2)^{\frac{Q+k-2}2}}{(Q+k)^\frac{Q+k-2}2} \,,
	$$
	which we can write as $2(1- x_0^{-1})^{x_0-1}$ with $x_0=\frac{Q+k}{2}$. We will prove momentarily that $x\mapsto(1-x^{-1})^{x-1}$ is decreasing for $x\geq 1$. Since it is equal to $1/2$ at $x=2$ and since $x_0> 2$ (provided $k-1\geq 0\geq 4-Q$), we deduce that $2(1- x_0^{-1})^{x_0-1}< 1$, proving that $\alpha_k/\alpha_{k-1}< 1$.
	
	It remains to prove that $x\mapsto(1-x^{-1})^{x-1}$ is decreasing for $x\geq 1$ or, what is the same, $x\mapsto (x-1)\ln(1-x^{-1})$ is decreasing for $x\geq 1$. This follows since the derivative of the latter function is $\ln(1-x^{-1}) + x^{-1}\leq 0$ for $x>1$.	
\end{proof}

We finally prove the main result of this subsection.

\begin{proof}[Proof of Proposition \ref{largedimension}]
	Inserting the bound \eqref{eq:cgnlower} into Corollary \ref{pleijellifting} we obtain
	\begin{align}\label{eq:gnbound}
		\gamma(\mathbb H_n\times\R^k) & \leq 
		 2^{2k} \, \frac{\Gamma(\frac{Q+k+1}2)\,\Gamma(\frac{Q+k+2}2)}{\Gamma(\frac{Q+1}2)\,\Gamma(\frac {Q+2}2)} \, \frac{(Q-1)^{\frac{Q-1}2} \, Q^\frac Q2}{(Q+k-1)^{\frac{Q+k-1}2} \, (Q+k)^\frac{Q+k}2} \, \tilde\gamma_n \,.
	\end{align}

	First consider $n\geq 4$. According to Lemma \ref{mono}, the right side of \eqref{eq:gnbound} is decreasing with respect to $k$. For $k=0$ it is equal to $\tilde\gamma_n$. Since we know that $\tilde\gamma_n<1$ for $n\geq 4$, we learn that $\gamma(\mathbb H_n\times\R^k)<1$ for all $n\geq 4$ and $k\geq 0$.
	
	Now let $n= 3$. Again by Lemma \ref{mono}, the right side of \eqref{eq:gnbound} is decreasing with respect to $k$. For $k=1$ it is given by
	$$
	4 \, \frac{\Gamma(\frac{11}2)}{\Gamma(\frac{9}2)} \, \frac{7^{\frac{7}2}}{9^\frac{9}2} \, \tilde\gamma_3 = 2 \, \left( \frac 79\right)^{\frac{7}2} \tilde\gamma_3
	\approx 8.872 \cdot 10^{-1} <1 \,.
	$$
	Here we used the value $\tilde\gamma_3 \approx 1.069$ from Subsection \ref{sec:maincompk0}. This proves that $\gamma(\mathbb H_3\times\R^k)<1$ for all $k\geq 1$. This completes the proof.
\end{proof}

%%%%%%%%%%%%%%%%%%%%%%%%%%%%%%%%%%%%
%%%%%%%%%%%%%%%%%%%%%%%%%%%%%%%%%%%%

\section{Continuation by looking at a direct product}\label{s10}
This section is a short aside from the main topic of this part. We will explain that most of the arguments in the previous sections generalize from $\mathbb H_n\times\R^k$ to the case of
$$
G = G_1 \times G_2 \,,
$$
where $G_1$ and $G_2$ are stratified, nilpotent Lie groups.  If $\Delta^{G_j}$ is the sub-Laplacian on $G_j$, then
$$
\Delta^G = \Delta^{G_1} \otimes I + I\otimes \Delta^{G_2}
$$
is the sub-Laplacian on $G$.

\medskip

\subsection{Spectral density}~\\
If $Q_1$ and $Q_2$ denote the homogeneous dimensions of $G_1$ and $G_2$, respectively, then
$$
Q:= Q_1 + Q_2
$$
is the homogeneous dimension of $G$.
We recall that for $G_1$ and $G_2$ we have the analogues of formula \eqref{w2}  
$$
\1(-\Delta^{G_1} <\mu)(\zeta_1,\zeta_1)=\mathcal W(G_1) \, \mu^{Q_1/2}
\qquad\text{for all}\ \zeta_1 \in G_1
$$
and
$$
\1 (-\Delta^{G_2} <\mu)(\zeta_2,\zeta_2)=\mathcal W(G_2) \, \mu^{Q_2/2}
\qquad\text{for all}\ \zeta_2 \in G_2 \,.
\,,
$$
We have the same formula for $G$
\begin{equation}\1 (-\Delta^{G_1\times G_2} <\mu)(\zeta,\zeta)=  \mathcal W(G) \, \mu^{Q/2} \qquad\text{for all}\ \zeta\in G
 \end{equation}
and the same argument as in Proposition \ref{wnk} shows that
\begin{equation}\label{eq:weylconstproduct}
\mathcal W(G) := \mathcal W(G_1) \, \mathcal W(G_2) \, \frac{\Gamma(\frac{Q_1}2 +1) \, \Gamma(\frac{Q_2}2 +1)}{\Gamma (\frac Q2+1)} \,.
\end{equation}

\medskip

\subsection{Sobolev inequality}~\\
We now discuss the Faber--Krahn part of Pleijel's proof. In fact, as in the previous section it will be more convenient to work with the Sobolev inequality rather than the Faber--Krahn inequality. Thus, in this subsection we try to get a good lower bound on the Sobolev constant $C^{\rm Sob}(G)$ in
$$
\iint_{G_1\times G_2}  |\nabla_G u|^2 d\zeta_1\,d\zeta_2 \geq C^{\rm Sob}(G) \left( \iint_{G_1\times G_2} |u|^q \,d\zeta_1\,d\zeta_2 \right)^\frac{2}{q},
$$
where we abbreviate
$$
q= \frac{2(Q_1+Q_2)}{Q_1+Q_2-2}= \frac{2 Q}{Q-2} \,.
$$
For the validity of the inequality we require $Q\geq 3$.

We use the Sobolev inequality on $G_1$ (assuming $Q_1 \geq 3$),
$$
\int_{G_1} |\nabla_{G_1} u|^2 d\zeta_1 \geq C^{\rm Sob}(G_1) \, \left( \int_{G_1} |u|^\frac{2Q_1}{Q_1 - 2}\,d\zeta_1 \right)^\frac{Q_1-2}{Q_1},
$$
as well as the following Sobolev interpolation inequality on $G_2$,
$$
\left( \int_{G_2} |\nabla_{G_2}  u|^2\,d\zeta_2 \right)^{\theta} \left( \int_{G_2} | u|^2\,d \zeta_2 \right)^{1-\theta}
\geq C^{\rm GN}_q(G_2)  \left( \int_{G_2} |u|^{q} \,d\zeta_2 \right)^\frac2 {q},
$$
where the parameters $q$ and $\theta$ are related by
$$
 \frac 1 q  = \frac {1-\theta}{2} + \theta \, \frac{Q_2-2}{2Q_2} \,. 
$$
When $Q_2\geq 3$, this inequality holds in the range $\theta\in[0,1]$, corresponding to $q\in [2,\frac{2Q_2}{Q_2-2}]$, while if $Q_2=1,2$, it holds in the range $\theta\in[0,\frac{Q_2}2)$, corresponding to $q\in[2,\infty)$.

These inequalities are known to hold; see, for example \cite{Fo, Va} for the case of dimension $Q_1,Q_2\geq 3$. We do not know of an explicit reference for the inequality in case $Q_2=1,2$, but it can easily be deduced from existing results in the literature. One way is to proceed as in the proof of Lemma \ref{faberkrahn} by `artificially' adding a factor of $\R^k$ (with its additive group structure) and applying the result in dimension $\geq 3$. Another way is to apply \cite[Corollary 3.6]{BaCoLeSC}, which reduces the validity of the Sobolev inequality to the validity of another inequality. In the present case we can use \cite[Theorem 8.1]{BaCoLeSC} in view of the validity of the isoperimetric inequality on $G_2$; see, for instance, \cite[Proposition 1.17 and the following remarks]{HaKo}.

We will apply the Sobolev interpolation inequality with the parameter
$$
\theta_{12} = \frac{Q_2}Q\,,\qquad \frac 1q = \frac 12 - \frac{1}{Q}\,.
$$
Note that this choice satisfies $\theta_{12}\in[0,1]$ if $Q_2\geq 3$ and $\theta_{12}\in[0,\frac{Q_2}{2})$ if $Q_2=1,2$ (recall that we assume $Q_1\geq 3$), so the Sobolev interpolation inequality is indeed valid.

Proceeding in the same way as in the proof of Proposition \ref{lifting}, we obtain the following bound for the Sobolev inequality on $G$:
\begin{equation}\label{eq:asym}
C^{\rm Sob}(G)\geq  C^{\rm GN}_q(G_2)\,  \Big(C^{\rm Sob} (G_1)\Big)^{Q_1/Q} \frac{Q} {Q_1^{Q_1/Q} Q_2^{Q_2/Q}}\,,
\end{equation}
where we recall that $q$ is determined according to $\frac 1q = \frac 12 - \frac{1}{Q}$.

When both $Q_1\geq 3$ and $Q_2 \geq 3$, we can apply H\"older's inequality to deduce that $C^{\rm GN}_q(G_2)\geq \Big( C^{\rm Sob}(G_1)\Big)^{Q_1/Q}$ and obtain the more symmetric bound
\begin{equation}\label{eq:sym}
C^{\rm Sob}(G)\geq  \Big( C^{\rm Sob}(G_1)\Big)^{Q_1/Q}\, \Big(C^{\rm Sob}(G_2)\Big)^{Q_2/Q}  \frac{Q} {Q_1^{Q_1/Q} Q_2^{Q_2/Q}}\,.
\end{equation}

%%%%%%%%%%%%%%%%%%%%%%%

\medskip

\subsection{Pleijel's Theorem}~\\
The proof of Pleijel's bound in Theorem \ref{pleijelnk} extends with obvious changes to the case of general stratified nilpotent groups. Bounding the Faber--Krahn constant that appears in this bound by the corresponding Sobolev constant as in Lemma \ref{fksob}, we arrive at a Pleijel-type bound with constant
$$
\tilde\gamma (G) := \Big(C^{\rm Sob}(G)\Big)^{-\frac{Q}{2}}\mathcal W(G)^{-1} \,.
$$

In particular, when both $Q_1\geq 3$ and $Q_2\geq 3$, we can insert formula \eqref{eq:weylconstproduct} and inequality \eqref{eq:sym} and arrive at the bound  
\begin{align}\label{eq:Pla}
\tilde\gamma (G) \leq \tilde\gamma (G_1) \, \tilde\gamma (G_2) \ \frac{Q_1^{\frac{Q_1}2}\, Q_2^{\frac{Q_2}2}}{Q^{\frac Q2}} \,
\frac{\Gamma (\frac Q2+1)}{\Gamma(\frac{Q_1}2 +1) \, \Gamma(\frac{Q_2}2 +1)} \,.
\end{align}

Specializing to the case $G_2=\R^k$ we can get

\begin{theorem}
Let $G_1$ a stratified nilpotent group of homogeneous dimension $Q_1\geq 3$. Then there exists $k_0(G_1)$ such that for any $k\geq k_0(G_1)$, Pleijel's Theorem holds  for the Dirichlet realization in any domain $\Omega\subset G_1\times\R^k$ of the canonical sub-Laplacian associated with  the  group $G_1 \times \mathbb R^k$.
\end{theorem}

More generally, assume that $G_1$ a stratified nilpotent group of homogeneous dimension $Q_1\geq 3$ and assume that $G_2^{(k)}$ is a sequence of nilpotent groups such that the homogeneous dimensions $Q_2^{(k)}$ tend to $\infty$ and the corresponding constants $\tilde\gamma(G_2^{(k)})$ tend to $0$. Then there exists $k_0$ such that for $k\geq k_0$ we have
$$
\gamma(G_1 \times G_2^{(k)} ) <1\,.
$$

\begin{proof}
	According to Stirling's approximation, we have
	\begin{equation}\label{eq:Stir}
		\ln \big( x^{-\frac{x+1}2} \Gamma(\frac{x+2}2) \big) = - \frac x2 \ln(2e) - \frac12 \ln \pi + \mathcal O(x^{-1})
		\qquad\text{as}\ x\to\infty \,.
	\end{equation}
	This implies that, as $k\rightarrow +\infty$, 
	$$
	\frac{Q_2^{Q_2/2}}{Q^{Q/2}} \frac{\Gamma (Q/2+1)}{\Gamma(Q_2/2 +1)}  \rightarrow (\frac{1}{2e})^{Q_1/2} \,.
	$$
	The assertion then follows from \eqref{eq:Pla} and the assumption $\tilde\gamma(G_2^{(k)})\to 0$.
\end{proof}

%%%%%%%%%%%%%%%%%%%%%%%%%%%%%%%%%%%%%%%%%
%%%%%%%%%%%%%%%%%%%%%%%%%%%%%%%%%%%%%%%%%

\section{A second bound on the Faber--Krahn constant}\label{sec:isoperimetric}

\medskip

\subsection{Preliminary discussion}~\\
While in the previous two sections we have used the Sobolev constant to get bounds on the Faber--Krahn constant,  we will use in this section the isoperimetric constant for this purpose. In those previous two sections we have faced the difficulty that the optimal Sobolev constant on $\mathbb H_n\times\R^k$ is not known when $k\geq 1$ and our main work consisted in getting bounds on this constant. Similarly, the optimal isoperimetric constant on $\mathbb H_n\times\R^k$ is not known when $k\geq 0$ and we will try to get bounds on it. We emphasize that the isoperimetric constant is even unknown in the case of $\mathbb H_n$ (that is, $k=0$), despite a famous conjecture of Pansu \cite{Pa2} and several efforts to prove it, which we cite below.

It is known that, in a relatively general setting, the validity of an isoperimetric inequality implies the validity of a Faber--Krahn inequality; see \cite{Car,BaCoLeSC,Va}. These works also provide bounds on the Faber--Krahn constant in terms of the isoperimetric constant, but the numerical values in these bounds have not been the main concern in these works and, as far as we know, they are not sufficient to deduce Pleijel's theorem. In contrast to these works, we will use methods that are more specific to the problem at hand and use the particularities of the Heisenberg group and of Euclidean space. (For instance, in Lemma \ref{folland} we use the form of the Green's function on the Heisenberg group, and in Theorem \ref{isoperk} we use the isoperimetric inequality on Euclidean space.)

%%%%%%%%%%%%%%%%%%%%%%%%%%%%%%%%%%%%%%%%%
 
\medskip
 
\subsection{Definitions}~\\
We study the isoperimetric inequality in $\mathbb H_n\times\R^k$ from the point of view of the theory of sets of finite perimeter. For an introduction to this theory in the Euclidean setting we recommend \cite{Ma}. From this theory we recall, for instance, that the perimeter of a measurable set $E\subset\R^k$ is defined as
$$
\per_{\R^k} E := \sup \left\{ \int_E {\rm div \,} \phi \, dw  :\ \phi\in C^1_c(\R^k,\R^k),\, |\phi|\leq 1 \right\}
$$
and that for bounded sets with $C^1$ (or even Lipschitz) boundary this coincides with the standard surface area of the set; see \cite[Examples 12.5 and 12.6]{Ma}. The definition of sets of finite perimeter goes back to Caccioppoli and their theory was developed by De Giorgi.

There is a natural and well-known extension of this theory to the setting of stratified nilpotent groups, but we limit ourselves here to the case of $\mathbb H_n\times\R^k$. In this case (horizontal) perimeter of a measurable set $E\subset\mathbb H_n\times\R^k$ is defined to be
\begin{align*}
	\per_{\mathbb H_n\times\R^k}( E):= & \sup\left\{ \int_E \left( \sum_{j=1}^n (X_j\phi+Y_j\phi) + \sum_{i=1}^k W_i\phi \right) dx\,dy\,dz\,dw :\ \right. \\ 
	& \qquad\qquad\qquad\qquad \left. {\phantom {\int_E}} \phi\in C^1_c(\mathbb H_n\times\R^k,\R^{2n+k}),\, |\phi|\leq 1 \right\}.
\end{align*}
We denote by $I(\mathbb H_n\times\R^k)$ the largest constant such that for every measurable set $E\subset\mathbb H_n\times\R^k$ of finite measure one has\footnote{More generally, for a nilpotent stratified group $G$ of homogeneous dimension $Q$ the isoperimetric inequality reads
$$
{\rm Per}_G (E) \geq I(G)\, |E|^{\frac{Q-1}{Q}}\,.
$$
}
$$
\per_{\mathbb H_n\times\R^k} (E) \geq I(\mathbb H_n\times\R^k) \, |E|^\frac{2n+1+k}{2n+2+k} \,.
$$
For background and further details we refer to \cite{GaNh}, as well as the textbook \cite{CaDaPaTy} treating the model case of $\mathbb H_1$.

%%%%%%%%%%%%%%%%%%%%%%%%

\medskip

\subsection{Symmetrization}~\\
In this subsection we show that a bound on the isoperimetric constant implies a bound on the Faber--Krahn constant. Arguments in this spirit are known in the Riemannian context and appear, for instance in \cite[Theorem I.1.5]{BerMe}.  As usual, $j_{\nu,1}$ denotes the first positive zero of the Bessel function $J_\nu$. For the definition and properties of Bessel functions, we refer to \cite{AbSt}.

\begin{proposition}\label{symmetrization}
	For $n\geq 0$ and $k\geq 0$, we have
	$$
	C^{\rm FK}(\mathbb H_n\times\R^k) \geq I(\mathbb H_n\times\R^k)^2 \, (2n+2+k)^{-2} j_{\frac{2n+k}{2},1}^2 \,.
	$$
\end{proposition}

The proof shows that equality holds in the Faber--Krahn inequality with the value given by the right side, provided that every superlevel set of $u$ attains equality in the isoperimetric inequality. This happens in the Euclidean setting where $n=0$. In this case, we have
\begin{equation}
	\label{eq:isofkexplicit}
	I(\R^k) = k^\frac{k-1}{k} |\Sph^{k-1}|^\frac1k 
	\qquad\text{and}\qquad
	C^{\rm FK}(\R^k)= k^{2\frac{k-1}{k}} |\Sph^{k-1}|^\frac2k k^{-2} j_{\frac{k-2}{2},1}^2 = \omega_k^{\frac2k} j_{\frac{k-2}{2},1}^2\,,
\end{equation}
where $\omega_k$ is the measure of the unit ball in $\R^k$; see, for instance, \cite[Theorem 14.1]{Ma} or \cite{Ta} for the first equality and \cite[Theorem 2.54]{FLW} or \cite[Theorem 5.1.2]{LMP} for the second one.

Incidentally, it follows from \eqref{eq:isofkexplicit} that the inequality in Proposition \ref{symmetrization} can be stated as
\begin{equation}
	\label{eq:symmetrization}
	C^{\rm FK}(\mathbb H_n\times\R^k)\, I(\mathbb H_n\times\R^k)^{-2} \geq
	C^{\rm FK}(\R^{2n+2+k}) \, I(\mathbb R^{2n+2+k})^{-2}\,.
\end{equation}

\begin{proof}
	We abbreviate $D := 2n+2+k$. 
	Our goal is to prove inequality \eqref{eq:fk} for a given open set $\Omega\subset\mathbb H_n\times\R^k$ and a given function $u\in C^1_c(\Omega)$. (It will then extend to $S^1_0(\Omega)$ by density.) By truncation properties of Sobolev functions we may assume that $u\geq 0$. We define $u_*$ to be a nonincreasing function on $(0,\infty)$ such that
	$$
	|\{ u >\tau \}| = \int_{\{u_*(r)>\tau\}} r^{D-1}\,dr
	\qquad\text{for all}\ \tau>0 \,.
	$$
	(We could multiply the right side by a constant, for instance $|\Sph^{D-1}|$, to make it look more like $\R^D$, but this does not change the final outcome.) It follows from the layer cake formula that, for any $q>0$,
	\begin{equation}
		\label{eq:layercake}
		\int_{\mathbb H_n\times\R^k} u^q\,dx\,dy\,dz\,dw = \int_0^\infty u_*(r)^q r^{D-1}\,dr \,.
	\end{equation}
	
	We are going to show that, for any $p\geq 1$,
	\begin{equation}
		\label{eq:rearrangegrad}
		\int_{\mathbb H_n\times\R^k} |\nabla_{\mathbb H_n} u|^p\,dx\,dy\,dz\,dw 
		\geq I_{n,k}^p D^{-p\frac{D-1}{D}} \int_0^\infty (-u_*'(r))^p\,r^{D-1}\,dr \,,
	\end{equation}
	where, for short, 
	$$
	I_{n,k}:=I(\mathbb H_n\times\R^k) \,.
	$$
	To prove this, we note that the coarea formula implies that
	\begin{equation}
		\label{eq:coarea}
			|\{ u>\tau\}| = \int_\tau^\infty \left( \int |\nabla_{\mathbb H_n\times\R^k} u|^{-1} \,d|\nabla_{\mathbb H_n\times\R^k}\1_{\{u>t\}}| \right) dt \,.
	\end{equation}
	Let us explain the notation used here. First, we abbreviated
	$$
	|\nabla_{\mathbb H_n\times\R^k} u| = \sqrt{ \sum_{j=1}^n ((X_j u)^2 + (Y_j u)^2) + \sum_{i=1}^k (W_i u)^2 } \,.
	$$
	Next, in the proof of the coarea formula one shows that the sets $\{ u> t\}$ have finite perimeter for almost every $t$, and $|\nabla_{\mathbb H_n\times\R^k}\1_{\{u>t\}}|$ denotes the corresponding perimeter measure. To recall the definition of the latter, let us note that, by Riesz's theorem, if $E$ is a set of finite perimeter, then there is an $\R^{2n+k}$-valued Radon measure $\mu_E$ on $\mathbb H_n\times\R^k$ such that
	$$
	\int_E \left( \sum_{j=1}^n (X_j\phi+Y_j\phi) + \sum_{i=1}^k W_i\phi \right) dx\,dy\,dz\,dw = - \int_E \phi\cdot d\mu_E
	$$
	for all $\phi\in C^1_c(\mathbb H_n\times\R^k,\R^{2n+k})$. The perimeter measure $|\nabla_{\mathbb H_n\times\R^k} \1_E|:= |\mu_E|$ is the corresponding total variation measure. Note that, in particular, we have 
	$$
	\per_{\mathbb H_n\times\R^k} E = \int d|\nabla_{\mathbb H_n\times\R^k}\1_E|\,.
	$$  
	(Whenever $\{u=t\}$ is a sufficiently smooth hypersurface the perimeter measure $|\nabla_{\mathbb H_n\times\R^k}\1_{\{u>t\}}|$ is concentrated on $\{u=t\}$.)
	
	For a textbook proof of the coarea formula in the Euclidean case we refer to \cite[Theorem 13.1]{Ma}. The corresponding formula in the case of stratified nilpotent groups appears in \cite[Theorem 5.2]{GaNh} and \cite[Corollary 2.3.5]{FrSeSC}.
	
	It follows from \eqref{eq:coarea} that
	$$
	\frac{d}{d\tau}|\{ u>\tau\}| = - \int |\nabla_{\mathbb H_n\times\R^k} u|^{-1} \,d|\nabla_{\mathbb H_n\times\R^k}\1_{\{u>\tau\}}| \,.
	$$
	
	By the assumed isoperimetric inequality and H\"older's inequality with $\frac1p + \frac1{p'}=1$, (in fact, only $p=2$ will be relevant), we have
	\begin{align*}
		I_{n,k} |\{ u>\tau\}|^\frac{D-1}{D} & \leq \per_{\mathbb H_n\times\R^k} \{u>\tau\} 
		= \int d|\nabla_{\mathbb H_n\times\R^k}\1_{\{u>\tau\}}| \\
		& \leq \left( \int |\nabla_{\mathbb H_n\times\R^k} u|^{-1} d|\nabla_{\mathbb H_n\times\R^k}\1_{\{u>\tau\}}| \right)^\frac1{p'} \\
		& \quad \times \left( \int |\nabla_{\mathbb H_n\times\R^k} u|^{p-1} d|\nabla_{\mathbb H_n\times\R^k}\1_{\{u>\tau\}}| \right)^\frac 1p.
	\end{align*}
	The previous two relations combined give
	$$
	\int |\nabla_{\mathbb H_n} u|^{p-1} d|\nabla_{\mathbb H_n}\1_{\{u>\tau\}}| \geq I_{n,k}^p |\{ u>\tau\}|^{p\frac{Q-1}{Q}} \left( - \frac{d}{d\tau} |\{ u>\tau\}| \right)^{-p+1}.
	$$
	By the coarea formula again, we deduce that
	\begin{align*}
		\int_{\mathbb H_n\times\R^k} |\nabla_{\mathbb H_n} u|^p \,dx\,dy\,dz\,dw & = \int_0^\infty \int |\nabla_{\mathbb H_n\times\R^k} u|^{p-1} \,d|\nabla_{\mathbb H_n\times\R^k}\1_{\{u>\tau\}}| \,d\tau \\
		& \geq I_{n,k}^p \int_0^\infty  |\{ u>\tau\}|^{p\frac{D-1}{D}} \left( - \frac{d}{d\tau} |\{ u>\tau\}| \right)^{-p+1}\,d\tau \,.
	\end{align*}
	Note that the right side depends on $u$ only through the function $\tau\mapsto|\{ u>\tau\}|$. We express it through the function $u_*$.
	
	For the sake of simplicity we assume that $u_*$ is strictly monotone. The general case can be treated as well, but the idea becomes clearer in this special case. The general can be handled along the lines of \cite[Theorem 15.29]{Leo}. Under the strict monotonicity assumption we can then define a function $[0,\|u\|_\infty]\ni\tau \mapsto R_\tau\in (0,\infty)$ by
	$$
	u_*(R_\tau) = \tau \,.
	$$
	It follows that
	$$
	|\{ u>\tau \}| = \int_{\{u_*(r)>\tau\}} r^{D-1}\,dr = D^{-1} R_\tau^D \,.
	$$
	Moreover, one can show that, if $u$ is weakly differentiable, then $u_*$ is absolutely continuous (an argument of this type in the Euclidean case can be found in \cite[Theorem~15.23]{Leo}) and we have
	$$
	u_*'(R_\tau) \dot R_\tau = 1
	$$
	where the dot denotes differentiation with respect to $\tau$. It follows that
	\begin{align*}
		\int_0^\infty  |\{ u>\tau\}|^{p\frac{Q-1}{Q}} \left( - \frac{d}{d\tau} |\{ u>\tau\}| \right)^{-p+1}\,d\tau
		& = D^{-p\frac{D-1}{D}} \int_0^\infty  R_\tau^{D-1} \left( - u_*'(R_\tau) \right)^p \dot R_\tau \,d\tau \\
		& = D^{-p\frac{D-1}{D}} \int_0^\infty r^{D-1} (-u_*'(r))^p\,dr \,.
	\end{align*}
	Thus, we have completed the proof of \eqref{eq:rearrangegrad}.
	
	At this point we specialize to the case $p=2$ and recall the one-dimensional inequality
	$$
	\int_0^R U'(r)^2 r^{D-1}\,dr \geq j_{\frac{D-2}{2},1}^2 R^{-2} \int_0^R U(r)^2 r^{D-1}\,dr \,,
	$$
	valid for all $R>0$ and all absolutely continuous functions $U$ on $(0,R)$ with $U(R)=0$. This follows from the Faber--Krahn inequality in $\R^D$ when restricted to radial functions; see, e.g., \cite[Theorem 2.54]{FLW}. We apply the above one-dimensional inequality to $U=u_*$ and $R=R_0$. If we insert
	$$
	|\Omega| = |\{ u>0\}| = D^{-1} R_0^D \,,
	$$
	we obtain the claimed inequality.
\end{proof}

%%%%%%%%%%%%%%%%%%%%%%%%%%%%%

\medskip

\subsection{Pansu's conjecture and its consequences}~\\
The validity of the isoperimetric inequality in $\mathbb H=\mathbb H_1$ (that is, the fact that $I(\mathbb H)>0$) is due to Pansu \cite{Pa}. In \cite{Pa2} Pansu made a conjecture about the set that realizes the optimal constant $I(\mathbb H)$. In the references given below this conjecture is generalized to $\mathbb H_n$ with $n\geq 2$ and, after a computation, we find
\begin{equation}
	\label{eq:pansusconj}
	I(\mathbb H_n) = \frac{2n}{2n+1} \, \frac{(2n+2)^{\frac{2n+1}{2n+2}} \, \Gamma(\tfrac{2n+3}{2})^\frac 1{2n+2} \pi^\frac{2n+1}{2(2n+2)} 2^\frac 1{n+1}}{\Gamma(n+1)^\frac 1{n+1}} 
	\qquad\text{(Pansu's conjecture)}\,;
\end{equation}
see, e.g., \cite{DGN2}. (The extra factor of $2^\frac1{n+1}$ compared to their formula comes from the fact that they use $\partial_{x_j} +\frac12 y_j\partial_z$ etc., while we use $\partial_{x_j} + 2y_j\partial_z$ etc.)

Pansu's conjecture is generally believed to be true and has been verified under certain additional assumptions \cite{RR,DGN2,RR2,Ri,Mon,MoRi}; see also \cite{LeRi,LeMa,ChChHwYa,ChHwMaYa} for a sample of contributions to this problem, as well as the textbook \cite{CaDaPaTy} for an introduction to the isoperimetric problem on the Heisenberg group.

We recall that in Subsection \ref{sec:maincompk0} we have proved Pleijel's theorem in $\mathbb H_n$ when $n\geq 4$. Now we will prove it in the remaining dimensions, assuming the validity of Pansu's conjecture.

\begin{corollary}\label{pansupleijelk0}
	Assume that Pansu's conjecture holds for some $n\in\{1,2,3\}$. Then the Pleijel constant $\gamma(\mathbb H_n)$ satisfies for the corresponding value of $n$
	\begin{equation}\label{eq:panpl1}
	\gamma(\mathbb H) \leq 0.406114 \,,
	\qquad
	\gamma(\mathbb H_2) \leq 0.155327 \,,
	\qquad
	\gamma(\mathbb H_3) \leq 0.0641172 \,.
	\end{equation}
\end{corollary}

\begin{proof}
	We abbreviate $Q=2n+2$. The conjectured value for $I(\mathbb H_n)$ given above leads, via Proposition \ref{symmetrization}, to the bound on the Faber--Krahn constant
	\begin{align*}
		C^{\rm FK}(\mathbb H_n) 
		& \geq \frac{Q^{2\,\frac{Q-1}{Q}}(Q-2)^2 \, \Gamma(\tfrac{Q+1}{2})^\frac 2Q \pi^\frac{Q-1}{Q} 4^\frac 2Q}{(Q-1)^2\,\Gamma(\frac Q2)^\frac 4Q} 
		Q^{-2} j_{\frac{Q-2}{2},1}^2 \\
		& =
		\frac{(Q-2)^2 \, \Gamma(\tfrac{Q+1}{2})^\frac 2Q \pi^\frac{Q-1}{Q} 4^\frac 2Q}{Q^{\frac 2Q}\, (Q-1)^2\,\Gamma(\frac Q2)^\frac 4Q} j_{\frac{Q-2}{2},1}^2 \,.
	\end{align*}
	This leads to the bound on the Pleijel constant
	$$
	\gamma(\mathbb H_n) \leq (C^{\rm FK}(\mathbb H_n))^{-\frac Q2} \mathcal W(\mathbb H_n)^{-1} 
	= \frac{Q (Q-1)^Q\,\Gamma(\frac Q2)^2} {(Q-2)^Q \, \Gamma(\tfrac{Q+1}{2}) \pi^{\frac{Q-1}2} 4} \  j_{\frac{Q-2}{2},1}^{-Q} \ \mathcal W(\mathbb H_n)^{-1} \,.
	$$
Using Subsection \ref{sec:weylheis1}, which gives
$$
\mathcal W(\mathbb H) = \frac1{128}\,,\qquad 
\mathcal W(\mathbb H_2) = \frac{1}{48^2\pi}\,,\qquad
\mathcal W(\mathbb H_3) = \frac{1}{2^7 \cdot 768}\,\frac{12 - \pi^2}{\pi^2}\,,
$$
and the values \cite[Table 9.5]{AbSt}
$$ 
j_{1,1}\sim 3.8317\,,\qquad 
j_{2,1}\sim 5.1356\,,\qquad
j_{3,1}\sim 6.3802\,,
$$
we get \eqref{eq:panpl1} and  the corollary.
\end{proof}

%%%%%%%%%%%%%%%%%%%%%%%%%%%%%%%%%%%%%

\medskip

\subsection{A bound on the isoperimetric constant in $\mathbb H_n$}~\\
In this subsection we prove a lower bound on $I(\mathbb H_n)$ which is reasonably good in small dimensions $n=1$ and $n=2$. To state this bound, we need two constants,
\begin{equation}\label{eq:sobrepconst}
	\mathcal C_n := \frac{2^{n-3} n\Gamma(\frac n2)^2}{\pi^{n+1}}
\end{equation}
and, setting $Q=2n+2$,
\begin{align}\label{eq:bathtubconst}
	\mathcal C_n' & := Q^\frac1Q \left( \pi^n \frac{\Gamma(\frac{2Q-1}{2(Q-1)})}{\Gamma(\frac{2Q-1}{2(Q-1)}+\frac12)} \frac{\Gamma(\frac n2+\frac{Q}{4(Q-1)})}{\Gamma(\frac n2+\frac{Q}{4(Q-1)}+\frac12)}
	\frac{\Gamma(1+\frac{Q}{2(Q-1)})}{\Gamma(n+\frac{Q}{2(Q-1)})} \right)^\frac{Q-1}{Q}.
\end{align}
For $n=1$ we can write the latter constant without Gamma functions, see \eqref{eq:bathtubconst1} below. In any dimension it can easily be evaluated numerically.

These two constants appear in our bound on the isoperimetric constant. More precisely, we will prove the following theorem.

\begin{theorem}\label{isoper}
	For any $n\in\N$,
	$$
	I(\mathbb H_n) \geq \mathcal C_n^{-1} (\mathcal C_n')^{-1} \,.
	$$
\end{theorem}	

%%%%%%%%%%%%%%%%%%%%%%%%

We will prove that for any measurable set $E\subset\mathbb H_n$
\begin{equation}
	\label{eq:isopersob}
	\int_E u(\zeta)\,d\zeta \leq \mathcal C_n \mathcal C_n' |E|^\frac1Q \int_{\mathbb H_n} |\nabla_{\mathbb H_n} u(\zeta)|\,d\zeta \,.
\end{equation}
Here and below, we denote variables in $\mathbb H_n$ by $\zeta$ and we write
$$
|\nabla_{\mathbb H_n} u(\zeta)| := \sqrt{ \sum_{j=1}^n ((X_ju)(\zeta)^2 +(Y_ju)(\zeta)^2) } \,.
$$
If we prove this inequality for, say, Lipschitz functions $u$ with compact support, then, by a standard approximation argument, it will extend, with the same constant, to all functions $g$ of bounded variation satisfying $|\{|u|>\lambda\}|<\infty$ for all $\lambda>0$. (The bounded variation condition is understood in the Heisenberg sense, as discussed, for instance, in \cite[Section 5.1]{CaDaPaTy}.) In particular \eqref{eq:isopersob} is valid for the characteristic function of a set of finite (horizontal) perimeter. Since
$$
\int_{\mathbb H_n} |\nabla_{\mathbb H_n} \1_E(\zeta)|\,d\zeta = \per_{\mathbb H_n} E
\qquad\text{and}\qquad
\int_E \1_E(\zeta)\,d\zeta = |E| \,,
$$
we obtain the inequality stated in the theorem. Thus, it remains to prove \eqref{eq:isopersob}.

Our analysis is based on the following well-known representation formula, where we use the notation
$$
\| (x,y,z) \| = ((|x|^2+|y|^2)^2 + z^2)^\frac14 
\qquad\text{for}\ (x,y,z)\in\mathbb H_n \,.
$$

\begin{lemma}\label{folland}
	With $\mathcal C_n$ from \eqref{eq:sobrepconst},
	$$
	u(\zeta) = \mathcal C_n \sum_{j=1}^n \int_{\mathbb H_n} \frac{(X_j\|\cdot\|)(\zeta^{-1} \eta) (X_ju)(\eta) + (Y_j\|\cdot\|)(\zeta^{-1}\eta) (Y_ju)(\eta)}{\|\zeta^{-1} \eta\|^{Q-1}}\,d\eta \,.
	$$
\end{lemma}

\begin{proof}
	As shown by Folland \cite{Fo0},
	$$
	u(\zeta) = - \tilde{\mathcal C}_n \sum_{j=1}^n \int_{\mathbb H_n} \|\zeta^{-1} \eta\|^{2-Q} ((X_j^2+Y_j^2) u)(\eta)\,d\eta
	$$
	with
	$$
	\tilde{\mathcal C}_n = \frac{2^{n-4}\Gamma(\frac n2)^2}{\pi^{n+1}} \,.
	$$
	Here $\zeta^{-1}\eta$ denotes the inverse of $\zeta$ composed with $\eta$ in the sense of the Heisenberg group. The explicit expression for $\tilde{\mathcal C}_n$ can be found, for instance, in \cite[Eq. between (1.2) and (1.3)]{BrFoMo} or \cite[Theorem 1.2]{CoLu}. (Meanwhile, there seems to be a computational error in the constant in \cite[Theorem 5.15]{CaDaPaTy}.) 
	
	Integrating by parts, we find
	\begin{align*}
		& - \int_{\mathbb H_n} \|\zeta^{-1} \eta\|^{2-Q} ((X_j^2+Y_j^2) u)(\eta)\,d\eta \\
		& \quad =  \int_{\mathbb H_n} \left( (X_j \|\zeta^{-1} \eta\|^{2-Q}) (X_j u)(\eta) + ( Y_j \|\zeta^{-1} \eta\|^{2-Q}) (Y_j u)(\eta) \right) d\eta \,.
	\end{align*}
	On the right side, the vector fields $X_j$ and $Y_j$ act with respect to the $\eta$ variable. By left invariance, we find
	$$
	X_j \|\zeta^{-1} \eta\|^{2-Q} = (X_j \|\cdot\|^{2-Q})(\zeta^{-1}\eta) = (2-Q)\|\zeta^{-1}\eta\|^{1-Q} (X_j\|\cdot\|)(\zeta^{-1}\eta)
	$$
	and similarly for $Y_j$. This proves the claimed formula.	
\end{proof}

We deduce from Lemma \ref{folland} that for any measurable set $E\subset\mathbb H_n$,
$$
\int_E u(\zeta)\,d\zeta =  \mathcal C_n \int_{\mathbb H_n} \sum_{j=1}^n \int_E \frac{(X_j\|\cdot\|)(\zeta^{-1} \eta) (X_ju)(\eta) + (Y_j\|\cdot\|)(\zeta^{-1}\eta) (Y_ju)(\eta)}{\|\zeta^{-1} \eta\|^{Q-1}}\,d\zeta \,d\eta \,.
$$
Here we used Fubini's theorem. The arguments given in Lemma \ref{bathtub} below show that the integrals are absolutely convergent, so the interchange of integrals is justified.

Our goal is to bound
$$
\sum_{j=1}^n \int_E \frac{(X_j\|\cdot\|)(\zeta^{-1} \eta) (X_ju)(\eta) + (Y_j\|\cdot\|)(\zeta^{-1}\eta) (Y_ju)(\eta)}{\|\zeta^{-1} \eta\|^{Q-1}}\,d\zeta
$$
from above, for any fixed $\eta$. \\ Abbreviating $a_j:=(X_ju)(\eta)$, $b_j:=(Y_ju)(\eta)$ and  $F:=E^{-1}\eta=\{ \zeta^{-1} \eta:\ \zeta\in E \}$, we need to bound
$$
\sum_{j=1}^n \int_F \frac{a_j (X_j\|\cdot\|)(\xi) + b_j (Y_j\|\cdot\|)(\xi)}{\|\xi\|^{Q-1}}\,d\xi
$$
from above. Note that 
$$
|F|=|E^{-1} \eta|
\qquad\text{and}\qquad
|a|^2+|b|^2 = |\nabla_{\mathbb H_n} g(v)|^2 \,.
$$

\begin{lemma}\label{bathtub}
	For any $a,b\in\R^n$ and any measurable $F\subset\mathbb H_n$,
	$$
	\sum_{j=1}^n \int_F \frac{a_j (X_j\|\cdot\|)(\xi) + b_j (Y_j\|\cdot\|)(\xi)}{\|\xi\|^{Q-1}}\,d\xi
	\leq \mathcal C_n' |F|^\frac{1}{Q} (|a|^2+|b|^2)^\frac12
	$$
	with $\mathcal C_n'$ from \eqref{eq:bathtubconst}.
\end{lemma}

Let us accept this lemma for the moment and use it to complete the proof of \eqref{eq:isopersob}. According to the above argument, we find, for any $\eta\in\mathbb H_n$,
$$
\sum_{j=1}^n \int_E \frac{(X_j\|\cdot\|)(\zeta^{-1} \eta) (X_ju)(\eta) + (Y_j\|\cdot\|)(\zeta^{-1}\eta) (Y_ju)(\eta)}{\|\zeta^{-1} \eta\|^{Q-1}}\,d\zeta
\leq \mathcal C_n' |E|^\frac1Q |\nabla_{\mathbb H_n} u(\eta)| \,.
$$
Consequently, we have shown that
$$
\int_E u(\zeta)\,d\zeta \leq \mathcal C_n \mathcal C_n' |E|^\frac1Q \int_{\mathbb H_n} |\nabla_{\mathbb H_n} u(\eta)| \,d\eta \,.
$$
This is the claimed inequality \eqref{eq:isopersob}. It remains to prove Lemma \ref{bathtub}.

%%%%%%%%%%%%%%%%%%%%%%%

\begin{proof}[Proof of Lemma \ref{bathtub}]
Let us abbreviate
$$	
I[F] := \sum_{j=1}^n \int_F \frac{a_j (X_j\|\cdot\|)(\xi) + b_j (Y_j\|\cdot\|)(\xi)}{\|\xi\|^{Q-1}}\,d\xi \,.
$$
We think of the inequality in the lemma as an optimization problem, where we want to maximize $I[F]$ among all sets $F$ of a given measure. By homogeneity the value of this fixed measure is irrelevant.

We know by the bathtub principle (see, for instance, \cite[Theorem 1.14]{LiLo}) that there is an optimal set $F_*$ for the inequality and that this optimal set is given by
$$
F_* = \Big\{ \xi \in \mathbb H_n :\  \sum_{j=1}^n \frac{a_j (X_j\|\cdot\|)(\xi) + b_j (Y_j\|\cdot\|)(\xi)}{\|\xi\|^{Q-1}} >\kappa \Big\}
$$
for some $\kappa>0$.\\
 In fact, as already mentioned, by homogeneity, the value of $\kappa$ is immaterial and we will set it equal to $1$ in what follows.
Thus, we have, for any set $F\subset\mathbb H_n$,
$$
\frac{I[F]}{|F|^\frac1Q} \leq \frac{I[F_*]}{|F_*|^\frac1Q}\,,
$$
and our task is to compute the right side. More precisely, we want to show that
$$
\frac{I[F_*]}{|F_*|^\frac1Q} = \mathcal C_n' (|a|^2 + |b|^2)^\frac12
$$
with $\mathcal C_n'$ from \eqref{eq:bathtubconst}.

We begin by bringing $I[F]$ in a more explicit form. We compute, writing $\xi=(x,y,t)$,
$$
(X_j\|\cdot\|)(\xi) = \|\xi\|^{-3} \left( (|x|^2+|y|^2) x_j + y_j t \right),
\
(Y_j\|\cdot\|)(\xi) = \|\xi\|^{-3} \left( (|x|^2+|y|^2) y_j - x_j t \right),
$$
so
\begin{align*}
	& \frac{a_j (X_j\|\cdot\|)(\xi) + b_j (Y_j\|\cdot\|)(\xi)}{\|\xi\|^{Q-1}} \\
	& = \frac{(|x|^2+|y|^2)(a\cdot x) + (a\cdot y)t + (|x|^2+|y|^2)(b\cdot y) - (b\cdot x)t}{\|\xi\|^{Q+2}} \,.
\end{align*}
In particular,	
\begin{align*}
	I[F] = \int_F \frac{(|x|^2+|y|^2)(a\cdot x) + (a\cdot y)t + (|x|^2+|y|^2)(b\cdot y) - (b\cdot x)t}{\|\xi\|^{Q+2}}\,d\xi
\end{align*}
and
$$
F_* = \left\{ \xi\in\mathbb H_n:\ \frac{(|x|^2+|y|^2)(a\cdot x) + (a\cdot y)t + (|x|^2+|y|^2)(b\cdot y) - (b\cdot x)t}{\|\xi\|^{Q+2}} > 1 \right\}.
$$

We claim that $|F_*|$ and $I[F_*]$ depend on $a$ and $b$ only through $|a|^2+|b|^2$. Indeed, from the above expression it is obvious that the two quantities do not change if $a$ and $b$ are simultaneously rotated by the same rotation matrix. This implies that the two quantities depend on $a$ and $b$ only through $|a|$, $|b|$ and $a\cdot b$. We also see that the two quantities do not change if we apply a two-dimensional rotation to $(a_j,b_j)$. This implies that the two quantities depend on $a_j$ and $b_j$ only through $a_j^2+b_j^2$. This proves the claim.

As a consequence, we may set $\alpha:=\sqrt{|a|^2+|b|^2}$ and assume that $a=\alpha e_n$ and $b=0$. Thus, we have
$$
F_* = \Big\{ \xi \in\mathbb H_n:\ \alpha \frac{(|x|^2+|y|^2)x_n + y_n t}{\|\xi\|^{Q+2}} >1 \Big\}
$$
and we need to compute
$$
|F_*| 
\qquad\text{and}\qquad
I[F_*] = \int_{F_*} \alpha \frac{(|x|^2+|y|^2)x_n + y_n t}{\|\xi\|^{Q+2}} \,d\xi \,.
$$

To perform these computations, we introduce coordinates
$$
x_j = r \sqrt{\sin\theta} \cos\phi_j \,\omega_j \,,
\qquad
y_j = r \sqrt{\sin\theta} \sin\phi_j \, \omega_j \,,
\qquad
t = r^2 \cos\theta \,,
$$
where $r>0$, $\theta\in(0,\pi)$, $\phi=(\phi_1,\ldots,\phi_n)\in(-\pi,\pi)^n=:T^n$ and $\omega \in\Sph^{n-1}\cap(0,\infty)^n=:\Sigma$. We claim that in this parametrization, the measure is given by
\begin{equation}
	\label{eq:measure}
	dx\,dy\,dt = r^{Q-1}dr\,\sin^{n-1}\theta\, d\theta\,d\mu(\omega)\,d\phi \,,
\end{equation}
where $d\phi$ is standard Lebesgue measure on $T^n$ and where $d\mu(\omega) = \omega_1\cdots\omega_n\,d\omega$ with the standard surface measure $d\omega$ on $\Sph^{n-1}$. More explicitly, if we parametrize $\omega\in\Sigma$ by
\begin{align*}
	\omega_1 & = \sin\theta_{n-1} \sin\theta_{n-2} \cdots \sin\theta_1 \,,\\
	\omega_2 & = \sin\theta_{n-1} \sin\theta_{n-2} \cdots \cos\theta_1 \,,\\
	\ldots \\
	\omega_{n-1} & = \sin\theta_{n-1} \cos\theta_{n-2} \,,\\
	\omega_n & = \cos\theta_{n-1} \,,
\end{align*}
with $0<\theta_j<\frac\pi 2$, then
$$
d\mu(\omega) = \omega_1\cdots\omega_n\,d\omega = \prod_{j=1}^{n-1} \sin^{2j-1} \theta_j\cos\theta_j\,d\theta_j \,.
$$
The latter formula follows immediately from $d\omega = \prod_{j=1}^{n-1} \sin^{j-1}\theta_j\,d\theta_j$.

We point out that when $n=1$, then $\Sigma=\{1\}$ and $\mu$ is trivial in the sense that $\mu(\{1\})=1$.

Let us provide the details for formula \eqref{eq:measure}. First, for each $j$, we introduce polar coordinates
$$
x_j = r_j \cos\phi_j \,,
\qquad
y_j = r_j \sin\phi_j \,,
$$
and note that
$$
dx_j\,dy_j = r_j dr_j \,d\phi_j \,.
$$
Next, we consider $(r_1,\ldots,r_n)$ as an element of $\R^n$ and introduce hyperspherical coordinates in $\R^n$,
$$
r_j = \rho \omega_j
$$
with $\rho>0$ and $\omega=(\omega_1,\ldots,\omega_n)\in\Sph^{n-1}\cap(0,\infty)^n=\Sigma$. We have
$$
dr_1\cdots dr_n = \rho^{n-1}d\rho\,d\omega \,,
$$
so
$$
dx\,dy = \rho^{2n-1} \, d\rho\, \omega_1\cdots\omega_n \,d\omega = \rho^{2n-1}\,d\rho\, d\mu(\omega) \,.
$$
Finally, we set
$$
\rho = r \sqrt{\sin\theta} \,,
\qquad
t= r^2 \cos\theta
$$
and compute easily
$$
d\rho\,dt = r^2 (\sin\theta)^{-\frac12}\,dr\,d\theta \,.
$$
Inserting this into the above formual for $dx\,dy$, we obtain the claimed formula \eqref{eq:measure}.

In these coordinates we have
$$
\frac{(|x|^2+|y|^2)x_n + y_n t}{\|\xi\|^{Q+2}} = \sqrt{\sin\theta}\ \omega_n\, \frac{\sin\theta \cos\phi_n + \cos\theta \sin\phi}{r^{Q-1}} = \frac{\sqrt{\sin\theta}\ \omega_n\, \sin(\theta+\phi_n)}{r^{Q-1}} \,.
$$
In particular, the constraint $\alpha \frac{(|x|^2+|y|^2)x_n + y_n t}{\|\xi\|^{Q+2}}>1$ can be written as
$$
r< \left( \alpha \sqrt{ \sin\theta}\, \omega_n \sin_+(\theta+\phi_n) \right)^\frac{1}{Q-1}.
$$
This allows us to carry out the $r$-integration and to obtain
\begin{align*}
	|F_*| & = \int_{T^n} d\phi \int_\Sigma d\mu(\omega) \int_0^\pi d\theta \,\sin^{n-1}\theta
	\int_0^\infty dr \, r^{Q-1} \1_F \\
	& = \alpha^\frac{Q}{Q-1} \frac1Q \int_{T^n} d\phi \int_\Sigma d\mu(\omega) \int_0^\pi d\theta \,\sin^{n-1}\theta \left( \sqrt{ \sin\theta}\, \omega_n \sin_+(\theta+\phi_n) \right)^\frac{Q}{Q-1}
\end{align*}
and
\begin{align*}
	I[F_*] & = \alpha \int_{T^n} d\phi \int_\Sigma d\mu(\omega) \int_0^\pi d\theta \,\sin^{n-1}\theta \int_0^\infty dr \, r^{Q-1} \1_F \frac{\sqrt{\sin\theta}\, \omega_n \sin(\theta+\phi_n)}{r^{Q-1}} \\
	& = \alpha \int_{T^n} d\phi \int_\Sigma d\mu(\omega) \int_0^\pi d\theta \,\sin^{n-1}\theta \\
	& \qquad\qquad\qquad \times \sqrt{\sin\theta}\, \omega_n \sin(\theta+\phi_n)
	\left( \alpha \sqrt{ \sin\theta}\, \omega_n \sin_+(\theta+\phi_n) \right)^\frac{1}{Q-1} \\
	& = \alpha^\frac{Q}{Q-1} \int_{T^n} d\phi \int_\Sigma d\mu(\omega) \int_0^\pi d\theta \,\sin^{n-1}\theta \left( \sqrt{ \sin\theta}\, \omega_n \sin_+(\theta+\phi_n) \right)^\frac{Q}{Q-1}.
\end{align*}
Thus,
$$
\frac{I[F_*]}{|F_*|^\frac1Q} = \alpha Q^\frac1Q \left( \int_{T^n} d\phi \int_\Sigma d\mu(\omega) \int_0^\pi d\theta \,\sin^{n-1}\theta \left( \sqrt{ \sin\theta}\, \omega_n \sin_+(\theta+\phi_n) \right)^\frac{Q}{Q-1} \right)^\frac{Q-1}{Q}.
$$

It remains to compute the integral on the right side. The integral over $(\phi_1,\ldots,\phi_{n-1})$ can be computed immediately. Moreover, for fixed $\theta$, we can compute the integral over $\phi_n$ by translation invariance. We obtain
\begin{align*}
	& \int_{T^n} d\phi \int_\Sigma d\mu(\omega) \int_0^\pi d\theta \,\sin^{n-1}\theta \left( \sqrt{ \sin\theta}\, \omega_n \sin_+(\theta+\phi_n) \right)^\frac{Q}{Q-1} \\
	& = (2\pi)^{n-1} \int_0^\pi d\phi\, \sin^\frac{Q}{Q-1}\phi \int_0^\pi d\theta \, \sin^{n-1+\frac{Q}{2(Q-1)}}\theta \int_\Sigma d\mu(\omega)\, \omega_n^\frac{Q}{Q-1}
\end{align*}
Using the beta function identity
$$
2\int_0^\frac\pi2 dt\, (\sin^{2z_1-1}t )\,(\cos^{2z_2-1}t) = B(z_1,z_2) = \frac{\Gamma(z_1)\,\Gamma(z_2)}{\Gamma(z_1+z_2)} \,,
$$
we compute
$$
\int_0^\pi d\phi\, (\sin^\frac{Q}{Q-1}\phi) = \sqrt\pi\, \frac{\Gamma(\frac{2Q-1}{2(Q-1)})}{\Gamma(\frac{2Q-1}{2(Q-1)}+\frac12)} \,,
$$
$$
\int_0^\pi d\theta \, (\sin^{n-1+\frac{Q}{2(Q-1)}}\theta) = \sqrt\pi\, \frac{\Gamma(\frac n2+\frac{Q}{4(Q-1)})}{\Gamma(\frac n2+\frac{Q}{4(Q-1)}+\frac12)}
$$
and, when $n\geq 2$,
\begin{align*}
	\int_\Sigma d\mu(\omega)\, \omega_n^\frac{Q}{Q-1} & = \left( \prod_{j=1}^{n-2} \int_0^\frac\pi2 d\theta_j\,\sin^{2j-1} \theta_j \cos\theta_j \right) \int_0^\frac\pi2 d\theta_{n-1} \,\sin^{2n-3} \theta_{n-1} \cos^{1+\frac{Q}{Q-1}} \theta_{n-1} \\
	& = \left( \prod_{j=1}^{n-2} \frac12 \frac{\Gamma(j)}{\Gamma(j+1)} \right) \frac12 \frac{\Gamma(n-1)\,\Gamma(1+\frac{Q}{2(Q-1)})}{\Gamma(n+\frac{Q}{2(Q-1)})} \\
	& = 2^{-n+1} \frac{\Gamma(1+\frac{Q}{2(Q-1)})}{\Gamma(n+\frac{Q}{2(Q-1)})} \,.
\end{align*}
For $n=1$ the same formula remains valid, for in this case the integral is trivially equal to one. Thus, we have shown that
$$
\frac{I[F_*]}{|F_*|^\frac1Q} = \alpha Q^\frac1Q \left( \pi^n \frac{\Gamma(\frac{2Q-1}{2(Q-1)})}{\Gamma(\frac{2Q-1}{2(Q-1)}+\frac12)} \frac{\Gamma(\frac n2+\frac{Q}{4(Q-1)})}{\Gamma(\frac n2+\frac{Q}{4(Q-1)}+\frac12)}
\frac{\Gamma(1+\frac{Q}{2(Q-1)})}{\Gamma(n+\frac{Q}{2(Q-1)})} \right)^\frac{Q-1}{Q}
= \alpha \mathcal C_n'.
$$
This completes the proof of the lemma.
\end{proof}

At this point the proof of Theorem \ref{isoper} is complete. We now bring the lower bound in the theorem in a more explicit form in dimensions $n=1$ and $n=2$.

\begin{corollary}\label{isopercor}
	For $n=1,2$ one has
	$$
	I(\mathbb H)\geq 8 \cdot 3^{-\frac 98} \pi^{\frac14}
	\qquad\text{and}\qquad
	I(\mathbb H_2) \geq \pi^\frac43 6^{-\frac16} \left( \frac{\Gamma(\frac{13}{5})}{\Gamma(\frac{11}{10})} \frac{\Gamma(\frac{9}{5})}{\Gamma(\frac{13}{10})} \right)^{\frac{5}{6}} \,.
	$$
\end{corollary}

Let us discuss the value of the lower bound of $I(\mathbb H)$. Numerically, one has $$8 \cdot 3^{-\frac 98} \pi^{\frac14} \approx 3.09468\,.
$$
This should be compared with Pansu's conjecture \eqref{eq:pansusconj}, which gives the value 
\begin{equation}
	\label{eq:pansun1}
	2^{\frac52} 3^{-\frac34} \pi^{\frac12} \approx 4.39854\,.
\end{equation}
Thus, our value is still quite a bit away from the conjectured sharp value. On the other hand, it improves over Pansu's original value $$(8\pi/3)^\frac14\approx1.7013\,.
 $$
 See \cite{Pa} and also \cite[(7.15)]{CaDaPaTy}. (The discrepancy to the latter formula by a factor of $4^\frac14$ comes from the different normalization of the vector fields used in that reference.)

\begin{proof}
	The claimed formula for $n=1$ follows from Theorem \ref{isoper}, the value $\mathcal C_1=(4\pi)^{-1}$ and
	\begin{equation}\label{eq:bathtubconst1}
		\mathcal C_1' = 2^{-1} 3^\frac 98 \pi^\frac34 \,.
	\end{equation}
	To prove the latter formula, we note that, since $Q=4$,
	$$
	\mathcal C_1' = 4^\frac14 \left( \pi \frac{\Gamma(\frac 76) \, \Gamma(\frac 56)}{\Gamma(\frac{5}{3})\,\Gamma(\frac 43)} \right)^\frac34.
	$$
	By the reflection formula for the gamma function we deduce
	$$
	\frac{\Gamma(\frac 76) \, \Gamma(\frac 56)}{\Gamma(\frac{5}{3})\,\Gamma(\frac 43)}
	= \frac{\frac16\,\Gamma(\frac16)\,\Gamma(\frac56)}{\frac23 \Gamma(\frac23)\,\frac13\,\Gamma(\frac13)} = \frac{3}{4} \,\frac{\Gamma(\frac16)\,\Gamma(\frac56)}{\Gamma(\frac13)\,\Gamma(\frac23)} = \frac34 \, \frac{\frac{\pi}{\sin\frac\pi 6}}{\frac{\pi}{\sin\frac\pi3}} = \frac34\,\frac{\sin\frac\pi6}{\sin\frac\pi3} = \frac34\,\frac{\frac{\sqrt 3}2}{\frac12} = \frac{3^\frac32}{4} \,.
	$$
	This proves the claimed formula \eqref{eq:bathtubconst1}.
	
	The formula for $n=2$ follows directly from Theorem \ref{isoper}, using the definition of $\mathcal C_2$ and $\mathcal C_2'$.
\end{proof}

\begin{remark}
	If one uses the lower bounds on $I(\mathbb H_n)$ from Corollary \ref{isopercor} for $n=1,2$ and inserts them into Proposition \ref{symmetrization}, one obtains a lower bound on $C^{\rm FK}(\mathbb H_n)$ that is better than the bound given by Lemma \ref{fksob} and the Jerison--Lee value \eqref{eq:jl}. Explicitly, we obtain $C^{\rm FK}(\mathbb H) \geq 8.78829 \approx 2\pi \times 1.3987$ and $C^{\rm FK}(\mathbb H_2) \geq 17.9011 \approx 2^{7/3} \pi \times 1.13064$. This leads to the bounds 
	$$
	\gamma(\mathbb H) \leq 1.65737 
	\qquad\text{and}\qquad
	\gamma(\mathbb H_2) \leq 1.26183\,,
	$$
	which are both unsatisfactory. We do see, however, that these values are better than the bounds from \eqref{eq:gammatilden}, which are stated after Proposition \ref{gammatildemono}, viz.
	$$
 	\gamma(\mathbb H) \leq 3.2423
 	\qquad\text{and}\qquad
 	\gamma(\mathbb H_2) \leq 1.8238 \,.
 	$$ 
 	When $n=3$, the lower bound from Theorem \ref{isoper}, when inserted into Proposition~\ref{symmetrization}, does not improve over the bound on $C^{\rm FK}(\mathbb H_3)$ given by Lemma \ref{fksob} and~\eqref{eq:jl}. 
\end{remark}

Despite the negative results in the previous remark, we will see that Theorem \ref{isoper} will be useful when dealing with $\mathbb H_n\times\R^k$ with $n=1$ and $n=2$.

%%%%%%%%%%%%%%%%%%%%%%%%%%%%

\subsection{A bound on the isoperimetric constant in $\mathbb H_n\times\R^k$}\label{sec:maincompn12}~\\
We recall that $I(\mathbb H_n\times\R^k)$ denotes the isoperimetric constant on $\mathbb H_n\times\R^k$. Here we prove a lower bound on this constant in terms of $I(\mathbb H_n)$ (and the known isoperimetric constants in Euclidean space).

\begin{theorem}\label{isoperk}
	For any $n,k\in\N$,
	$$
	I(\mathbb H_n\times\R^k) \geq I(\R^{2n+2+k}) \left( \frac{I(\mathbb H_n)}{I(\R^{2n+2})} \right)^\frac{2n+2}{2n+2+k} \,.
	$$
\end{theorem}

The proof of the theorem relies on the following lemma that concerns functions $v$ of bounded variations; see, e.g., \cite[Chapters 2 and 7]{Leo}. By definition, the distributional derivative of $v$ is a measure. By Lebesgue's decomposition theorem we can write this derivative as the sum of a singular measure $Dv^{({\rm s})}$ and a measure that is absolutely continuous with respect to Lebesgue measure and whose density we denote by $v'$.

\begin{lemma}\label{isocartesian}
	Let $N\geq 2$ and let $v\in L^\frac{N-2}{N-1}(\R)$ be a function of bounded variation. Then, for any $\beta>0$,
	$$
	\int_\R \sqrt{ (\beta v^\frac{N-2}{N-1})^2 + (v')^2}\,dt + |Dv^{({\rm s})}|(\R) \geq I(\R^N) \left( \frac{\beta}{I(\R^{N-1})} \right)^\frac{N-1}{N} \left( \int_\R |v| \,dt \right)^\frac{N-1}{N}.
	$$	
\end{lemma}

If $N=2$, we understand $v^\frac{N-2}{N-1}$ as the characteristic function of $\{|v|\neq 0\}$, and we understand the assumption $v\in L^\frac{N-2}{N-1}(\R)$ as requiring that this set has finite measure.

If in the first formula in \eqref{eq:isofkexplicit} we replace $|\Sph^{k-1}|$ by $2\pi^\frac k2 \Gamma(\frac k2)^{-1}$, then we see that $I(\R^k)$ can be defined for any (not necessarily integer) real number $k>0$. With this definition, Lemma \ref{isocartesian} remains valid for (not necessarily integer) $N\geq 2$. However, for us only the case where $N$ is an integer case is relevant and we only provide the \emph{proof} in this case. We note that by scaling, if the inequality holds for one value of $\beta>0$, then it holds for any such value. Thus, we may assume that $\beta = I(\R^{N-1})$. In this case, the lemma follows from the isoperimetric inequality in $\R^N$ applied to the set $\{ (x',x_N)\in\R^{N-1}\times\R:\ \omega_{N-1} |x'|^{N-1}<v(x_N)\}$, where $\omega_{N-1}$ is the volume of the unit ball in $\R^{N-1}$.

\begin{proof}[Proof of Theorem \ref{isoperk}]
	Let $E\subset\mathbb H_n\times\R^k$ have finite measure and finite perimeter and set
	\begin{align*}
		v(t) & := |\{ \zeta\in\mathbb H_n\times\R^{k-1} :\ (\zeta,t)\in E\}| \,, \\
		p(t) & := \per_{\mathbb H_n\times\R^{k-1}} \{ \zeta\in\mathbb H_n\times\R^{k-1} :\ (\zeta,t)\in E\} \,.
	\end{align*}
	It follows from Fubini's theorem that $v$ is an integrable function with
	$$
	\int_\R v(t)\,dt = |E| \,.
	$$
	Moreover, one can show that $v$ is of bounded variation and that $p$ is integrable and that
	$$
	\per_{\mathbb H_n\times\R^k} E \geq \int_\R \sqrt{ p^2 + (v')^2}\,dt + |Dv^{({\rm s})}|(\R) \,.
	$$
	(This requires some work. Similar facts appear in the review paper \cite{Ta} on the isoperimetric inequality on $\R^k$.) We bound
	$$
	p(t) \geq I(\mathbb H_n\times\R^{k-1})\, v(t)^\frac{2n+k}{2n+k+1}
	\qquad\text{for all}\ t\in\R
	$$
	and apply the lemma with $\beta =  I(\mathbb H_n\times\R^{k-1})$ and $N=2n+2+k$ to deduce that
	$$
	\per_{\mathbb H_n\times\R^k} E \geq 
	I(\R^{2n+2+k}) \left(\frac{I(\mathbb H_n\times\R^{k-1})}{I(\R^{2n+1+k})}\right)^\frac{2n+1+k}{2n+2+k} |E|^\frac{2n+1+k}{2n+2+k} \,.
	$$
	Thus,
	$$
	I(\mathbb H_n\times\R^k) \geq 
	I(\R^{2n+2+k}) \left(\frac{I(\mathbb H_n\times\R^{k-1})}{I(\R^{2n+1+k})}\right)^\frac{2n+1+k}{2n+2+k} \,.
	$$
	Dropping $n$ from the notation and abbreviating $\iota_k:= \left( I(\mathbb H_n\times\R^k)/I(\R^{2n+2+k}) \right)^{2n+2+k}$, we can write this bound as $\iota_k\geq\iota_{k-1}$. Thus, $\iota_k\geq\iota_0$, which is the assertion of the theorem.
\end{proof}

As an immediate consequence of Theorem \ref{isoperk} and Proposition \ref{symmetrization} (in the form \eqref{eq:symmetrization}) we obtain the following lower bound on the Faber--Krahn constant.

\begin{corollary}\label{isoperkcor}
	For any $n,k\in\N$,
	\begin{equation}
		\label{eq:isoperkcoralt}
		C^{\rm FK}(\mathbb H_n\times\R^k) \geq C^{\rm FK}(\R^{2n+k+2}) \left( \frac{I(\mathbb H_n)}{I(\R^{2n+2})} \right)^\frac{2(2n+2)}{2n+2+k} \,.
	\end{equation}	
\end{corollary}

One can bring the bound \eqref{eq:isoperkcoralt} into a somewhat more explicit form by using the explicit expressions for $C^{\rm FK}(\R^{2n+k+2})$ and $I(\R^{2n+2})$ from \eqref{eq:isofkexplicit}.

%%%%%%%%%%%%%%%%%%%%%%

With the help of Corollary \ref{isopercor} we are ready to prove that $\gamma(\mathbb H_n\times\R^k)<1$ if $n=1$ and $k\geq 2$ or if $n=2$ and $k\geq 1$.

\smallskip

\begin{proof}[Proof of Theorem \ref{maincomp}. Cases $n=1$ and $n=2$]~\\ 
	We recall the expression for $\mathcal W(\mathbb H_n\times\R^k)$ from \eqref{eq:wnk} and write it as
	$$
	\mathcal W(\mathbb H_n\times\R^k) = \mathcal W(\R^{2n+2+k}) \mathcal W(\mathbb H_n) (4\pi)^{n+1} \Gamma(n+2) \,. 
	$$
	Using Corollary \ref{isoperkcor} we find that
	\begin{equation}
		\label{eq:gammankbound}
		\gamma(\mathbb H_n\times\R^k) \leq \gamma(\R^{2n+2+k})\ \left( \frac{I(\R^{2n+2})}{I(\mathbb H_n)} \right)^{2n+2} \mathcal W(\mathbb H_n)^{-1} (4\pi)^{-n-1} \Gamma(n+2)^{-1} \,.
	\end{equation}
	Note that the upper bound depends on $k$ only through $\gamma(\R^{2n+2+k})$. We recall from \cite[Theorem~5.1]{HPS} that the sequence $d\mapsto \gamma(\R^d)$ is decreasing. Therefore, if we can show that the upper bound is $<1$ for $(n,k)=(1,2)$ and $(n,k)=(2,1)$, then the assertions of the theorem for $n=1$ and $n=2$ will follow.
	
	\medskip
	
	\emph{Case $n=1$.} Recalling from \eqref{eq:W1}  that $\mathcal W(\mathbb H) = \frac1{128}$, we get
	\begin{align}\label{eq:gammankboundn1}
		\gamma(\mathbb H\times\R^k) 
		& \leq 4^4 \left( \frac{\Gamma(\frac{6+k}{2})}{\Gamma(3)} \right)^2
		\left( I(\mathbb H) \right)^{-4} \left( \frac{1}{j_{\frac{2+k}{2},1}} \right)^{4+k}
		\left( \mathcal W(\mathbb H) \right)^{-1} 2^k \notag \\
		& = 2^{13+k} \Gamma(\tfrac{6+k}2)^2 \left( I(\mathbb H) \right)^{-4} \left( \frac{1}{j_{\frac{2+k}{2},1}} \right)^{4+k} \,.
	\end{align}
	Specializing further to $k=2$,  this bound becomes
	\begin{align*}
		\gamma(\mathbb H\times\R^2) 
		& \leq 2^{17} 3^2 \left( I(\mathbb H) \right)^{-4} \left( \frac{1}{j_{2,1}} \right)^6 \,.
	\end{align*}
	We have $j_{2,1}\approx 5.13562$ \cite[Table 9.5]{AbSt} and, using the lower bound on $I(\mathbb H)$ from Corollary \ref{isopercor}, we obtain
	$$
	\gamma(\mathbb H\times\R^2) \leq 0.701019 \,.
	$$
	This is $<1$, as desired.
	
	\medskip
	
	\emph{Case $n=2$.} Recalling from \eqref{eq:W2} that $\mathcal W(\mathbb H_2) = \frac{1}{48^2\, \pi}$, we get
	\begin{align*}
		\gamma(\mathbb H_2\times\R^k) 
		& \leq 
		6^6 \left( \frac{\Gamma(\frac{8+k}{2})}{\Gamma(4)} \right)^2
		\left( I(\mathbb H_2) \right)^{-6} \left( \frac{1}{j_{\frac{4+k}{2},1}} \right)^{6+k}
		\left( \mathcal W(\mathbb H_2) \right)^{-1} 2^k \\
		& = 2^{12+k} 3^6 \pi \Gamma(\tfrac{8+k}{2})^2 \left( I(\mathbb H_2) \right)^{-6} \left( \frac{1}{j_{\frac{4+k}{2},1}} \right)^{6+k} \,.
	\end{align*}
	Specializing further to $k=1$, this bound becomes
	\begin{align*}
		\gamma(\mathbb H_2\times\R) 
		& \leq 2^{5} 3^8 5^2 7^2 \pi^2 \left( I(\mathbb H_2) \right)^{-6} \left( \frac{1}{j_{\frac{5}{2},1}} \right)^{7}
		\,.
	\end{align*}
	We have $j_{\frac52,1}\approx 5.76346$ \cite[Table 10.6]{AbSt} and, using the lower bound on $I(\mathbb H_2)$ from Corollary \ref{isopercor}, we obtain
	$$
	\gamma(\mathbb H_2\times\R) \leq 0.823715 \,.
	$$
	This is $<1$, as desired.
\end{proof}

As another application of Corollary \ref{isoperkcor} we now show that Pansu's conjecture implies Pleijel's theorem.

\begin{proof}[Proof of Proposition \ref{pansupleijel}]
	In view of Theorem \ref{maincomp} and Corollary \ref{pansupleijelk0} it suffices to prove that Pansu's conjecture in $\mathbb H$ implies $\gamma(\mathbb H\times\R)<1$. Arguing as in the previous proof, we want to show that the right side of \eqref{eq:gammankboundn1} is $<1$ when $k=1$ and when $I(\mathbb H)$ is replaced by the conjectured value on the right side of \eqref{eq:pansusconj}, which is $2^{\frac 52} 3^{-\frac 34}\pi^{\frac 12}$; see \eqref{eq:pansun1}. Using $j_{\frac32,1}\approx 4.49341$ \cite[Table 10.6]{AbSt}, we obtain
	$$
	2^{14} \, \Gamma(\tfrac 72)^2 \left( 2^{\frac 52} 3^{-\frac 34}\pi^{\frac 12} \right)^{-4} j_{\frac32,1}^{-5} \approx 2.639 \cdot 10^{-1} <1 \,,
	$$
	as claimed.
\end{proof}

\newpage 
%%%%%%%%%%%%%%%%%%%%%%%%%%%%%%%%%%%%
%%%%%%%%%%%%%%%%%%%%%%%%%%%%%%%%%%%%
\appendix

\part{Appendix}
 
\section{Local linear independence}\label{AppA2}

In our description of the nilpotent approximation in Section \ref{s1} and in the proof of our main theorem in Section \ref{sec:mainsubriemproof} we made use of results by Rothschild \cite{Rot} that were established under the assumption that the vectors $X_1(x),\ldots,X_p(x)$ are linearly independent at one (and hence, by equiregularity) any point $x\in M$. In this appendix we show that locally around any given point we may always reduce ourselves to the case where this is satisfied. We emphasize that this is well known in sub-Riemannian geometry and we briefly sketch the proof in that language in Subsection \ref{sec:subriemannianortho}. We think, however, that it is beneficial to also give an elementary proof ``by hand''. Finally, in Subsection \ref{ssA.4} we give discuss topological obstructions to a global version of this statement.

More precisely, we work under Assumptions \ref{CHOr} and \ref{equireg} and we recall that $n_j=\dim\mathcal D_j(x)$. We shall prove the following.

\begin{lemma}\label{locindep}
	For any $x_0\in M$ there is an open neighborhood $W\subset M$ and smooth vector fields $\tilde X_1,\ldots,\tilde X_{n_1}$ defined in $W$ such that
	\begin{equation}
		\label{eq:loclinindepapp}
		{\rm span}\{ \tilde X_1(x),\ldots \tilde X_{n_1}(x) \} = \mathcal D_1(x)
		\qquad\text{for all}\ x\in W
	\end{equation}
	and
	$$
	-\Delta_{\bf X}^{M,\mu} f = \sum_{j=1}^{n_1} \tilde X_j^\star \tilde X_j f
	\qquad\text{for all} \ f\in C^2(W) \,.
	$$
	The vector fields $\tilde X_1,\ldots,\tilde X_{n_1}$ still satisfy Assumptions \ref{CHOr} (with the same $r$) and \ref{equireg} (with the dimensions $n_j$).
\end{lemma}

\medskip

\subsection{A first example}~\\ 
To start, we consider the example where
$$
-\Delta_\Xb = X_1^\star X_1 + X_2^\star X_2 + X_3^\star X_3 \,.
$$
We suppose that Assumptions \ref{CHOr} with $r=2$ and Assumption \ref{equireg} are satisfied, but we assume that at each point $x\in M$ the space generated by $X_1(x), X_2(x), X_3(x)$ is of dimension $2$.

Let us derive Lemma \ref{locindep} in this particular case, that is, show that locally, we can represent $-\Delta_\Xb$ in the form
 $$
- \Delta_\Xb =  \tilde X_1^\star \tilde  X_1 + \tilde X_2^\star \tilde X_2 \,,
 $$
for suitable  vector fields $\tilde X_1,\tilde X_2$, such that for all $x$ in a neighborhood of a given point the vectors $\tilde X_1(x)$ and $\tilde X_2(x)$ are linearly independent and belong to $\mathcal D_1(x)$.
 
We proceed as for the proof of the Morse lemma. Let us assume that $X_1$ and $X_2$ are linearly independent in some open set $W$ in $M$. Hence we can write for $x\in W$
$$
X_3(x) = a_1(x) X_1(x) + a_2(x) X_2(x)
$$
with two smooth functions $a_1$ and $a_2$ on $W$. The two desired vector fields are given by
\begin{equation}
 \begin{array}{ll}
 \tilde X_1 &= \sqrt{1+a_1^2} \, X_1 + \frac{a_1 a_2}{\sqrt{1+ a_1^2}} \, X_2 \,, \\
 \tilde X_2 &= \sqrt{\frac{1+ a_1^2 + a_2^2}{1+a_1^2}} \, X_2\,.
  \end{array}
\end{equation}
Clearly $\tilde X_1$ and $\tilde X_2$ are linearly independent. Note also that
  $$
  [\tilde X_1,\tilde X_2] = \sqrt{1+a_1^2 +a_2^2}\, [X_1,X_2] + b_1 X_1 + b_2 X_2
  $$
with two smooth functions $b_1$ and $b_2$ on $W$. This proves the assertions of Lemma \ref{locindep} in this example.

In passing we mention that in this case we can take as an adapted flag (see Section~\ref{s1}) the vector fields $\tilde X_1, \tilde X_2,  [\tilde X_1,\tilde X_2] $, so that the group $G_x$ is the Heisenberg group.

%%%%%%%%%%%%%%%%%%%%%

%%%%%%%%%%%%%%%%%%%%%

\medskip

\subsection{General argument}~\\
The general argument is as follows: we consider a sub-Laplacian $\Delta_\Xb^{M,\mu} =-\sum_{i=1}^p X_i^\star X_i$ and assume that $p>n_1$ (for otherwise there is nothing to prove). For $f\in C^\infty(M)$ we consider
$$
Q[f](x) := \sum_{j=1}^p (X_j f (x))^2 \,.
$$

We choose $n_1$ vector fields, which we can assume to be $X_1,\ldots,X_{n_1}$ (up to relabeling), that are linearly independent in a neighborhood of some point $x_0\in M$. In the sequel $x$ denotes the variable in this neighborhood. For $j=n_1+1,\ldots,p$,  we find coefficients $a_{ij}$ such that
$$
X_j=\sum_{\ell=1}^{n_1} a_{j\ell}(x)X_\ell \,.
$$
Expanding, we obtain for $f$ with compact support in this neighborhood
$$
Q[f](x)=\sum_{\ell, \ell'=1}^{n_1} b_{\ell \ell'}(x)X_\ell f(x) X_{\ell'}f(x)
$$
with some coefficients $b_{\ell \ell'}$ computed in terms of the $a_{j\ell}$. The matrix
$$
B_x :=(b_{\ell \ell'}(x))_{1\leq \ell,\ell'\leq n_1}
$$
is symmetric and positive definite.

In this case, we can have a normal form (see the proof of the Morse lemma):
$$
B_x = T_x^t\,  T_x \,,
$$
where $T_x = (t_{i\ell}(x))$ is triangular, invertible,  depending smoothly on $x$ in the construction.  
We set for $i=1,\ldots,n_1$, 
$$
\tilde{X}_i(x) := \sum_{\ell=i}^{n_1} t_{i\ell}(x)X_\ell(x)
$$
and get
$$
Q[f](x) =\sum_{i=1}^{n_1} (\tilde{X}_i f(x))^2 \,.
$$
Integrating this identity with respect to $\mu$ we conclude that, locally,
$$
\Delta = - \sum_{i=1}^{n_1} \tilde{X}_i^\star \tilde{X}_i \,,
$$
as claimed. Clearly, at each point $x$ in the relevant neighborhood the span of the vectors $\tilde X_1(x),\ldots,\tilde X_{n_1}(x)$ is equal to $\mathcal D_1(x)$. Since the $\mathcal D_j$ for $j\geq 1$ depend on the vector fields $X_1,\ldots,X_p$ only through their span $\mathcal D_1$, we see that the $\tilde X_1,\ldots,\tilde X_{n_1}$ satisfy Assumptions \ref{CHOr} (with the same $r$) and \ref{equireg} (with the dimensions $n_j$). This concludes the proof of Lemma \ref{locindep}.

\medskip

\subsection{Sub-Riemannian geometric construction}\label{sec:subriemannianortho}~\\
We provide\footnote{Discussions with Y.~ Colin de Verdi\`ere, L.~Hillairet and C.~Letrouit.}an alternative approach to the local argument of the previous subsection.
We refer for example to the book \cite{Je}. We introduce the following metric on the distribution: for $x\in M$ and $v$ a linear combination of $X_1(x),\ldots,X_p(x)$ we set
$$
g_x(v) := \inf\left\{\left(\sum_{i=1}^p u_i^2\right)^{1/2} :\ \sum_{i=1}^p u_iX_i(x)=v\right\}.
$$
Then we pick a smooth orthonormal frame $(\tilde{X}_1(x),\ldots,\tilde{X}_{n_1}(x))$ for the above metric. Finally we check that
\begin{equation}\label{e:equality}
\sum_{i=1}^p (X_if)^2=\sum_{i=1}^{n_1} (\tilde{X}_if )^2 \,.
\end{equation}
This is done as follows: we pick coefficients $u_{ij}$ such that $\tilde{X}_i=\sum_{j=1}^p u_{ij} X_j$ with $\sum_{j=1}^p u_{ij}^2=1$. By orthogonality of the $\widetilde{X}_i$'s, we know that $\sum_{j=1}^p u_{ij} u_{i' j}=0$ for $i\neq i'$. Plugging this into the right side of \eqref{e:equality}, we obtain its left side.

\medskip

\subsection{Topological constraints}\label{ssA.4}~\\
The existence of $n_1$ global vector fields in $\mathcal D_1$ giving for each 
$x\in M$ a basis of $T_xM$ is only possible under strong topological conditions on $M$ that involve its orientability, its Euler characteristic, the Euler class of $\mathcal D_1$ and other invariants\footnote{Discussion with V.~Colin.}.

Note first that the existence of a global non-zero vector field implies that the Euler characteristic of $M$ is zero by the Poincar\'e--Hopf theorem. 
When $M$ has dimension~3 and is orientable, this is not an obstruction since this Euler characteristic is zero. A second obstruction related to the Euler class of the fiber bundle also disappears when $M=S^3$.  If $n_1=2$, one can instead of a basis of $\mathcal D_1(x)$ consider the normal to the plane. This gives a map from $S^3$ to $S^2$ and the associated Hopf invariant  (which belongs to $\mathbb Z$) should be zero, which is not the case in general.

%%%%%%%%%%%%%%%%%%%%%%%%%
%%%%%%%%%%%%%%%%%%%%%%%%%

\section{Weyl law and heat kernel}\label{AppC}
 This section is based on notes kindly transmitted to us by C.\ Letrouit. We make a connection with the framework of \cite{CHT1}.

\medskip

\subsection{Nilpotentization of measure}~\\ 
We follow \cite[Section A.5.6]{CHT1}. Let $\mu$ be a smooth measure on $M$ and let $x \in M$. Recall that privileged coordinates are defined by $\theta_x$, which maps a neighborhood of $x$ in $M$ to a neighborhood of $0$ in $\R^n$. We also recall the definition of the family of dilations $\delta_t$ on $\R^n$ from \eqref{eq:dilation}. Setting $\delta_x^{(\epsilon)} := \theta_x^{-1} \circ \delta_\varepsilon$, for $\varepsilon>0$ we can define the measures
$$
\mu_x^{(\varepsilon)} := \varepsilon^{-Q}(\delta_x^{(\varepsilon)})^*\mu
$$
on a neighborhood of the origin in $\R^n$. The nilpotentization $\widehat\mu_x$ of $\mu$ at $x$, which we defined in \eqref{eq:nilpmeas} ``by hand'', satisfies
$$
\widehat{\mu}_x=\lim_{\varepsilon\rightarrow 0_+}\mu_x^{(\epsilon)}
$$
with convergence in the vague topology.

\medskip

\subsection{Nilpotentization of spectral function and heat kernel}~\\
We follow \cite[Section A.8.3]{CHT1}. The nilpotentized sub-Laplacian at $x$ is defined as an operator on functions on $G_x\simeq \R^n$ by
$$
\widehat{\Delta}_x= \sum_{i=1}^m (\widehat{X}_{i,x})^2.
$$
It is self-adjoint in $L^2(G_x,\widehat\mu_x)$ with the usual domain and nonnegative.

In this appendix we denote by
$$
\widehat e_x(\lambda,u,v) := \1(-\widehat\Delta_x <\lambda)(u,v)
$$
the spectral function of $-\widehat\Delta_x$, that is, the integral kernel of the spectral projection $\1(-\widehat\Delta_x<\lambda)$ for $\lambda>0$.

The operator $\partial_t- \widehat{\Delta}_x$ is hypoelliptic and therefore $\exp (t \widehat \Delta_x)$ is an integral operator whose integral kernel is denoted by
$$
\widehat k_x(t,u,v) := (\exp(t\widehat\Delta_x))(u,v) \,,
$$
that is,
$$
(\exp(t \widehat \Delta_x) f)(u)= \int_{G_x} \widehat k_x(t,u,v)  f(v)\, d \widehat{\mu}_x(v)
\qquad \forall u\in G_x\,.
$$
According to the functional calculus we have
\begin{equation}
	\label{eq:heatkernelspecfcn}
	\widehat k_x(t,u,v) = \int_0^\infty e^{-t\lambda} \widehat e_x(\lambda,u,v)\,dt 
	\qquad\text{for all}\ t>0 \,.
\end{equation}

\medskip

\subsection{Effect of changes of coordinates}~\\
We follow \cite[second part of Appendix A.8.2]{CHT1}. The above nilpotentizations of the spectral function and of the heat kernel depend on the choice of privileged coordinates (we have made a specific choice in the main text) and on the measure $\widehat\mu_x$. We claim that diagonal values
$$
\widehat e_x(\lambda,u,u)
\qquad\text{and}\qquad
\widehat k_x(t,u,u)
$$
do not depend on this choice.

Indeed, this follows from \cite[third relation in (85)]{CHT}, which says that changing variables both in the operator and in the measure has no effect on the heat kernel. More precisely, if we denote the integral kernel of $-\Delta^{M,\mu}_{\bf X}$ by $k_{\Delta,\mu}$ and if $\varphi$ is a diffeomorphism of $M$, then
\begin{equation}
	\label{eq:diffeoheat}
	k_{\varphi^* \Delta \varphi_*, \varphi^*\mu}(t,u,v)=k_{\Delta,\mu}(t,\varphi(u),\varphi(v)) \,.
\end{equation}
We apply this formula with $M$ replaced by $G_x$, with $-\Delta^{M,\mu}_{\bf X}$ replaced by $-\widehat\Delta_x$ and with $\mu$ replaced by $\widehat\mu_x$. A similar argument applies to the spectral function.

Choosing $\varphi$ in \eqref{eq:diffeoheat} as the diffeomorphism on $G_x$ given either by dilation or by translation by a group element, we see that
$$
\widehat k_x(t,u,u) = c^{\rm heat}_x \, t^{-\frac Q2}
\qquad\text{for all}\ u\in G_x \,,\\ t>0 \,,
$$
where
$$
c^{\rm heat}_x = \widehat k_x(1,0,0) \,.
$$
These are the analogues of \eqref{eq:specfcndiag1} and \eqref{eq:specfcndiag2}. It follows from \eqref{eq:heatkernelspecfcn} that
\begin{equation}
	\label{eq:heatkernelspecfcnconst}
	c_x^{\rm heat} = \Gamma(\tfrac Q2+1) \, c^{\rm Weyl}_x \,.
\end{equation}

\medskip

\subsection{Heat kernel asymptotics and Weyl law}~\\
The following appears in \cite[Theorems I and 4.1]{CHT1}:

\begin{theorem}\label{weylapp}
In the equiregular case, for every $f\in C^\infty(M)$, we have
$$
\Tr f e^{t\Delta^{M,\mu}_{\bf X}} =
\int_M f(x)k_{\Delta,\mu}(t,x,x)\,d\mu(x) = t^{-\frac Q2}\, F(t)
$$
for some $F\in C^\infty(\R)$ with
$$
F(0)=\int_M f(x) \, c^{\rm heat}_x \,d\mu(x) \,.
$$
Moreover, the eigenvalue counting function satisfies
\begin{equation}\label{e:specasymp}
N(\lambda,-\Delta^{M,\mu}_{\bf X})\sim \int_M c^{\rm Weyl}_x \,d\mu(x) \ \lambda^{\frac Q2}
\qquad\text{as}\ \lambda\to\infty \,.
\end{equation}
\end{theorem}

Behind the proof of this theorem is the fact (see \cite[Theorem C.1 and Eq.\ (96)]{CHT1}) that $t^{\frac Q2} k_{\Delta,\mu}(t,x,x)$ converges to $\widehat{e}^x(1,0,0)=c^{\rm heat}_x$ as $t\rightarrow 0^+$, uniformly with respect to $x$. This uses the equiregularity assumption. The spectral asymptotics \eqref{e:specasymp} follows from the heat kernel asymptotics by a Tauberian theorem. The coefficient there coincide with that in \cite{CHT1} in view of \eqref{eq:heatkernelspecfcnconst}.

%%%%%%%%%%%%%%%%%%%%%%%%%%%%%%%%%%
%%%%%%%%%%%%%%%%%%%%%%%%%%%%%%%%%% 

\newpage

%%%%%%%%%%%%%%%%%%%%%%%%%%%%%%%%%%
%%%%%%%%%%%%%%%%%%%%%%%%%%%%%%%%%% 

%\nocite{*}

\bibliographystyle{amsalpha}

\begin{thebibliography}{83}
 
\bibitem{AbSt} M. Abramowitz, I. A. Stegun. 
\newblock Handbook of mathematical functions with formulas, graphs, and mathematical tables.
\newblock National Bureau of Standards Applied Mathematics Series, No. 55 U. S. Government Printing Office, Washington, D.C., 1964
 
 \bibitem{An} C. Ann\'e.
 \newblock Bornes sur la multiplicit\'e.
 \newblock Pr\'epublications EPFL (1992).
 
 \bibitem{ArRi} V. Arnaiz and G. Rivi\`ere.
 \newblock Quantum limits of perturbed sub-Riemannian contact Laplacians in dimension 3.
 \newblock Preprint (2023), arXiv:2306.10757v1 (2023).
 
\bibitem{Ba} H. Bahouri.
\newblock Sur la propri\'et\'e de prolongement unique pour les
op\'erateurs de H\"ormander.
\newblock 
Journ\'ees \'equations aux d\'eriv\'ees partielles (1983), p. 1-7 and 
Annales de l'Institut Fourier, tome 36, no 4 (1986),  137--155.

\bibitem{BaGo} M.S. Baouendi and C. Goulaouic.
\newblock Non-analytic-hypoellipticity for some degenerate elliptic operator.
\newblock  Bulletin of the American Mathematical Society, 
Volume 78, Number 3, 483--486, (1972).


\bibitem{BaCoLeSC} D. Bakry, T. Coulhon, M. Ledoux, and L. Saloff-Coste.
\newblock Sobolev inequalities in disguise.
\newblock Indiana Univ. Math. J. 44 (1995), no. 4, 1033--1074.

\bibitem{BeRi} A. Bella\"\i che  and J.J. Risler.
\newblock Sub-Riemannian  Geometry. Birkh\"auser (1996).

\bibitem{BerMe} P. B\'erard and D. Meyer.
\newblock In\'egalit\'es isop\'erim\'etriques et applications. 
\newblock Annales scientifiques de l'\'Ecole Normale Sup\'erieure, S\'erie 4,  Tome 15, no 3,  513--541 (1982).


\bibitem{Bo} J.M. Bony.
\newblock Principe du maximum, in\'egalit\'e de Harnack et unicit\'e du probl\`eme de Cauchy pour des op\'erateurs elliptiques d\'eg\'en\'er\'es.
\newblock Annales de l'institut Fourier, (1969), Vol. 19, no 1, 277--304.




\bibitem{BrFoMo} T. P. Branson, L. Fontana, and C. Morpurgo.
\newblock Moser--Trudinger and Beckner--Onofri's inequalities on the CR sphere.
\newblock Ann. of Math. (2) 177 (2013), no. 1, 1--52.

\bibitem{CaDaPaTy} L. Capogna, D. Danielli, S. D. Pauls, and J. T. Tyson.
\newblock An introduction to the Heisenberg group and the sub-Riemannian isoperimetric problem. 
\newblock Progr. Math., 259. Birkh\"auser Verlag, Basel (2007).

\bibitem{Car} G. Carron.
\newblock In\'egalit\'es isop\'erim\'etriques de Faber-Krahn et cons\'equences.
\newblock S\'emin. Congr., 1, Soci\'et\'e Math\'ematique de France, Paris, 1996, 205--232.

\bibitem{ChChHwYa} J.-H. Cheng, H.-L. Chiu, J.-F. Hwang, and  P. Yang.
\newblock Umbilicity and characterization of Pansu spheres in the Heisenberg group.
\newblock J. Reine Angew. Math. 738 (2018), 203--235.

\bibitem{ChHwMaYa} J.-H. Cheng, J.-F. Hwang, A. Malchiodi, and P. Yang.
\newblock A Codazzi-like equation and the singular set for $ C^1 $ smooth surfaces in the Heisenberg group.
\newblock J. Reine Angew. Math. 671 (2012), 131--198.

\bibitem{ChoPo1} W. Choi and R. Ponge.
\newblock Privileged coordinates and nilpotent approximation for Carnot manifolds. General results. 
\newblock J. Dyn. Control Syst. 25 (2019), 109--157.

\bibitem{ChoPo2} W. Choi and R. Ponge.
\newblock Privileged coordinates and nilpotent approximation for Carnot manifolds. II Carnot coordinates.
\newblock J. Dyn. Control Syst. 25 (2019), 631--670.

\bibitem{CoLu} W. S. Cohn, G. Lu.
\newblock Best constants for Moser-Trudinger inequalities on the Heisenberg group.
\newblock Indiana Univ. Math. J. 50 (2001), no. 4, 1567--1591.

 
\bibitem{CHT0} Y. Colin de Verdi\`ere, L. Hillairet, and E. Tr\'elat.
 \newblock Spectral asymptotics for sub-Riemannian  Laplacians, I: Quantum ergodicity and quantum limits in the 3-dimensional contact
case.
\newblock Duke Math. Journal, 167 (1): 109--174  (2018).

\bibitem{CHT} Y. Colin de Verdi\`ere, L.  Hillairet, and E. Tr\'elat.
 \newblock Small-time asymptotics of hypoelliptic heat kernels near the diagonal, nilpotentization and related results.
 \newblock  Annales Henri Lebesgue 4 (2021): 897--971.

\bibitem{CHT1} Y.  Colin de Verdi\`ere, L. Hillairet,  and E.  Tr\'elat.
 \newblock Spectral asymptotics for sub-Riemannian  Laplacians.
 \newblock Preprint (2022), arXiv:2212.02920.

\bibitem{CH} R. Courant and D. Hilbert.
\newblock Methods of Mathematical Physics: Partial Differential Equations, Vol.~2. 
\newblock John Wiley and Sons, (2008).


 \bibitem{DGN} D. Danielli, N. Garofalo, and D.-M. Nhieu.
\newblock Non-doubling Ahlfors measures, perimeter measures, and the characterization of the trace spaces of Sobolev functions in Carnot--Carath\'eodory spaces.
\newblock Memoirs of the AMS, Vol. 182, No 537 (2006).

\bibitem{DGN2} D. Danielli, N. Garofalo, and D.-M. Nhieu.
\newblock A partial solution of the isoperimetric problem for the Heisenberg group.
\newblock Forum Math. 20 (2008), no. 1, 99--143.

\bibitem{De} M. Derridj.
\newblock Un probl\`eme aux limites pour une classe d'op\'erateurs du second ordre hypoelliptiques.
\newblock  Ann. Inst. Fourier, Grenoble, 21, 4 (1971), 99--148.

 \bibitem{DFV} N. De Ponti, S. Farinelli, and  I.Y. Violo. 
\newblock Pleijel nodal domain theorem in non-smooth setting.
\newblock ArXiv:2307.13983 (2023).

\bibitem{EL} S. Eswarathasan and C. Letrouit.
\newblock Nodal sets of Eigenfunctions of sub-Laplacians.
\newblock ArXiv January 2023. International Mathematics Research Notices (2023), to appear.

\bibitem{Fo0} G. Folland. 
\newblock A fundamental solution for a subelliptic operator. 
\newblock Bull. Amer. Math. Soc. 79 (1973), 373--376.

\bibitem{Fo} G. Folland. 
\newblock Subelliptic estimates and function spaces on nilpotent groups.
\newblock Ark. Math. 13 (1975), 161--207.

\bibitem{FS} G.B. Folland and E.M. Stein.
\newblock Estimates for the $\overline{\partial}_b$-complex and analysis on the Heisenberg group.
\newblock Comm. Pure Appl. Math. 27 (1974), 429--522.

\bibitem{FrSeSC} B. Franchi, R. Serapioni, and F. Serra Cassano.
\newblock Meyers--Serrin type theorems and relaxation of variational integrals depending on vector fields.
\newblock Houston J. Math. 22 (1996), no. 4, 859--890.

\bibitem{FLW} R.L. Frank, A. Laptev, and T. Weidl.
\newblock Schr\"odinger operators: eigenvalues and Lieb--Thirring inequalities.
\newblock Cambridge Stud. Adv. Math., 200, Cambridge University Press, Cambridge (2023).

\bibitem{FrLi} Rupert L. Frank and Elliott  H. Lieb. Sharp constants in several inequalities on the Heisenberg group.
\newblock Annals of Mathematics (2012), 349--381.

\bibitem{FGL} R.L. Frank, D. Gontier, and M. Lewin.
\newblock The nonlinear Schr\"odinger equation for  orthonormal functions II: application to Lieb--Thirring inequalities.
\newblock Commun. Math. Phys. 384 (2021), 1783--1828.

\bibitem{FrHe1} R.L. Frank, B. Helffer.
\newblock On Courant and Pleijel theorems for sub-Riemannian Laplacians.
\newblock Preprint (2024), arXiv:2402.13953v2. 

\bibitem{FrHe2} R.L. Frank, B. Helffer.
\newblock On Courant and Pleijel theorems for sub-Riemannian Laplacians.
\newblock In: D. Rottensteiner, M. Ruzhansky, and V. Kumar (eds.), Pseudo-differential operators and related topics. Extended Abstracts PSORT 2024, Trends in Mathematics, Research Perspectives Ghent Analysis and PDE Center, Birkh\"auser.

\bibitem{GaNh} N. Garofalo and  D.-M. Nhieu. 
\newblock Isoperimetric and Sobolev inequalities for Carnot--Carath\'eodory spaces and the existence of minimal surfaces.
\newblock  Comm. Pure Appl. Math. \textbf{49} (1996), no. 10, 1081--1144.

\bibitem{Go} R.W. Goodman.
\newblock Nilpotent Lie groups.
\newblock Lecture Notes in Mathematics. No 562. Springer (1976).

\bibitem{HaKo} P. Haj\l asz, P. Koskela, SobFot Poincar\'e. 
\newblock Mem. Amer. Math. Soc. \textbf{145} (2000), no. 688. 

\bibitem{HL} A. M. Hansson and A. Laptev.
\newblock  Sharp spectral inequalities for the Heisenberg Laplacian. 
\newblock London Math.
Soc. Lecture Note Ser., 354, 100--115, (2008).

\bibitem{HS} A. Hassannezhad and D.  Sher. 
\newblock On Pleijel's nodal domain theorem for the Robin problem.
\newblock arXiv: 2303.08094   (2023).

\bibitem{HeKiMa} J. Heinonen, T. Kilpel\"ainen, and O. Martio, Nonlinear potential theory of degenerate elliptic equations. Unabridged republication of the 1993 original. Dover Publications, Inc., Mineola, NY, (2006).

\bibitem{Hel} B. Helffer.
\newblock Conditions n\'ecessaires d'hypoanalyticit\'e pour des op\'eerateurs invariants \`a gauche sur un groupe nilpotent gradu\'e.
\newblock J. Diff. Eq. 44 (1982) 460--481.

\bibitem{HN} B. Helffer and J. Nourrigat.
\newblock Approximation d'un syst\`eme de champs de vecteurs et applications \`a l'hypoellipticit\'e. 
\newblock  Ark. Mat. 17(1-2): 237-254 (1979). 

\bibitem{HN1} B. Helffer and J. Nourrigat.
\newblock Hypoellipticit\'e Maximale pour des Op\'erateurs Polyn\^omes de Champs de Vecteurs.
\newblock Progress in Mathematics. Birkh\"auser. 1985 

\bibitem{HPS} B. Helffer and M. Persson Sundqvist.
\newblock On nodal domains in Euclidean balls.
\newblock Proc. Amer.Math. Soc. 144 (11): 4777--4791 (2017).

\bibitem{HeNoSw} J. G. Heywood, E. S. Noussair, and C. A. Swanson. \newblock On the zeros of solutions of elliptic inequalities in bounded domains. 
\newblock J. Differential Equations \textbf{28} (1978), no.3, 345--353.

\bibitem{Ho} L. H\"ormander.
\newblock 
Hypoelliptic second order differential equations.
\newblock 
Acta Math. 119: 147-171 (1967). 



\bibitem{Je} F. Jean.
\newblock Control of nonholonomic systems: from sub-Riemannian  geometry to motion planning.
\newblock Monograph. Springer (2014).


\bibitem{Jer} D. Jerison.
\newblock The Dirichlet problem for the Kohn Laplacian on the Heisenberg group.
\newblock J. of Functional Analysis 43 (1981), Part I, 97--141, Part II, 224--257.

	
\bibitem{JL} D. Jerison and J.M. Lee.
\newblock Extremal for the Sobolev inequality on the Heisenberg group and the CR Yamabe problem.
\newblock Journal of the AMS, Vol. 1, Number 1, January 1988.

\bibitem{LaWe} A. Laptev and  T. Weidl.
\newblock Sharp Lieb-Thirring inequalities in high dimensions.
\newblock Acta Math. 184 (2000), no. 1, 87--111.

\bibitem{Le} C. L\'ena.
\newblock Pleijel's nodal domain theorem for Neumann and Robin eigenfunctions. 
\newblock Annales de l'Institut Fourier 69 (1), 283--301 (2019). 

\bibitem{LeMa} G. P. Leonardi and  S. Masnou. 
\newblock On the isoperimetric problem in the Heisenberg group $\mathbb H^n$.
\newblock Ann. Mat. Pura Appl. (4) 184 (2005), no. 4, 533--553.

\bibitem{LeRi} G. P. Leonardi, S. Rigot.
\newblock Isoperimetric sets on Carnot groups.
\newblock Houston J. Math. 29 (2003), no. 3, 609--637.

\bibitem{Leo} G. Leoni.
\newblock A first course in Sobolev spaces. Second edition.
\newblock Grad. Stud. Math., 181, American Mathematical Society, Providence, RI, (2017).



\bibitem{LMP} M. Levitin, D. Mangoubi, and I. Polterovich.
\newblock Topics in spectral Geometry.
\newblock  AMS Graduate Studies in Mathematics series, volume 237 (2023).

\bibitem{LiLo} E. H. Lieb, M. Loss.
\newblock Analysis
\newblock Grad. Stud. Math., 14, American Mathematical Society, Providence, RI, 2001, xxii+346 pp.

\bibitem{LiTh} E. H. Lieb, W. E. Thirring.
\newblock Inequalities for the moments of the eigenvalues of the Schr\"odinger Hamiltonian and their relation to Sobolev inequalities. 
\newblock Stud. math. Phys., Essays Honor Valentine Bargmann, 269--303 (1976).

\bibitem{Ma} F. Maggi.
\newblock Sets of finite perimeter and geometric variational problems
\newblock Cambridge Stud. Adv. Math., 135, Cambridge University Press, Cambridge, 2012, xx+454 pp.

\bibitem{MOS}  W. Magnus, F. Oberhettinger, and R.P. Soni.
\newblock Formulas and Theorems for the special functions of mathematical Physics.
\newblock Die Grundlehre der mathematischen Wissenschaften. Band 52. Third edition. Springer (1966).

\bibitem{MS1} A. Menikoff and J. Sj\"ostrand. 
\newblock On the eigenvalues of a class of hypoelliptic operators.
\newblock Math. Ann. 235 (1978), 55--85.

\bibitem{MS2} A. Menikoff and J. Sj\"ostrand.  
\newblock On the eigenvalues of a class of hypoelliptic operators II.
\newblock Global analysis (Proc. Biennial Sem. Canad. Math. Congr., Univ. Calgary, Calgary,
Alta., 1978),  201--247, Lecture Notes in Math., 755, Springer, Berlin, 1979.


\bibitem{Me} G. M\'etivier.
\newblock Fonction spectrale et valeurs propres d'une classe
d'op\'erateurs non elliptiques.
\newblock Comm. in PDE 1 (1976), 467--519.

\bibitem{Me2} G. M\'etivier.
\newblock Hypoellipticit\'e analytique sur des groupes nilpotents de rang $2$. 
\newblock  Duke Math. Journal, Vol . 47, No 1   (1980), 195--221.



\bibitem {Miln}
John~W. Milnor.
\newblock  Topology from the differentiable viewpoint.
\newblock Princeton University Press, Princeton, NJ, (1997).
\newblock Revised reprint of the 1965 original.

\bibitem{Mo} A. Mohamed.
\newblock 
\'Etude spectrale d'op\'erateurs hypoelliptiques \'a caract\'eristiques multiples I.
\newblock  
Annales de l'Institut Fourier, Tome 32 (1982) no. 3,  39--90.

\bibitem{Mon} R. Monti.
\newblock Heisenberg isoperimetric problem. The axial case.
\newblock Adv. Calc. Var. 1 (2008), no. 1, 93--121.

\bibitem{MM} R. Monti and D. Morbidelli.
\newblock Non-tangentially accessible domains for vector fields.
\newblock Indiana University Mathematics Journal, Vol. 54, No. 2 (2005).

\bibitem{MoRi} R. Monti and  M. Rickly.
\newblock Convex isoperimetric sets in the Heisenberg group.
\newblock Ann. Sc. Norm. Super. Pisa Cl. Sci. (5) 8 (2009), no. 2, 391--415.

\bibitem{MP} E. M\"uller-Pfeiffer. 
\newblock On the number of nodal domains for eigenfunctions of elliptic differential operators.
\newblock  J. London Math. Soc. (2) \textbf{31} (1985), no. 1, 91--100.

\bibitem{Na} B.D.S. Nagy.
\newblock \"Uber Integralungleichungen zwischen einer Funktion und ihrer Ableitung.
\newblock Acta Sci. Math. 10, 64--74 (1941).


\bibitem{Pa} P. Pansu. 
\newblock An isoperimetric inequality for the Heisenberg group.\newblock   Comptes Rendus Acad. Sc. Math. 295.2 (1982), 127--130.

\bibitem{Pa2} P. Pansu. 
\newblock An isoperimetric inequality on the Heisenberg group. 
\newblock Conference on differential geometry on homogeneous spaces (Torino, 1983).
\newblock  Rend. Sem. Mat. Univ. Politec. Torino 1983, Special Issue, 159--174 (1984).

\bibitem{Pl} A.~Pleijel. 
\newblock Remarks on Courant's nodal theorem.
\newblock  Comm. Pure. Appl. Math., 9: 543--550 (1956). 

\bibitem{Po} I. Polterovich.
\newblock Pleijel's nodal domain theorem for free membranes.
\newblock Proc. Amer. Math. Soc. 137 (2009), no. 3, 1021--1024.

\bibitem{Ri} M. Ritor\'e.
\newblock A proof by calibration of an isoperimetric inequality in the Heisenberg group $\mathbb H^n$.
\newblock Calc. Var. Partial Differential Equations 44 (2012), no. 1-2, 47--60.

\bibitem{RR} M. Ritor\'e, C. Rosales.
\newblock Rotationally invariant hypersurfaces with constant mean curvature in the Heisenberg group $\mathbb H^n$.
\newblock J. Geom. Anal. 16 (2006), no. 4, 703--720.

\bibitem{RR2} M. Ritor\'e, C. Rosales.
\newblock Area-stationary surfaces in the Heisenberg group $\mathbb H^1$.
\newblock Adv. Math. 219 (2008), no. 2, 633--671.

\bibitem{Rot} L.P. Rothschild.
\newblock A criterion for hypoellipticity of operators constructed of vector fields.
\newblock Comm. in PDE 4 (6) (1979), 248--315.

\bibitem{RS} L.P. Rothschild and E.M. Stein.
\newblock Hypoelliptic differential operators and nilpotent groups.
\newblock Acta Mathematica 137, 248--315. (1976)

\bibitem{Ta} G. Talenti.
\newblock The standard isoperimetric theorem.
\newblock North-Holland Publishing Co., Amsterdam, (1993), 73--123.


\bibitem{Va} N. Th. Varopoulos.
\newblock  Analysis on nilpotent groups.
\newblock Journal of Functional Analysis 66, 406--431 (1986).

\bibitem{Wa} K. Watanabe.
\newblock  Sur l'unicit\'e du prolongement des solutions des \'equations elliptiques d\'eg\'en\'er\'ees.
\newblock Tohoku Math. Journ. 34 (1982), 239--249.

\end{thebibliography}

\end{document}